\documentclass[letter,10pt]{amsart}
\usepackage{amssymb,amsmath,amsfonts}
\usepackage{mathrsfs}
\usepackage[left=1 in, right=1 in,top=1 in, bottom=1 in]{geometry}
\usepackage{mathenv}

\newtheorem{theorem}{Theorem}[section]
\newtheorem{lemma}[theorem]{Lemma}

\newtheorem{remark}[theorem]{Remark}
\newtheorem{definition}[theorem]{Definition}

\makeatletter
\renewcommand \theequation {%
\ifnum \c@section>\z@ \@arabic\c@section.%
\fi\@arabic\c@equation} \@addtoreset{equation}{section}
\makeatother

\providecommand{\abs}[1]{\left\vert#1\right\vert}
\providecommand{\nm}[1]{\left\Vert#1\right\Vert}

\providecommand{\br}[1]{\left\langle #1 \right\rangle}
\providecommand{\tm}[1]{\left\Vert#1\right\Vert_{L^2}}
\providecommand{\im}[1]{\left\Vert#1\right\Vert_{L^{\infty}}}

\providecommand{\lnm}[2]{\left\Vert#1\right\Vert_{L^{\infty}_{#2}}}
\providecommand{\tnm}[1]{\left\Vert#1\right\Vert_{L^2}}
\providecommand{\lnnm}[2]{{\left\Vert\left\vert#1\right\vert\right\Vert}_{L^{\infty}L^{\infty}_{#2}}}
\providecommand{\tnnm}[1]{{\left\Vert\left\vert#1\right\vert\right\Vert}_{L^2L^2}}
\providecommand{\ltnm}[2]{{\left\Vert#1\right\Vert}_{L^{\infty}L^2_{#2}}}

\providecommand{\lnmv}[1]{\left\Vert#1\right\Vert_{L^{\infty}_{\vth,\varrho}}}
\providecommand{\lnnmv}[1]{{\left\Vert\left\vert#1\right\vert\right\Vert}_{L^{\infty}L^{\infty}_{\vth,\varrho}}}

\providecommand{\um}[1]{\left\Vert#1\right\Vert_{L^2_{\nu}}}
\providecommand{\tss}[2]{\left\vert#1\right\vert_{L^2_{#2}}}
\providecommand{\iss}[2]{\left\vert#1\right\vert_{L^{\infty}_{#2}}}
\providecommand{\ts}[1]{\left\vert#1\right\vert_{L^2}}

\def\ud{\mathrm{d}}

\def\p{\partial}
\def\ls{\lesssim}

\def\half{\frac{1}{2}}
\def\rt{\rightarrow}
\def\r{\mathbb{R}}
\def\no{\nonumber}
\def\ue{\mathrm{e}}

\def\ds{\displaystyle}

\def\pp{\mathcal{P}}
\def\e{\epsilon}
\def\s{\mathbb{S}}
\def\vx{\vec x}
\def\vw{\vec v}
\def\vv{\vec v}
\def\nx{\nabla_{x}}

\def\ll{\mathcal{L}}
\def\k{\kappa}
\def\q{Q}
\def\qb{\mathscr{Q}}
\def\d{\delta}
\def\vn{\vec n}

\def\id{{\bf{1}}}
\def\rr{\mathscr{R}}

\def\va{v_{\eta}}
\def\vb{v_{\phi}}

\def\f{F}
\def\fb{\mathscr{F}}
\def\fs{\mathfrak{F}}

\def\pe{\mathcal{P}^{\e}}
\def\vu{\vec u}
\def\vuu{\vec{\mathfrak{u}}}

\def\v{v}
\def\vww{\vec{\mathfrak{v}}}
\def\vvv{\vec{\mathfrak{v}}}
\def\u{u}
\def\uh{\vec u}
\def\th{\theta}
\def\rh{\rho}
\def\vo{\vec\omega}
\def\sn{\nu^{\frac{1}{2}}}

\def\m{\mu}
\def\mb{\m_{\bb}^{\e}}
\def\bb{b}

\def\vth{\vartheta}

\def\bv{\br{\vw}^{\vth}\ue^{\varrho\abs{\vw}^2}}
\def\bvv{\br{\vvv}^{\vth}\ue^{\varrho\abs{\vvv}^2}}
\def\dx{\Delta_x}
\def\bvvp{\br{\vvv'}^{\vth}\ue^{\varrho\abs{\vvv'}^2}}

\def\g{g}
\def\nk{\mathcal{N}}

\def\t{\mathcal{T}}
\def\gg{\mathcal{G}}

\def\v{\gg}
\def\w{\mathscr{W}}
\def\d{\delta}
\def\pp{\mathcal{P}}
\def\vn{\vec\nu}
\def\xc{X_{cl}}
\def\vc{V_{cl}}

\def\t{\mathcal{T}}
\def\k{\mathcal{K}}
\def\a{\mathscr{A}}
\def\b{\mathscr{B}}
\def\c{\mathscr{C}}
\def\rk{R_{\kappa}}

\def\id{{\bf{1}}}

\def\pk{{\mathbb{P}}}
\def\ik{{\mathbb{I}}}
\def\vbb{\vec b}
\def\n{n}
\def\nn{\mathcal{V}}
\def\vh{w_{\xi}}
\def\tvh{\tilde{w}_{\xi}}
\def\vs{v}

\def\lll{\mathscr{L}}

\def\ua{\mathfrak{u}_{\eta}}
\def\ub{\mathfrak{u}_{\phi}}

\begin{document}

\title{Boundary Layer of Boltzmann Equation in 2D Convex Domains}

\author[Lei Wu]{Lei Wu}
\address[Lei Wu]{
   \newline\indent Department of Mathematical Sciences, Carnegie Mellon University
\newline\indent Pittsburgh, PA 15213, USA}
\email{zjkwulei1987@gmail.com}


\begin{abstract}
Consider the stationary Boltzmann equation in 2D convex domains with diffusive boundary condition. In this paper, we establish the hydrodynamic limits while the boundary layers are present, and derive the steady Navier-Stokes-Fourier system with non-slip boundary conditions. Our contribution focuses on novel weighted $W^{1,\infty}$ estimates for the Milne problem with geometric correction. Also, we develop stronger remainder estimates based on an $L^{2m}-L^{\infty}$ framework.
\end{abstract}

\maketitle

\tableofcontents

\newpage


\pagestyle{myheadings} \thispagestyle{plain} \markboth{LEI WU}{BOUNDARY LAYER OF STATIONARY BOLTZMANN EQUATION}

\section{Introduction}

\subsection{Problem Presentation}

We consider the stationary Boltzmann equation in a two-dimensional smooth convex domain $\Omega\ni\vx=(x_1,x_2)$
with velocity $\vw=(v_1,v_2)\in\r^2$. The density function $\fs^{\e}(\vx,\vw)$ satisfies
\begin{eqnarray}\label{large system}
\left\{
\begin{array}{rcl}
\e\vw\cdot\nx \fs^{\e}&=&Q[\fs^{\e},\fs^{\e}]\ \ \text{in}\ \
\Omega\times\r^2,\\\rule{0ex}{1.5em} \fs^{\e}(\vx_0,\vw)&=&P^{\e}
[\fs^{\e}](\vx_0,\vw) \ \ \text{for}\ \ \vx_0\in\p\Omega\ \ \text{and}\ \ \vw\cdot\vn(\vx_0)<0,
\end{array}
\right.
\end{eqnarray}
where $\vn(\vx_0)$ is the unit outward normal vector at $\vx_0$, the
Knudsen number $\e$ satisfies $0<\e<<1$, the diffusive boundary
\begin{eqnarray}
P^{\e}
[\fs^{\e}](\vx_0,\vw)=\mb(\vx_0,\vw)\displaystyle\int_{\vuu\cdot\vn(\vx_0)>0}
\fs^{\e}(\vx_0,\vuu)\abs{\vuu\cdot\vn(\vx_0)}\ud{\vuu}.
\end{eqnarray}
The boundary Maxwellian
\begin{eqnarray}
\mb(\vx_0,\vw)=\frac{\rh^{\e}_{\bb}(\vx_0)}{\th^{\e}_{\bb}(\vx_0)\sqrt{2\pi}}
\exp\left(-\frac{\abs{\vw-\uh^{\e}_{\bb}(\vx_0)}^2}{2\th^{\e}_{\bb}(\vx_0)}\right),
\end{eqnarray}
is a perturbation of the standard Maxwellian
\begin{eqnarray}
\m_{\bb}(\vw)=\frac{1}{\sqrt{2\pi}}\exp\left(-\frac{\abs{\vw}^2}{2}\right).
\end{eqnarray}
It is normalized to satisfy
\begin{eqnarray}
\int_{\vw\cdot\vn(\vx_0)>0}\mb(\vx_0,\vw)\abs{\vw\cdot\vn(\vx_0)}\ud{\vw}=\int_{\vw\cdot\vn(\vx_0)>0}\m_{\bb}(\vw)\abs{\vw\cdot\vn(\vx_0)}\ud{\vw}=1.
\end{eqnarray}
For simplicity, we just denote $\m=\m_{\bb}$.
We further assume that $\rh^{\e}_{\bb}$, $\vu^{\e}_{\bb}$ and $\th^{\e}_{\bb}$ can be expanded into a power series with respect to $\e$,
\begin{eqnarray}
\rh^{\e}_{\bb}(\vx_0)&=&1+\sum_{k=1}^{\infty}\e^k\rh_{\bb,k}(\vx_0),\\
\vu^{\e}_{\bb}(\vx_0)&=&0+\sum_{k=1}^{\infty}\e^k\vu_{\bb,k}(\vx_0),\\
\th^{\e}_{\bb}(\vx_0)&=&1+\sum_{k=1}^{\infty}\e^k\th_{\bb,k}(\vx_0).
\end{eqnarray}
Hence, we may also expand the boundary Maxwellian $\mb$ into power series with respect to $\e$,
\begin{eqnarray}\label{expansion assumption}
\mb(\vx_0,\vw)=\m(\vw)+\m^{\frac{1}{2}}(\vw)\left(\sum_{k=1}^{\infty}\e^k\m_{k}(\vx_0,\vw)\right).
\end{eqnarray}
In particular, we have
\begin{eqnarray}
\m_1(\vx_0,\vw)&=&\m^{\frac{1}{2}}(\vw)\bigg(\rh_{\bb,1}(\vx_0)+\vu_{\bb,1}(\vx_0)\cdot\vw+\th_{\bb,1}(\vx_0)\frac{\abs{\vw}^2-2}{2}\bigg).
\end{eqnarray}
It is easy to check that
\begin{eqnarray}\label{smallness assumption}
\abs{\bv\frac{\mb-\m}{\m^{\frac{1}{2}}}}\leq C_0(\varrho,\vth)\e,
\end{eqnarray}
for any $0\leq\varrho<\dfrac{1}{4}$ and integer $\vth\geq3$. We assume that $C_0>0$ is sufficiently small. Based on the expansion,
we naturally have
\begin{eqnarray}\label{boundary compatibility}
\int_{\vw\cdot\vn(\vx_0)>0}\m_k(\vx_0,\vw)\m^{\frac{1}{2}}(\vw)\abs{\vw\cdot\vn(\vx_0)}\ud{\vw}&=&0\ \ \text{for}\ \ k\geq1.
\end{eqnarray}
Here we have the nonlinear collision term
\begin{eqnarray}
Q[F,G]&=&\int_{\r^2}\int_{\s^1}q(\vo,\abs{\vuu-\vw})\bigg(F(\vuu_{\ast})G(\vw_{\ast})-F(\vuu)G(\vw)\bigg)\ud{\vo}\ud{\vuu},
\end{eqnarray}
with
\begin{eqnarray}
\vuu_{\ast}=\vuu+\vo\bigg((\vw-\vuu)\cdot\vo\bigg),\qquad \vw_{\ast}=\vw-\vo\bigg((\vw-\vuu)\cdot\vo\bigg),
\end{eqnarray}
and the hard-sphere collision kernel
\begin{eqnarray}
q(\vo,\abs{\vuu-\vw})=q_0\vo\cdot(\vw-\vuu),
\end{eqnarray}
for a positive constant $q_0$.
We intend to study the behavior of $\fs^{\e}$ as $\e\rt0$.

\subsection{Linearization}

The solution $\fs^{\e}$ can be expressed as a perturbation of the standard Maxwellian
\begin{eqnarray}
\fs^{\e}(\vx,\vw)&=&M_0\m(\vw)+\m^{\frac{1}{2}}(\vw)f^{\e}(\vx,\vw),
\end{eqnarray}
for some constant $M_0>0$ with the normalization condition
\begin{eqnarray}
\int_{\Omega}\int_{\r^2}f^{\e}(\vx,\vw)\m^{\frac{1}{2}}(\vw)\ud{\vw}\ud{\vx}=0.
\end{eqnarray}
Then $f^{\e}$ satisfies the equation
\begin{eqnarray}\label{small system_}
\left\{
\begin{array}{rcl}
\e\vw\cdot\nx
f^{\e}+\ll[f^{\e}]&=&\Gamma[f^{\e},f^{\e}],\\\rule{0ex}{1.5em}
f^{\e}(\vx_0,\vw)&=&\pe[f^{\e}](\vx_0,\vw) \ \ \text{for}\ \ \vx_0\in\p\Omega\ \ \text{and}\ \ \vw\cdot\vn(\vx_0)<0,
\end{array}
\right.
\end{eqnarray}
where
\begin{eqnarray}
\ll[f^{\e}]&=&-2\m^{-\frac{1}{2}}Q\Big[\m,\m^{\frac{1}{2}}f^{\e}\Big],\\
\Gamma[f^{\e},f^{\e}]&=&\m^{-\frac{1}{2}}Q\Big[\m^{\frac{1}{2}}f^{\e},\m^{\frac{1}{2}}f^{\e}\Big],
\end{eqnarray}
and
\begin{eqnarray}
\\
\pe[f^{\e}](\vx_0,\vw)=\mb(\vx_0,\vv)\m^{-\frac{1}{2}}(\vv)
\displaystyle\int_{\vuu\cdot\vn(\vx_0)>0}\m^{\frac{1}{2}}(\vuu)
f^{\e}(\vx_0,\vuu)\abs{\vuu\cdot\vn(\vx_0)}\ud{\vuu}+\m^{-\frac{1}{2}}(\vw)\bigg(\mb(\vx_0,\vv)-\m(\vv)\bigg).\no
\end{eqnarray}
Hence, in order to study $\fs^{\e}$, it suffices to consider $f^{\e}$.

\subsection{Background and Methods}

\subsubsection{Asymptotic Analysis}

Hydrodynamic limits are central to connecting the kinetic theory and fluid mechanics. Since early 20th century, this type of problems have been extensively studied in many different settings: stationary or evolutionary, linear or nonlinear, strong solution or weak solution, etc.

The early result dates back to 1912 by Hilbert himself, using the so-called Hilbert's expansion, i.e. an expansion of the distribution function $\fs^{\e}$ as a power series of the Knudsen number $\e$. Since then, a lot of works on Boltzmann equation in $\r^n$ or $\mathbb{T}^n$ have been presented, including \cite{Golse.Saint-Raymond2004}, \cite{Masi.Esposito.Lebowitz1989}, \cite{Bardos.Golse.Levermore1991}, \cite{Bardos.Golse.Levermore1993}, \cite{Bardos.Golse.Levermore1998}, \cite{Bardos.Golse.Levermore2000}, for either smooth solutions or renormalized solutions.

The general theory of initial-boundary-value problems was first developed in 1963 by Grad \cite{Grad1963}, and then extended by Darrozes \cite{Darrozes1969}, Sone and Aoki  \cite{Sone1969}, \cite{Sone1971}, \cite{Sone1991}, \cite{Sone.Aoki1987}, for both the evolutionary and stationary equations. In the classical books \cite{Sone2002} and \cite{Sone2007}, Sone provided a comprehensive summary of previous results and gave a complete analysis of such approaches.

For stationary Boltzmann equation where the state of gas is close to a uniform state at rest, the expansion of the perturbation $f^{\e}$ consists of two parts: the interior solution $\f$, which is based on a hierarchy of linearized Boltzmann equations and satisfies a steady Navier-Stokes-Fourier system, and the boundary layer $\fb$, which is based on a half-space kinetic equation and decays rapidly when it is away from the boundary.

The justification of hydrodynamic limits usually involves two steps:
\begin{enumerate}
\item
Expanding $\f=\ds\sum_{k=1}^{\infty}\e^k\f_k$ and $\fb=\ds\sum_{k=1}^{\infty}\e^k\fb_k$ as power series of $\e$ and proving the coefficients $\f_k$ and $\fb_k$ are well-defined. Traditionally, the estimates of interior solutions $\f_k$ are relatively straightforward. On the other hand, boundary layers $\fb_k$ satisfy one-dimensional half-space problems which lose some key structures of the original equations. The well-posedness of boundary layer equations are sometimes extremely difficult and it is possible that they are actually ill-posed (e.g. certain type of Prandtl layers).
\item
Proving that $R=f^{\e}-\e\f_1-\e\fb_1=o(\e)$ as $\e\rt0$. Ideally, this should be done just by expanding to the leading-order level $\f_1$ and $\fb_1$. However, in singular perturbation problems, the estimates of the remainder $R$ usually involves negative powers of $\e$, which requires expansion to higher order terms $\f_N$ and $\fb_N$ for $N\geq2$ such that we have sufficient power of $\e$. In other words, we define $R=f^{\e}-\ds\sum_{k=1}^{N}\e^k\f_k-\ds\sum_{k=1}^{N}\e^k\fb_k$ for $N\geq2$ instead of $R=f^{\e}-\e\f_1-\e\fb_1$ to get better estimate of $R$.
\end{enumerate}
Above formulation is for the convergence in the $L^{\infty}$ sense. If instead we consider $L^p$ convergence for $1\leq p<\infty$, then the boundary layer $\fb_1$ is of order $\e^{\frac{1}{p}}$ due to rescaling, which is negligible compared with $\f_1$. \cite{Esposito.Guo.Kim.Marra2015} justifies the $L^p$ convergence under the same formulation as ours without taking boundary layer expansion into consideration. On the other hand, the effect of boundary layers constitutes the major upshot of our paper.

\subsubsection{Classical Approach}

The classical construction of boundary layers requires the analysis of the flat Milne problem. In detail, let $\eta$ denote the rescaled normal variable with respect to the boundary, $\theta$ the tangential variable, and $\vvv=(\va,\vb)$ the normal and tangential velocity. The boundary layer $\fb_1$ satisfies
\begin{eqnarray}
\va\dfrac{\p\fb_1}{\p\eta}+\ll[\fb_1]&=&0,
\end{eqnarray}
where $\ll$ is the linearized Boltzmann operator.

Although a rigorous proof of such expansions has not been presented, it is widely believed that the motivation of this approach is natural and the difficulties are purely technical. Besides the fact that this idea is an intuitive application of the Hilbert's expansion, it is strongly supported by \cite{Bardos.Caflisch.Nicolaenko1986} which justifies the well-posedness and decay of the above flat Milne problem.

This idea is easily adapted to other kinetic models. As a linear prototype of Boltzmann equation, the case of neutron transport equation was carefully investigated. In particular,
the hydrodynamic limit was proved in the remarkable paper \cite{Bensoussan.Lions.Papanicolaou1979} by Bensoussan, Lions and Papanicolaou. This is widely regarded as the foundation of rigorous analysis of boundary layers in kinetic equations.

Unfortunately, in \cite{AA003}, we demonstrated that both the proof and results of this formulation in \cite{Bensoussan.Lions.Papanicolaou1979} are invalid due to a lack of regularity in estimating $\dfrac{\p\fb_1}{\p\theta}$.  Similarly, counterexamples were proposed in \cite{AA004} that this idea is also invalid in the Boltzmann equation. Basically, this pulls the whole study back to the starting point and any later results based on this type of boundary layers should be reexamined.

In detail, in order to show the hydrodynamic limits, we need $\dfrac{\p\fb_1}{\p\theta}\in L^{\infty}$ since it is part of the remainder. However, though $\fb_1\in L^{\infty}$ as is shown in \cite{Bardos.Caflisch.Nicolaenko1986}, we do not necessarily have $\dfrac{\p\fb_1}{\p\eta}\in L^{\infty}$. Furthermore, the singularity $\dfrac{\p\fb_1}{\p\eta}\notin L^{\infty}$ will be transferred to $\dfrac{\p\fb_1}{\p\theta}\notin L^{\infty}$. This singularity was rigorously shown in \cite{AA003} through a careful construction of the boundary data, i.e. the chain of estimates
\begin{eqnarray}
R=o(\e)\ \Leftarrow\ \fb_2\in L^{\infty}\ \Leftarrow\ \dfrac{\p\fb_1}{\p\theta}\in L^{\infty}\ \Leftarrow\ \dfrac{\p \fb_1}{\p\eta}\in L^{\infty},
\end{eqnarray}
is broken since the rightmost estimate is wrong.

\subsubsection{Geometric Correction}

While the classical method breaks down, a new approach with geometric correction to the boundary layer construction has been developed to ensure regularity in the cases of disk and annulus in \cite{AA003}, \cite{AA006}, \cite{AA005} and \cite{AA004}. The new boundary layer $\fb_1$ satisfies the $\e$-Milne problem with geometric correction,
\begin{eqnarray}
\va\dfrac{\p\fb_1}{\p\eta}-\dfrac{\e}{\rk-\e\eta}\bigg(\vb^2\dfrac{\p\fb_1}{\p\va}-\va\vb\dfrac{\p\fb_1}{\p\vb}\bigg)+\ll[\fb_1]&=&0,
\end{eqnarray}
where $\rk$ is the curvature on the boundary curve. We proved that the solution recovers the well-posedness and exponential decay as in flat Milne problem, and the regularity in $\theta$ is indeed improved, i.e. $\dfrac{\p\fb_1}{\p\theta}\in L^{\infty}$.

However, this new method fails to treat more general domains. Roughly speaking, we have two contradictory goals to achieve:
\begin{enumerate}
\item
To prove hydrodynamic limits, the remainder estimates require higher-order regularity estimate of the boundary layer.
\item
The geometric correction $\dfrac{\e}{\rk-\e\eta}\bigg(\vb^2\dfrac{\p\fb_1}{\p\va}-\va\vb\dfrac{\p\fb_1}{\p\vb}\bigg)$ in the boundary layer equation is related to the curvature of the boundary curve, which prevents higher-order regularity estimates.
\end{enumerate}
In other words, the improvement of regularity is still not enough to close the proof. To be more specific, the discussion of domains is as follows:
\begin{itemize}
\item
In the absence of the geometric correction $\dfrac{\e}{\rk-\e\eta}\bigg(\vb^2\dfrac{\p\fb_1}{\p\va}-\va\vb\dfrac{\p\fb_1}{\p\vb}\bigg)$, which is the flat Milne problem as in \cite{Sone2002} and \cite{Sone2007}, the key tangential derivative $\dfrac{\p\fb_1}{\p\theta}$ is not bounded. Therefore, the expansion breaks down.
\item
In the domain of disk or annulus, when $R_{\kappa}$ is constant, as in \cite{AA004}, $\dfrac{\p\fb_1}{\p\theta}$ is bounded, since the tangential derivative $\dfrac{\p}{\p\theta}$ commutes with the equation, and thus we do not need the estimate of the singular term $\dfrac{\p\fb_1}{\p\eta}$.
\item
For general smooth convex domains, when $R_{\kappa}$ is a function of $\theta$, $\dfrac{\p\fb_1}{\p\theta}$ relates to the normal derivative $\dfrac{\p\fb_1}{\p\eta}$, which has been shown possibly unbounded in \cite{AA004}. Therefore, we get stuck again at the regularity estimates.
\end{itemize}

\subsubsection{Diffusive Boundary}

In this paper, we will push the above argument from both sides (remainder estimates and regularity estimates) and prove the hydrodynamic limits for the nonlinear Boltzmann equation in 2D smooth convex domains. Our contribution consists of the following:
\begin{itemize}

\item
{\textbf{Remainder Estimates:}} \\
We first prove an almost optimal remainder estimates and reduce the regularity requirement of the expansion. In the remainder equation for $R(\vx,\vv)=f^{\e}-\f-\fb$,
\begin{eqnarray}
\e\vw\cdot\nx
R+\ll[R]&=&S,
\end{eqnarray}
the estimate in \cite{AA004} is
\begin{eqnarray}
\im{R}\ls \frac{1}{\e^3}\tm{S}+\text{higher order terms}.
\end{eqnarray}
We intend to show that $\im{R}=o(\e)$ as $\e\rt0$. Since $S$ contains the term related to $\dfrac{\p\fb}{\p\theta}$, the coefficients $\e^{-3}$ is too singularity. The key observation here is that due to the rescaling in the normal direction, the smaller $p\geq1$ is, the better estimate $\nm{\dfrac{\p\fb}{\p\theta}}_{L^p}$ will be. Therefore, we successfully prove a stronger estimate for $m\geq2$,
\begin{eqnarray}
\im{R}\ls \frac{1}{\e^{2+\frac{1}{m}}}\nm{S}_{L^{\frac{2m}{2m-1}}}+\text{higher order terms}.
\end{eqnarray}
This is achieved by an innovative $L^{2m}$-$L^{\infty}$ framework. The main idea is to introduce special test functions in the weak formulation to treat kernel and non-kernel parts of $\ll$ separately, and further to bootstrap to improve the $L^{\infty}$ estimate by a modified double Duhamel's principle. The proof relies on a delicate analysis using interpolation and Young's inequality.\\

\item
{\textbf{Regularity of Boundary Layer:}}\\
Consider the boundary layer expansion
\begin{eqnarray}
\fb(\eta,\theta,\vvv)\sim \e\fb_1(\eta,\theta,\vvv)+\e^2\fb_2(\eta,\theta,\vvv).
\end{eqnarray}
The diffusive boundary condition leads to an important simplification that $\fb_1=0$. Thus the next-order boundary layer $\fb_2$ must formally satisfy
\begin{eqnarray}\label{intro 11}
\va\dfrac{\p\fb_2}{\p\eta}-\dfrac{\e}{\rk-\e\eta}\bigg(\vb^2\dfrac{\p\fb_2}{\p\va}-\va\vb\dfrac{\p\fb_2}{\p\vb}\bigg)+\ll[\fb_2]&=&0.
\end{eqnarray}
The remainder estimate requires the estimate of $\dfrac{\p \fb_2}{\p\theta}$, whose boundedness had remained open..

The key observation here is that the estimate of $\dfrac{\p \fb_2}{\p\theta}$ relies on $\va\dfrac{\p\fb_2}{\p\eta}$, not $\dfrac{\p\fb_2}{\p\eta}$ itself. This extra $\va$ saves us and avoids singularity. Still, we cannot naively take $\eta$ derivatives on both sides of (\ref{intro 11}) since the geometric correction forbids further estimates. Our main idea is to track the solution $\fb_2$ along the characteristic curves in the presence of the non-local operator $\ll$ and consider the intertwined $\dfrac{\p\fb_2}{\p\eta}$, $\dfrac{\p\fb_2}{\p\va}$, and $\dfrac{\p\fb_2}{\p\vb}$ simultaneously.

Our proof is intricate and relies on the weighted $L^{\infty}$ estimates for the normal derivative, which is inspired by \cite{AA007} and \cite{AA009}. The convexity and invariant kinetic distance
\begin{eqnarray}
\zeta(\eta,\theta,\va,\vb)=\left(\left(\va^2+\vb^2\right)-\left(\frac{\rk(\theta)-\e\eta}{\rk(\theta)}\right)^2\vb^2\right)^{\frac{1}{2}},
\end{eqnarray}
plays the crucial role.\\

\item
{\textbf{Expansion and Nonlinearity:}}\\
The matching procedure between the interior solution and boundary layer is actually a very tricky step, which is often ignored in the related literature. \cite{Sone2002} and \cite{Sone2007} offers a details justification of this procedure, but it has not taken the new boundary layer construction into consideration.

Here, we provide a clear description on how to delicately determine the macroscopic variables and how to handle the nonlinearity in the remainder estimates. In particular, we enforce the well-known Boussinesq relation at the leading order through the conservation of mass and an intricate designing of the boundary layer expansion. Also, we use a special $L^{2m}-L^{\infty}$ method to absorb the nonlinear terms contribution.

\end{itemize}

\subsection{Main Theorem}

\begin{theorem}
For given $M_0>0$ and $\mb>0$ satisfying (\ref{expansion assumption}) and (\ref{smallness assumption}) with $0<\e<<1$, there exists a unique positive
solution $\fs^{\e}=M_0\m+\m^{\frac{1}{2}}f^{\e}$ to the stationary Boltzmann equation (\ref{large system}), and $f^{\e}$ fulfils that for integer $\vth\geq3$ and $0\leq\varrho<\dfrac{1}{4}$,
\begin{eqnarray}
\im{\bv\Big(f^{\e}-\e\f\Big)}\leq C(\d)\e^{2-\d},
\end{eqnarray}
for any $0<\d<<1$, where
\begin{eqnarray}
\f&=&\m^{\frac{1}{2}}\left(\rh+\vu\cdot\vw+\th\frac{\abs{\vw}^2-2}{2}\right),
\end{eqnarray}
satisfies the steady Navier-Stokes-Fourier system
\begin{eqnarray}\label{interior 1}
\left\{
\begin{array}{rcl}
\nx(\rh +\th )&=&0,\\\rule{0ex}{1.0em}
\uh\cdot\nx\uh -\gamma_1\dx\uh +\nx P_2 &=&0,\\\rule{0ex}{1.0em}
\nx\cdot\uh &=&0,\\\rule{0ex}{1.0em}
\uh \cdot\nx\th -\gamma_2\dx\th &=&0,\\\rule{0ex}{1.0em}
\rh (\vx_0)&=&\rh_{\bb,1}(\vx_0)+M(\vx_0),\\
\uh (\vx_0)&=&\vu_{\bb,1}(\vx_0),\\
\th (\vx_0)&=&\th_{\bb,1}(\vx_0),\\
\end{array}
\right.
\end{eqnarray}
where $\gamma_1>0$ and $\gamma_2>0$ are some constants, $M(\vx_0)$ is a constant such that the Boussinesq relation
\begin{eqnarray}
\rh+\th=\text{constant},
\end{eqnarray}
and the normalization condition
\begin{eqnarray}
\int_{\Omega}\int_{\r^2}\f(\vx,\vw)\m^{\frac{1}{2}}(\vw)\ud{\vw}\ud{\vx}=0,
\end{eqnarray}
hold.
\end{theorem}

\begin{remark}
The case $\rh_{\bb,1}(\vx_0)=0$, $\vu_{\bb,1}(\vx_0)=0$ and $\th_{\bb,1}(\vx_0)\neq0$ is called the non-isothermal model, which represents a system that only has heat transfer through the boundary but has no work between the environment and the system. Based on above theorem, the hydrodynamic limit is a steady Navier-Stokes-Fourier system with non-slip boundary condition. This provides a rigorous derivation of this important fluid model.
\end{remark}

Throughout this paper, $C>0$ denotes a constant that only depends on
the parameter $\Omega$, but does not depend on the data. It is
referred as universal and can change from one inequality to another.
When we write $C(z)$, it means a certain positive constant depending
on the quantity $z$. We write $a\ls b$ to denote $a\leq Cb$.

This paper is organized as follows: in Section 2, we list some preliminary results on the linearized Boltzmann operator and the weak formulation; in Section 3, we present the asymptotic analysis of the equation (\ref{small system_}); in Section 4, we establish the $L^{\infty}$ well-posedness of
the linearized Boltzmann equation; in Section 5, we prove the well-posedness and decay of the $\e$-Milne problem with geometric correction; in Section 6, we study the weighted regularity of the $\e$-Milne problem with geometric correction; finally, in Section 7, we prove the main theorem.

\begin{remark}
The general structure of this paper is very similar to that of \cite{AA007} and \cite{AA004}. In particular, Section 4 and 5 seem to be an adaption of the corresponding theorems there. However, our results hold for nonlinear Boltzmann equation in a finite domain and they are highly non-trivial, so it is better to start from scratch.
\end{remark}

\newpage

\section{Preliminaries}

\subsection{Linearized Boltzmann Operator}

\cite[Chapter 3]{Glassey1996} provides the simplified linearized Boltzmann operator $\ll$ as
\begin{eqnarray}
\ll[f]&=&-2\m^{-\frac{1}{2}}Q[\m,\m^{\frac{1}{2}}f]=\nu(\vw)
f-K[f],
\end{eqnarray}
where
\begin{eqnarray}
\nu(\vw)&=&\int_{\r^2}\int_{\s^1}q(\vo,\abs{\vuu-\vw})\m(\vuu)\ud{\vo}\ud{\vuu},\\
K[f](\vw)&=&K_2[f](\vw)-K_1[f](\vw)=\int_{\r^2}k(\vuu,\vw)f(\vuu)\ud{\vuu},\no\\
\ \no\\
K_1[f](\vw)&=&\m^{\frac{1}{2}}(\vw)\int_{\r^2}\int_{\s^1}q(\vo,\abs{\vuu-\vw})\m^{\frac{1}{2}}(\vuu)f(\vuu)\ud{\vo}\ud{\vuu},\\
K_2[f](\vw)&=&\int_{\r^2}\int_{\s^1}q(\vo,\abs{\vuu-\vw})\m^{\frac{1}{2}}(\vuu)\bigg(\m^{\frac{1}{2}}(\vw_{\ast})f(\vuu_{\ast})
+\m^{\frac{1}{2}}(\vuu_{\ast})f(\vw_{\ast})\bigg)\ud{\vo}\ud{\vuu},
\end{eqnarray}
for some kernel $k(\vuu,\vw)$.

Let $\br{\cdot,\cdot}$ be the standard $L^2$ inner product in
$\Omega\times\r^2$. We define the $L^p$ and $L^{\infty}$ norms in
$\Omega\times\r^2$ as usual:
\begin{eqnarray}
\nm{f}_{L^p}&=&\bigg(\int_{\Omega}\int_{\r^2}\abs{f(\vx,\vv)}^p\ud{\vv}\ud{\vx}\bigg)^{\frac{1}{p}},\\
\nm{f}_{L^{\infty}}&=&\sup_{(\vx,\vv)\in\Omega\times\r^2}\abs{f(\vx,\vv)}.
\end{eqnarray}
Define the weighted $L^{2}$ norm as follows:
\begin{eqnarray}
\um{f}=\tm{\nu^{\frac{1}{2}}f}.
\end{eqnarray}
Define the weighted $L^{\infty}$ norm as follows:
\begin{eqnarray}
\lnmv{f}&=&\sup_{(\vx,\vv)\in\Omega\times\r^2}\bigg(\bv\abs{f(\vx,\vv)}\bigg),
\end{eqnarray}
Define $\ud{\gamma}=\abs{\vv\cdot\vn}\ud{\varpi}\ud{\vv}$ on the
boundary $\p\Omega\times\r^2$ for $\varpi$ as the curve measure. Define the $L^p$ and
$L^{\infty}$ norms on the boundary as follows:
\begin{eqnarray}
\abs{f}_{L^p}&=&\bigg(\iint_{\gamma}\abs{f(\vx,\vv)}^p\ud{\gamma}\bigg)^{1/p},\\
\abs{f}_{L^p_{\pm}}&=&\bigg(\iint_{\gamma_{\pm}}\abs{f(\vx,\vv)}^p\ud{\gamma}\bigg)^{1/p},\\
\abs{f}_{L^{\infty}}&=&\sup_{(\vx,\vv)\in\gamma}\abs{f(\vx,\vv)},\\
\abs{f}_{L^{\infty}_{\pm}}&=&\sup_{(\vx,\vv)\in\gamma_{\pm}}\abs{f(\vx,\vv)}.
\end{eqnarray}
Denote the Japanese bracket as
\begin{eqnarray}
\br{\vv}=\left(1+\abs{\vv}^2\right)^{\frac{1}{2}}
\end{eqnarray}
Define the kernel operator $\pk$ as
\begin{eqnarray}
\pk[f]=\m^{\frac{1}{2}}(\vv)\bigg(a_f(\vx)+\vv\cdot
\vbb_f(\vx)+\frac{\abs{\vv}^2-2}{2}c_f(\vx)\bigg),
\end{eqnarray}
where $\pk$ is in the null space of $\ll$, and the non-kernel operator $\ik-\pk$ as
\begin{eqnarray}
(\ik-\pk)[f]=f-\pk[f].
\end{eqnarray}
with
\begin{eqnarray}
\int_{\r^2}(\ik-\pk)[f]\left(\begin{array}{c}1\\\vv\\\abs{\vv}^2\end{array}\right)\ud{\vv}=0
\end{eqnarray}
\begin{lemma}\label{Milne property}
For the operator $\ll=\nu I-K$, we have the estimates
\begin{eqnarray}
&&\lnm{\frac{\p\nu}{\p\abs{\vv}}}{}\leq C,\\
&&\nu_0(1+\abs{\vv})\leq\nu(\vv)\leq\nu_1(1+\abs{\vv}),\\
&&\br{f,\ll[f]}=\br{(\ik-\pk)[f],\ll\Big[(\ik-\pk)[f]\Big]}\geq C\tnm{\nu^{\frac{1}{2}}(\ik-\pk)[f]}^2,\\
&&\tnm{\ll\Big[(\ik-\pk)[f]\Big]}^2\geq C\tnm{\nu^{\frac{1}{2}}(\ik-\pk)[f]}^2,\\
&&\tm{\pk[f]}\leq \tm{\nu\pk[f]}\leq C\tm{\pk[f]}.
\end{eqnarray}
for $\nu_0$, $\nu_1$ and $C$ positive constants.
\end{lemma}
\begin{proof}
See \cite[Chapter 3]{Glassey1996}.
\end{proof}

\begin{lemma}\label{wellposedness prelim lemma 8}
Let
$\vh(\vv)=w_{\xi,\beta,\varrho}(\vv)=\left(1+\xi^2\abs{\vv}^2\right)^{\frac{\beta}{2}}\ue^{\varrho\abs{\vv}^2}$,
for $\xi,\beta>0$ and $0\leq\varrho\leq \dfrac{1}{4}$. Then there exists $0\leq C_1(\varrho)<1$ and
$C_2(\varrho)>0$ such that for $0\leq \d\leq C_1(\varrho)$,
\begin{eqnarray}
\int_{\r^2}\ue^{\d\abs{\vuu-\vv}^2}k(\vuu,\vv)
\frac{\vh(\vv)}{\vh(\vuu)}\ud{\vuu}
\leq\frac{C_2(\varrho)}{1+\abs{\vv}},\\
\int_{\r^2}\ue^{\d\abs{\vuu-\vv}^2}\frac{1}{\abs{\vuu}}k(\vuu,\vv)
\frac{\vh(\vv)}{\vh(\vuu)}\ud{\vuu}
\leq C_2(\varrho),\\
\int_{\r^2}\ue^{\d\abs{\vuu-\vv}^2}\nabla_{v}k(\vuu,\vv)
\frac{\vh(\vv)}{\vh(\vuu)}\ud{\vuu}
\leq C_2(\varrho).
\end{eqnarray}
For $m\in\mathbb{N}$, we have
\begin{eqnarray}
\nm{K[f]}_{L^{2m}}&\leq&C\nm{f}_{L^{2m}}.
\end{eqnarray}
\end{lemma}
\begin{proof}
See \cite[Lemma 3]{Guo2010}.
\end{proof}

\begin{lemma}\label{prelim 1}
We have
\begin{eqnarray}
\lnmv{K[f]}&\leq& C\lnmv{\dfrac{f}{\nu}},\\
\lnmv{\nabla_{v}K[f]}&\leq& C\lnmv{f}.
\end{eqnarray}
\end{lemma}
\begin{proof}
Consider the fact that for $\vth=\beta$, we have
\begin{eqnarray}
C_1\bv\leq\vh\leq C_2\bv,
\end{eqnarray}
for some constant $C_1,C_2>0$. Then this is a natural corollary of Lemma \ref{wellposedness prelim lemma 8}. See \cite{Guo2010} and \cite{Guo.Kim.Tonon.Trescases2013}.
\end{proof}

\begin{lemma}\label{nonlinear}
The nonlinear term $\Gamma$ satisfies for $\beta>0$ and $0\leq\varrho\leq \dfrac{1}{4}$,
\begin{eqnarray}
\lnmv{\nu^{-1}\Gamma[f,f]}&\leq&C\lnmv{f}^2,\\
\int_{\Omega\times\r^2}\Gamma[f,g]h&\leq&C\um{h}\bigg(\nm{f}_{L^{2}}\lnmv{g}+\nm{g}_{L^{2}}\lnmv{f}\bigg),\\
\int_{\Omega\times\r^2}\Gamma[f,g]h&\leq&C\um{h}\bigg(\nm{g}_{L^{2}}\lnmv{f}+\im{f}\um{g}\bigg).
\end{eqnarray}
\end{lemma}
\begin{proof}
See \cite[Lemma 2.3]{Guo2002} and \cite[Chapter 3]{Glassey1996}.
\end{proof}

\subsection{Formulation and Estimates}

Based on the flow direction, we can divide the boundary
$\gamma=\{(\vx_0,\vw):\ \vx_0\in\p\Omega,\vw\in\r^2\}$ into the in-flow boundary
$\gamma_-$, the out-flow boundary $\gamma_+$, and the grazing set
$\gamma_0$ as
\begin{eqnarray}
\gamma_{-}&=&\{(\vx_0,\vw):\ \vx_0\in\p\Omega,\ \vw\cdot\vn(\vx_0)<0\},\\
\gamma_{+}&=&\{(\vx_0,\vw):\ \vx_0\in\p\Omega,\ \vw\cdot\vn(\vx_0)>0\},\\
\gamma_{0}&=&\{(\vx_0,\vw):\ \vx_0\in\p\Omega,\ \vw\cdot\vn(\vx_0)=0\}.
\end{eqnarray}
It is easy to see $\gamma=\gamma_+\cup\gamma_-\cup\gamma_0$. Also,
the boundary condition is only given on $\gamma_{-}$.
\begin{lemma}\label{wellposedness prelim lemma 1}
Define the near-grazing set of $\gamma_+$ or $\gamma_-$ as
\begin{eqnarray}
\gamma_{\pm}^{\d}=\left\{(\vx,\vv)\in\gamma_{\pm}:
\abs{\vn(\vx)\cdot\vv}\leq\d\ \text{or}\ \abs{\vv}\geq\frac{1}{\d}\ \text{or}\
\abs{\vv}\leq\d\right\}.
\end{eqnarray}
Then
\begin{eqnarray}
\abs{f\id_{\gamma_{\pm}\backslash\gamma_{\pm}^{\d}}}_{L^1}\leq
C(\delta)\bigg(\nm{f}_{L^1}+\nm{\vv\cdot\nx f}_{L^1}\bigg).
\end{eqnarray}
\end{lemma}
\begin{proof}
See \cite[Lemma 2.1]{Esposito.Guo.Kim.Marra2013}.
\end{proof}
\begin{lemma}(Green's Identity)\label{wellposedness prelim lemma 2}
Assume $f(\vx,\vv),\ g(\vx,\vv)\in L^2(\Omega\times\r^2)$ and
$\vv\cdot\nx f,\ \vv\cdot\nx g\in L^2(\Omega\times\r^2)$ with $f,\
g\in L^2(\gamma)$. Then
\begin{eqnarray}
\iint_{\Omega\times\r^2}\bigg((\vv\cdot\nx f)g+(\vv\cdot\nx
g)f\bigg)\ud{\vx}\ud{\vv}=\int_{\gamma_+}fg\ud{\gamma}-\int_{\gamma_-}fg\ud{\gamma}.
\end{eqnarray}
\end{lemma}
\begin{proof}
See \cite[Lemma 2.2]{Esposito.Guo.Kim.Marra2013}.
\end{proof}

\newpage

\section{Asymptotic Analysis}

In this section, we will construct the asymptotic expansion of the equation
\begin{eqnarray}\label{small system}
\left\{
\begin{array}{rcl}
\e\vw\cdot\nx
f^{\e}+\ll[f^{\e}]&=&\Gamma[f^{\e},f^{\e}],\\\rule{0ex}{1.5em}
f^{\e}(\vx_0,\vw)&=&\pe[f^{\e}](\vx_0,\vw) \ \ \text{for}\ \ \vx_0\in\p\Omega\ \ \text{and}\ \ \vw\cdot\vn(\vx_0)<0,
\end{array}
\right.
\end{eqnarray}
with the normalization condition
\begin{eqnarray}\label{normalization}
\int_{\Omega}\int_{\r^2}f^{\e}(\vx,\vw)\m^{\frac{1}{2}}(\vw)\ud{\vw}\ud{\vx}=0.
\end{eqnarray}

\subsection{Interior Expansion}

We define the interior expansion
\begin{eqnarray}\label{interior expansion}
\f(\vx,\vw)\sim\sum_{k=1}^{3}\e^k\f_k(\vx,\vw).
\end{eqnarray}
Plugging it into the equation (\ref{small system}) and comparing the order of $\e$, we obtain
\begin{eqnarray}
\ll[\f_1]&=&0,\label{interior expansion 1}\\
\ll[\f_2]&=&-\vw\cdot\nx\f_1+\Gamma[\f_1,\f_1],\label{interior expansion 2}\\
\ll[\f_3]&=&-\vw\cdot\nx\f_2+2\Gamma[\f_1,\f_2].\label{interior expansion 3}
\end{eqnarray}
The following analysis is standard and well-known. We mainly refer to the method in \cite{Sone2002, Sone2007}. The solvability of
\begin{eqnarray}
\ll[\f_k]&=&S
\end{eqnarray}
requires that
\begin{eqnarray}
\int_{\r^2}S(\vw)\psi(\vw)\ud{\vw}=0
\end{eqnarray}
for any $\psi$ satisfying $\ll[\psi]=0$. Then each $\f_k$ consists of three parts:
\begin{eqnarray}
\f_k(\vx,\vw)=A_k (\vx,\vw)+B_k (\vx,\vw)+C_k (\vx,\vw),
\end{eqnarray}
where
\begin{eqnarray}
A_k (\vx,\vw)=\m^{\frac{1}{2}}(\vw)\left(A_{k,0} (\vx)+A_{k,1} (\vx)v_1+A_{k,2} (\vx)v_2+A_{k,3} (\vx)\bigg(\frac{\abs{\vw}^2-2}{2}\bigg)\right),
\end{eqnarray}
is the macroscopic part,
\begin{eqnarray}
B_k (\vx,\vw)=\m^{\frac{1}{2}}(\vw)\left(B_{k,0} (\vx)+B_{k,1} (\vx)v_1+B_{k,2} (\vx)v_2+B_{k,3} (\vx)\bigg(\frac{\abs{\vw}^2-2}{2}\bigg)\right),
\end{eqnarray}
is the connection part, with $B_{k}$ depending on $A_{s} $ for $1\leq s\leq k-1$ as
\begin{eqnarray}\label{at 12}
B_{k,0} &=&0,\\
B_{k,1} &=&\sum_{i=1}^{k-1}A_{i,0} A_{k-i,1} ,\\
B_{k,2} &=&\sum_{i=1}^{k-1}A_{i,0} A_{k-i,2} ,\\
B_{k,3} &=&\sum_{i=1}^{k-1}\bigg(A_{i,0} A_{k-i,3} +A_{i,1} A_{k-i,1} +A_{i,2} A_{k-i,2}
+\sum_{j=1}^{k-1-i}A_{i,0} (A_{j,1} A_{k-i-j,1} +A_{j,2} A_{k-i-j,2} )\bigg),
\end{eqnarray}
and $C_k(\vx,\vw)$ is the orthogonal part satisfying
\begin{eqnarray}
\int_{\r^2}\m^{\frac{1}{2}}(\vw)C_k (\vx,\vw)\left(\begin{array}{c}1\\\vw\\\abs{\vw}^2
\end{array}\right)\ud{\vw}=0,
\end{eqnarray}
with
\begin{eqnarray}\label{at 13}
\ll[C_k ]&=&-\vw\cdot\nx\f_{k-1}+\sum_{i=1}^{k-1}\Gamma[\f_i,\f_{k-i}],
\end{eqnarray}
which can be uniquely determined. Hence, we only need to determine $A_k$. Traditionally, we write
\begin{eqnarray}
A_k=\m^{\frac{1}{2}}\left(\rh_k +\vu_k\cdot\vw+\th_k \left(\frac{\abs{\vw}^2-2}{2}\right)\right),
\end{eqnarray}
where the coefficients $\rh_k$, $\u_k$ and $\th_k$ represent density, velocity and temperature in the macroscopic scale. Then the analysis in \cite{Sone2002, Sone2007} shows that $A_k $ satisfies the equations as follows:\\
\ \\
$1^{st}$-order expansion:
\begin{eqnarray}
P_1 -(\rh_1 +\th_1 )&=&0,\\
\nx P_1 &=&0,\\
\nx\cdot\uh_1 &=&0,
\end{eqnarray}
$2^{nd}$-order expansion:
\begin{eqnarray}
P_2 -(\rh_2 +\th_2 +\rh_1 \th_1 )&=&0,\\
\uh_1 \cdot\nx\uh_1 -\gamma_1\dx\uh_1 +\nx P_2 &=&0,\\
\uh_1 \cdot\nx\th_1 -\gamma_2\dx\th_1 &=&0,\\
\nx\cdot\uh_2 +\uh_1\cdot\nx\rh_1&=&0.
\end{eqnarray}
Here $P_1$ and $P_2$ represent the pressure, $\gamma_1$ and $\gamma_2$ are constants.

\subsection{Boundary Layer Expansion with Geometric Correction}

We will use the Cartesian coordinate
system for the interior solution, and a local coordinate system in a neighborhood of the boundary for the boundary layer.

Assume the Cartesian coordinate is $\vx=(x_1,x_2)$. Using polar coordinates system $(r,\theta)\in[0,\infty)\times[-\pi,\pi)$ and choosing pole in $\Omega$, we assume $\vx_0\in\p\Omega$ is
\begin{eqnarray}
\left\{
\begin{array}{rcl}
x_{1,0}&=&r(\theta)\cos\theta,\\
x_{2,0}&=&r(\theta)\sin\theta,
\end{array}
\right.
\end{eqnarray}
where $r(\theta)>0$ is a given function describing the boundary curve. Our local coordinate system is a modification of the polar coordinate
system.

In the domain near the boundary, for each $\theta$, we have the
outward unit normal vector
\begin{eqnarray}
\vn=\left(\frac{r(\theta)\cos\theta+r'(\theta)\sin\theta}{\sqrt{r(\theta)^2+r'(\theta)^2}},\frac{r(\theta)\sin\theta-r'(\theta)\cos\theta}{\sqrt{r(\theta)^2+r'(\theta)^2}}\right),
\end{eqnarray}
where $r'(\theta)=\dfrac{\ud{r}}{\ud{\theta}}$. We can determine each
point $\vx\in\bar\Omega$ as $\vx=\vx_0-\mathfrak{N}\vn$ where $\mathfrak{N}$ is the normal distance to the boundary point $\vx_0$. In detail, this means
\begin{eqnarray}\label{local}
\left\{
\begin{array}{rcl}
x_1&=&r(\theta)\cos\theta-\mathfrak{N}\dfrac{r(\theta)\cos\theta+r'(\theta)\sin\theta}{\sqrt{r(\theta)^2+r'(\theta)^2}},\\\rule{0ex}{2.0em}
x_2&=&r(\theta)\sin\theta-\mathfrak{N}\dfrac{r(\theta)\sin\theta-r'(\theta)\cos\theta}{\sqrt{r(\theta)^2+r'(\theta)^2}}.
\end{array}
\right.
\end{eqnarray}
It is easy to see that $\mathfrak{N}=0$ denotes the boundary $\p\Omega$ and $\mathfrak{N}>0$ denotes the interior of $\Omega$. $(\mathfrak{N},\theta)$ is the desired local coordinate system.

Direct computation in \cite{AA007} reveals that
\begin{eqnarray}
\frac{\p\theta}{\p x_1}=\frac{MP}{P^3+Q\mathfrak{N}},&\quad&
\frac{\p\mathfrak{N}}{\p x_1}=-\frac{N}{P},\\
\frac{\p\theta}{\p x_2}=\frac{NP}{P^3+Q\mathfrak{N}},&\quad&
\frac{\p\mathfrak{N}}{\p x_2}=\frac{M}{P},
\end{eqnarray}
where
\begin{eqnarray}
P&=&(r^2+r'^2)^{\frac{1}{2}},\\
Q&=&rr''-r^2-2r'^2,\\
M&=&-r\sin\theta+r'\cos\theta,\\
N&=&r\cos\theta+r'\sin\theta.
\end{eqnarray}
Therefore, noting the fact that for $C^2$ convex domains, the curvature
\begin{eqnarray}
\kappa(\theta)=\frac{r^2+2r'^2-rr''}{(r^2+r'^2)^{\frac{3}{2}}}>0,
\end{eqnarray}
and the radius of curvature
\begin{eqnarray}
R_{\kappa}(\theta)=\frac{1}{\kappa(\theta)}=\frac{(r^2+r'^2)^{\frac{3}{2}}}{r^2+2r'^2-rr''}>0,
\end{eqnarray}
we define substitutions as follows:\\
\ \\
Substitution 1: Coordinate Substitution\\
Let $(x_1,x_2)\rt (\mathfrak{N},\theta)$ with
$0\leq\mathfrak{N}<R_{\min}$ for $R_{\min}=\min_{\theta}R_{\kappa}$ as
\begin{eqnarray}\label{substitution 1}
\left\{
\begin{array}{rcl}
x_1&=&r(\theta)\cos\theta-\mathfrak{N}\dfrac{r(\theta)\cos\theta+r'(\theta)\sin\theta}{\sqrt{r(\theta)^2+r'(\theta)^2}},\\\rule{0ex}{2.0em}
x_2&=&r(\theta)\sin\theta-\mathfrak{N}\dfrac{r(\theta)\sin\theta-r'(\theta)\cos\theta}{\sqrt{r(\theta)^2+r'(\theta)^2}},
\end{array}
\right.
\end{eqnarray}
and then the equation (\ref{small system}) is transformed into
\begin{eqnarray}
\left\{
\begin{array}{l}
\displaystyle\e\Bigg(v_1\frac{-r\cos\theta-r'\sin\theta}{(r^2+r'^2)^{\frac{1}{2}}}+v_2\frac{-r\sin\theta+r'\cos\theta}{(r^2+r'^2)^{\frac{1}{2}}}\Bigg)\frac{\p f^{\e}}{\p\mathfrak{N}}\\
\displaystyle+\e\Bigg(v_1\frac{-r\sin\theta+r'\cos\theta}{(r^2+r'^2)}+v_2\frac{r\cos\theta+r'\sin\theta}{(r^2+r'^2)}\Bigg)
\frac{1}{(1-\kappa\mathfrak{N})}\frac{\p f^{\e}}{\p\theta}+\ll[f^{\e}]=\Gamma[f^{\e},f^{\e}],\\\rule{0ex}{2.0em}
f^{\e}(0,\theta,\vw)=\pe[f^{\e}](0,\theta,\vw)\ \ \text{for}\
\ \vw\cdot\vn<0,
\end{array}
\right.
\end{eqnarray}
where
\begin{eqnarray}
\vw\cdot\vn=v_1\frac{-r\cos\theta-r'\sin\theta}{(r^2+r'^2)^{\frac{1}{2}}}+v_2\frac{-r\sin\theta+r'\cos\theta}{(r^2+r'^2)^{\frac{1}{2}}},
\end{eqnarray}
and
\begin{eqnarray}
\pe[f^{\e}](0,\theta,\vw)=\mb(\theta,\vv)\m^{-\frac{1}{2}}(\vv)
\displaystyle\int_{\vuu\cdot\vn(\theta)>0}\m^{\frac{1}{2}}(\vuu)
f^{\e}(0,\theta,\vuu)\abs{\vuu\cdot\vn(\theta)}\ud{\vuu}+\m^{-\frac{1}{2}}(\vv)\bigg(\mb(\theta,\vv)-\m(\vv)\bigg).
\end{eqnarray}
\ \\
Substitution 2: Velocity Substitution.\\
Define the orthogonal velocity substitution $\vw=(v_1,v_2)\rt\vww=(\va,\vb)$ as
\begin{eqnarray}
\left\{
\begin{array}{rcl}
v_1\dfrac{-r\cos\theta-r'\sin\theta}{(r^2+r'^2)^{\frac{1}{2}}}+v_2\dfrac{-r\sin\theta+r'\cos\theta}{(r^2+r'^2)^{\frac{1}{2}}}&=&\va,\\\rule{0ex}{2.0em}
v_1\dfrac{-r\sin\theta+r'\cos\theta}{(r^2+r'^2)^{\frac{1}{2}}}+v_2\dfrac{r\cos\theta+r'\sin\theta}{(r^2+r'^2)^{\frac{1}{2}}}&=&\vb.
\end{array}
\right.
\end{eqnarray}
Then we have
\begin{eqnarray}
\frac{\p}{\p\theta}&\rt&\frac{\p}{\p\theta}-\kappa(r^2+r'^2)^{\frac{1}{2}}\vb\frac{\p}{\p\va}
+\kappa(r^2+r'^2)^{\frac{1}{2}}\va\frac{\p}{\p\vb}.
\end{eqnarray}
The transport operator is
\begin{eqnarray}
\vw\cdot\nx&=&\va\frac{\p}{\p\mathfrak{N}}-\frac{\vb}{\rk-\mathfrak{N}}\dfrac{\rk}{(r^2+r'^2)^{\frac{1}{2}}}\frac{\p}{\p\theta}
-\frac{\vb^2}{\rk-\mathfrak{N}}\dfrac{\p}{\p\va}+\frac{\va\vb}{\rk-\mathfrak{N}}\dfrac{\p}{\p\vb}.
\end{eqnarray}
Hence,
the equation (\ref{small system}) is transformed into
\begin{eqnarray}
\left\{
\begin{array}{l}\displaystyle
\e\va\dfrac{\p f^{\e}}{\p\mathfrak{N}}-\e\frac{\vb}{\rk-\mathfrak{N}}\dfrac{\rk}{(r^2+r'^2)^{\frac{1}{2}}}\frac{\p f^{\e}}{\p\theta}
-\e\frac{\vb^2}{\rk-\mathfrak{N}}\dfrac{\p f^{\e}}{\p\va}+\e\frac{\va\vb}{\rk-\mathfrak{N}}\dfrac{\p f^{\e}}{\p\vb}
+\ll[f^{\e}]=\Gamma[f^{\e},f^{\e}],\\\rule{0ex}{2.0em}
f^{\e}(0,\theta,\vww)=\pe[f^{\e}](0,\theta,\vww)\ \
\text{for}\ \ \va>0,
\end{array}
\right.
\end{eqnarray}
where
\begin{eqnarray}
\pe[f^{\e}](0,\theta,\vww)&=&\mb(\theta,\vww)\m^{-\frac{1}{2}}(\vww)
\displaystyle\int_{\vuu_{\eta}<0}\m^{\frac{1}{2}}(\vuu_{\eta})
f^{\e}(0,\theta,\vuu_{\eta})\abs{\vuu_{\eta}}\ud{\vuu}+\m^{-\frac{1}{2}}(\vvv)\bigg(\mb(\theta,\vww)-\m(\vww)\bigg).
\end{eqnarray}
\ \\
Substitution 3: Scaling Substitution.\\
We define the rescaled variable $\eta=\dfrac{\mathfrak{N}}{\e}$, which implies $\dfrac{\p}{\p\mathfrak{N}}=\dfrac{1}{\e}\dfrac{\p}{\p\eta}$. Then, under the substitution $\mathfrak{N}\rt\eta$, the equation (\ref{small system}) is transformed into
\begin{eqnarray}\label{small system.}
\left\{
\begin{array}{l}\displaystyle
\va\dfrac{\p f^{\e}}{\p\eta}-\e\frac{\vb}{\rk-\e\eta}\dfrac{\rk}{(r^2+r'^2)^{\frac{1}{2}}}\frac{\p f^{\e}}{\p\theta}
-\e\frac{\vb^2}{\rk-\e\eta}\dfrac{\p f^{\e}}{\p\va}+\e\frac{\va\vb}{\rk-\e\eta}\dfrac{\p f^{\e}}{\p\vb}
+\ll[f^{\e}]=\Gamma[f^{\e},f^{\e}],\\\rule{0ex}{2.0em}
f^{\e}(0,\theta,\vww)=\pe[f^{\e}](0,\theta,\vww)\ \
\text{for}\ \ \va>0,
\end{array}
\right.
\end{eqnarray}
where
\begin{eqnarray}
\pe[f^{\e}](0,\theta,\vww)&=&\mb(\theta,\vww)\m^{-\frac{1}{2}}(\vww)
\displaystyle\int_{\vuu_{\eta}<0}\m^{\frac{1}{2}}(\vuu_{\eta})
f^{\e}(0,\theta,\vuu_{\eta})\abs{\vuu_{\eta}}\ud{\vuu}+\m^{-\frac{1}{2}}(\vvv)\bigg(\mb(\theta,\vww)-\m(\vww)\bigg).
\end{eqnarray}
We define the boundary layer expansion as follows:
\begin{eqnarray}\label{boundary layer expansion}
\fb(\eta,\theta,\vww)\sim\sum_{k=1}^{2}\e^k\fb_k(\eta,\theta,\vww),
\end{eqnarray}
where $\fb_k$ can be defined by comparing the order of $\e$ via
plugging (\ref{boundary layer expansion}) into the equation
(\ref{small system.}). Thus, in a neighborhood of the boundary, we have
\begin{eqnarray}
\va\dfrac{\p\fb_1}{\p\eta}-\dfrac{\e}{\rk-\e\eta}\bigg(\vb^2\dfrac{\p\fb_1}{\p\va}-\va\vb\dfrac{\p\fb_1}{\p\vb}\bigg)+\ll[\fb_1]&=&0,\label{expansion temp 6}\\
\\
\va\dfrac{\p\fb_2}{\p\eta}-\dfrac{\e}{\rk-\e\eta}\bigg(\vb^2\dfrac{\p\fb_2}{\p\va}-\va\vb\dfrac{\p\fb_2}{\p\vb}\bigg)+\ll[\fb_2]&=&
2\Gamma[\f_1,\fb_1]+\Gamma[\fb_1,\fb_1]+\frac{\vb}{\rk-\e\eta}\dfrac{\rk}{(r^2+r'^2)^{\frac{1}{2}}}\frac{\p \fb_1}{\p\theta}.\no
\end{eqnarray}

\subsection{Expansion of Boundary Conditions}

The bridge between the interior solution and boundary layer
is the boundary condition
\begin{eqnarray}
f^{\e}(\vx_0,\vw)&=&\pe[f^{\e}](\vx_0,\vw),
\end{eqnarray}
where
\begin{eqnarray}
\\
\pe[f^{\e}](\vx_0,\vw)=\mb(\vx_0,\vv)\m^{-\frac{1}{2}}(\vv)
\displaystyle\int_{\vuu\cdot\vn(\vx_0)>0}\m^{\frac{1}{2}}(\vuu)
f^{\e}(\vx_0,\vuu)\abs{\vuu\cdot\vn(\vx_0)}\ud{\vuu}+\m^{-\frac{1}{2}}(\vw)\bigg(\mb(\vx_0,\vv)-\m(\vv)\bigg).\no
\end{eqnarray}
Plugging the combined expansion
\begin{eqnarray}
f^{\e}\sim\sum_{k=1}^{3}\e^k\f_k+\sum_{k=1}^{2}\e^k\fb_k,
\end{eqnarray}
into the boundary condition and comparing the order of $\e$, we obtain
\begin{eqnarray}
\f_1+\fb_1&=&\m^{\frac{1}{2}}(\vw)
\int_{\vuu\cdot\vn(\vx_0)>0}\m^{\frac{1}{2}}(\vuu)(\f_1+\fb_1)\abs{\vuu\cdot\vn(\vx_0)}\ud{\vuu}
+\m_1(\vx_0,\vw),\\
\f_2+\fb_2&=&\m^{\frac{1}{2}}(\vw)
\int_{\vuu\cdot\vn(\vx_0)>0}\m^{\frac{1}{2}}(\vuu)(\f_2+\fb_2)\abs{\vuu\cdot\vn(\vx_0)}\ud{\vuu}\\
&&+\m_1(\vx_0,\vw)
\int_{\vuu\cdot\vn(\vx_0)>0}\m^{\frac{1}{2}}(\vuu)(\f_1+\fb_1)\abs{\vuu\cdot\vn(\vx_0)}\ud{\vuu}
+\m_2(\vx_0,\vw).\no
\end{eqnarray}
In particular, we do not further expand the boundary layer, so we directly require
\begin{eqnarray}
\f_3&=&\m^{\frac{1}{2}}(\vw)
\int_{\vuu\cdot\vn(\vx_0)>0}\m^{\frac{1}{2}}(\vuu)\f_3\abs{\vuu\cdot\vn(\vx_0)}\ud{\vuu}\\
&&+\m_2(\vx_0,\vw)
\int_{\vuu\cdot\vn(\vx_0)>0}\m^{\frac{1}{2}}(\vuu)(\f_1+\fb_1)\abs{\vuu\cdot\vn(\vx_0)}\ud{\vuu}\no\\
&&+\m_1(\vx_0,\vw)
\int_{\vuu\cdot\vn(\vx_0)>0}\m^{\frac{1}{2}}(\vuu)(\f_2+\fb_2)\abs{\vuu\cdot\vn(\vx_0)}\ud{\vuu}
+\m_3(\vx_0,\vw).\no
\end{eqnarray}
Define
\begin{eqnarray}
\pp[f](\vx_0,\vw)=\m^{\frac{1}{2}}(\vw)
\int_{\vuu\cdot\vn(\vx_0)>0}\m^{\frac{1}{2}}(\vuu)f(\vx_0,\vuu)\abs{\vuu\cdot\vn(\vx_0)}\ud{\vuu}.
\end{eqnarray}
Then we have
\begin{eqnarray}
\f_1+\fb_1&=&\pp[\f_1+\fb_1]+\m_1(\vx_0,\vw),\\
\f_2+\fb_2&=&\pp[\f_2+\fb_2]+\m_1(\vx_0,\vw)
\int_{\vuu\cdot\vn(\vx_0)>0}\m^{\frac{1}{2}}(\vuu)(\f_1+\fb_1)\abs{\vuu\cdot\vn(\vx_0)}
\ud{\vuu}+\m_2(\vx_0,\vw),
\end{eqnarray}
and
\begin{eqnarray}
\f_3&=&\pp[\f_3]+\m_2(\vx_0,\vw)
\int_{\vuu\cdot\vn(\vx_0)>0}\m^{\frac{1}{2}}(\vuu)(\f_1+\fb_1)\abs{\vuu\cdot\vn(\vx_0)}\ud{\vuu}\\
&&+\m_1(\vx_0,\vw)
\int_{\vuu\cdot\vn(\vx_0)>0}\m^{\frac{1}{2}}(\vuu)(\f_2+\fb_2)\abs{\vuu\cdot\vn(\vx_0)}\ud{\vuu}
+\m_3(\vx_0,\vw).\no
\end{eqnarray}
This is the boundary conditions $\f_k$ and $\fb_k$ need to satisfy.

\subsection{Matching Procedure}

Define the length of boundary layer $L=\e^{-s}$ for $0<s<\dfrac{1}{2}$.
Also, denote $\rr[\va,\vb]=(-\va,\vb)$. We divide the construction of the asymptotic expansion into several steps for each $k\geq1$:\\
\ \\
Step 1: Construction of $\f_1$ and $\fb_1$.\\
A direct computation reveals that $\f_1=A_1+B_1+C_1$, where $B_1=C_1=0$. Based on our expansion,
\begin{eqnarray}
\m_1&=&\m^{\frac{1}{2}}\left(\rh_{\bb,1}+\vu_{\bb,1}\cdot\vw+\th_{\bb,1}\frac{\abs{\vw}^2-2}{2}\right).
\end{eqnarray}
Define
\begin{eqnarray}
\f_1&=&\m^{\frac{1}{2}}\left(\rh_{1}+\vu_{1}\cdot\vw+\th_{1}\frac{\abs{\vw}^2-2}{2}\right),
\end{eqnarray}
satisfying the Navier-Stokes-Fourier system as
\begin{eqnarray}\label{interior 1}
\left\{
\begin{array}{rcl}
\nx(\rh_1 +\th_1 )&=&0,\\\rule{0ex}{1.0em}
\uh_1\cdot\nx\uh_1 -\gamma_1\dx\uh_1 +\nx P_2 &=&0,\\\rule{0ex}{1.0em}
\nx\cdot\uh_1 &=&0,\\\rule{0ex}{1.0em}
\uh_1 \cdot\nx\th_1 -\gamma_2\dx\th_1 &=&0,\\\rule{0ex}{1.0em}
\rh_1 (\vx_0)&=&\rh_{\bb,1}(\vx_0)+M_1(\vx_0),\\
\uh_1 (\vx_0)&=&\vu_{\bb,1}(\vx_0),\\
\th_1 (\vx_0)&=&\th_{\bb,1}(\vx_0),\\
\end{array}
\right.
\end{eqnarray}
where $M_1(\vx_0)$ is a constant such that the Boussinesq relation
\begin{eqnarray}
\rh_1+\th_1=\text{constant},
\end{eqnarray}
is satisfied. Note that this constant is determined by the normalization condition.
\begin{eqnarray}
\int_{\Omega}\int_{\r^2}\f_1(\vx,\vw)\m^{\frac{1}{2}}(\vw)\ud{\vw}\ud{\vx}=0,
\end{eqnarray}
and we are able to add $M_1(\vx_0)$ freely since $\mu^{\frac{1}{2}}=\pp[\mu^{\frac{1}{2}}]$. Then
based on the compatibility condition of $\mu_1$ as
\begin{eqnarray}
\int_{\vuu\cdot\vn(\vx_0)>0}\m^{\frac{1}{2}}(\vuu)\mu_1(\vx_0,\vuu)\abs{\vuu\cdot\vn(\vx_0)}\ud{\vuu}=0,
\end{eqnarray}
we naturally obtain $\pp[\f_1]=M_1\mu^{\frac{1}{2}}$, which means
\begin{eqnarray}
\f_1=\pp[\f_1]+\m_1\ \ \text{on}\ \ \p\Omega.
\end{eqnarray}
Therefore, it is not necessary to introduce the boundary layer at this order and we simply take $\fb_1=0$.\\
\ \\
Step 2: Construction of $\f_2$ and $\fb_2$.\\
Define $\f_2=A_2+B_2+C_2$, where $B_2$ and $C_2$ can be uniquely determined following previous analysis, and
\begin{eqnarray}
A_2&=&\m^{\frac{1}{2}}\left(\rh_{2}+\vu_{2}\cdot\vw+\th_{2}\frac{\abs{\vw}^2-2}{2}\right),
\end{eqnarray}
satisfying a more complicated fluid-type equation as in \cite{Sone2002, Sone2007}. On the other hand,
$\fb_2$ satisfies the $\e$-Milne problem with geometric correction
\begin{eqnarray}
\left\{
\begin{array}{l}\displaystyle
\va\frac{\p \fb_2}{\p\eta}-\dfrac{\e}{\rk-\e\eta}\bigg(\vb^2\dfrac{\p\fb_2}{\p\va}-\va\vb\dfrac{\p\fb_2}{\p\vb}\bigg)+\ll[\fb_2]
=0\ \ \text{for}\ \ (\eta,\theta,\vvv)\in[0,L]\times[-\pi,\pi)\times\r^2,\\\rule{0ex}{2.0em}
\fb_2(0,\theta,\vvv)=h(\theta,\vvv)-\tilde h(\theta,\vvv)\ \
\text{for}\ \ \va>0,\\\rule{0ex}{2.0em}
\displaystyle\fb_2(L,\theta,\vvv)
=\fb_2(L,\theta,\rr[\vvv]),
\end{array}
\right.
\end{eqnarray}
with the in-flow boundary data
\begin{eqnarray}
\\
h(\theta,\vvv)&=&\m_1(\vx_0,\vw)
\int_{\vuu\cdot\vn(\vx_0)>0}\m^{\frac{1}{2}}(\vuu)(\f_1+\fb_1)\abs{\vuu\cdot\vn(\vx_0)}
\ud{\vuu}+\m_2(\vx_0,\vw)-\bigg((B_2+C_2)-\pp[B_2+C_2]\bigg).\no
\end{eqnarray}
Based on Theorem \ref{Milne theorem 2}, there exists
\begin{eqnarray}
\tilde h(\theta,\vvv)=\m^{\frac{1}{2}}\bigg(\tilde
D_{0}(\theta)+\tilde D_{1}(\theta)\va+\tilde D_{2}(\theta)\vb+\tilde
D_{3}(\theta)\frac{\abs{\vvv}^2-2}{2})\bigg),
\end{eqnarray}
such that the $\e$-Milne problem with geometric correction is well-posed and the solution decays exponentially fast. In particular, $\tilde D_1=0$.
Then we further require that $A_2$ satisfies the boundary condition
\begin{eqnarray}
A_2(\vx_0,\vw)=\tilde h(\theta,\vvv)+M_2(\vx_0)\m^{\frac{1}{2}}(\vv).
\end{eqnarray}
Here, the constant $M_2(\vx_0)$ is chosen to enforce the Boussinesq relation
\begin{eqnarray}
P_2 -(\rh_2 +\th_2 +\rh_1 \th_1 )&=&0.
\end{eqnarray}
Similar to the construction of $\f_1$, we can choose the constant to satisfy the normalization condition
\begin{eqnarray}
\int_{\Omega}\int_{\r^2}(\f_2+\fb_2)(\vx,\vw)\m^{\frac{1}{2}}(\vw)\ud{\vw}\ud{\vx}=0.
\end{eqnarray}
Also, the construction implies that at boundary, we have
\begin{eqnarray}
A_2+\fb_2&=&M_2\m^{\frac{1}{2}}+h\\
&=&M_2\m^{\frac{1}{2}}+\m_1(\vx_0,\vw)
\int_{\vuu\cdot\vn(\vx_0)>0}\m^{\frac{1}{2}}(\vuu)(\f_1+\fb_1)\abs{\vuu\cdot\vn(\vx_0)}
\ud{\vuu}+\m_2(\vx_0,\vw)\no\\
&&-\bigg((B_2+C_2)-\pp[B_2+C_2]\bigg).\no
\end{eqnarray}
Comparing this with the desired boundary expansion
\begin{eqnarray}
\f_2+\fb_2=\pp[\f_2+\fb_2]+\b_2,
\end{eqnarray}
we only need to verify that
\begin{eqnarray}
\pp[A_2+\fb_2]=M_2\m^{\frac{1}{2}}.
\end{eqnarray}
We may direct verify the zero mass-flux condition of $\fb_2$ as
\begin{eqnarray}
\int_{\r^2}\m^{\frac{1}{2}}(\vuu)\fb_2(\vx,\vuu)(\vuu\cdot\vn)\ud{\vuu}=0,
\end{eqnarray}
and the compatibility condition
\begin{eqnarray}
\int_{\vuu\cdot\vn(\vx_0)>0}\m^{\frac{1}{2}}(\vuu)\mu_1(\vx_0,\vuu)\abs{\vuu\cdot\vn(\vx_0)}\ud{\vuu}
=\int_{\vuu\cdot\vn(\vx_0)>0}\m^{\frac{1}{2}}(\vuu)\mu_2(\vx_0,\vuu)\abs{\vuu\cdot\vn(\vx_0)}\ud{\vuu}=0.
\end{eqnarray}
Then we naturally derive
\begin{eqnarray}
&&\pp[A_2+\fb_2]\\
&=&
\m^{\frac{1}{2}}\int_{\vuu\cdot\vn>0}\m^{\frac{1}{2}}(\vuu)A_2(\vx,\vuu)(\vuu\cdot\vn)\ud{\vuu}
+\m^{\frac{1}{2}}\int_{\vuu\cdot\vn>0}\m^{\frac{1}{2}}(\vuu)\fb_2(\vx,\vuu)(\vuu\cdot\vn)\ud{\vuu}\no\\
&=&
M_2\m^{\frac{1}{2}}+\m^{\frac{1}{2}}\int_{\vuu\cdot\vn>0}\m^{\frac{1}{2}}(\vuu)\tilde h(\vx,\vuu)(\vuu\cdot\vn)\ud{\vuu}+
\m^{\frac{1}{2}}\int_{\vuu\cdot\vn>0}\m^{\frac{1}{2}}(\vuu)\fb_2(\vx,\vuu)(\vuu\cdot\vn)\ud{\vuu}\no\\
&=&
M_2\m^{\frac{1}{2}}+\m^{\frac{1}{2}}\int_{\vuu\cdot\vn>0}\m^{\frac{1}{2}}(\vuu)\tilde h(\vx,\vuu)(\vuu\cdot\vn)\ud{\vuu}-
\m^{\frac{1}{2}}\int_{\vuu\cdot\vn<0}\m^{\frac{1}{2}}(\vuu)\fb_2(\vx,\vuu)(\vuu\cdot\vn)\ud{\vuu}\no\\
&=&
M_2\m^{\frac{1}{2}}+\m^{\frac{1}{2}}\int_{\vuu\cdot\vn>0}\m^{\frac{1}{2}}(\vuu)\tilde h(\vx,\vuu)(\vuu\cdot\vn)\ud{\vuu}-
\m^{\frac{1}{2}}\int_{\vuu\cdot\vn<0}\m^{\frac{1}{2}}(\vuu)(h-\tilde h)(\vx,\vuu)(\vuu\cdot\vn)\ud{\vuu}\no\\
&=&
M_2\m^{\frac{1}{2}}+\m^{\frac{1}{2}}\int_{\r^2}\m^{\frac{1}{2}}(\vuu)\tilde h(\vx,\vuu)(\vuu\cdot\vn)\ud{\vuu}-
\m^{\frac{1}{2}}\int_{\vuu\cdot\vn<0}\m^{\frac{1}{2}}(\vuu)h(\vx,\vuu)(\vuu\cdot\vn)\ud{\vuu}\no\\
&=&M_2\m^{\frac{1}{2}}+0-0\no\\
&=&M_2\m^{\frac{1}{2}}.\no
\end{eqnarray}
$F_3$ can be defined in a similar fashion which satisfies an even more complicated fluid-type system.

\newpage

\section{Remainder Estimates}

We consider the linearized stationary Boltzmann equation
\begin{eqnarray}\label{linear steady}
\left\{
\begin{array}{l}
\e\vv\cdot\nx f+\ll[f]=S(\vx,\vv)\ \ \text{in}\ \ \Omega,\\
f(\vx_0,\vv)=\pp[f](\vx_0,\vv)+h(\vx_0,\vv)\ \ \text{for}\ \ \vx_0\in\p\Omega\ \
\text{and}\ \ \vv\cdot\vn<0,
\end{array}
\right.
\end{eqnarray}
where
\begin{eqnarray}
\pp[f](\vx_0,\vw)=\m^{\frac{1}{2}}(\vw)
\int_{\vuu\cdot\vn(\vx_0)>0}\m^{\frac{1}{2}}(\vuu)f(\vx_0,\vuu)\abs{\vuu\cdot\vn(\vx_0)}\ud{\vuu},
\end{eqnarray}
provided the compatibility condition
\begin{eqnarray}\label{linear steady compatibility}
\int_{\Omega\times\r^2}S(\vx,\vv)\m^{\frac{1}{2}}(\vv)\ud{\vv}\ud{\vx}&=&0,\ \ \ \
\int_{\gamma_-}h(\vx,\vv)\m^{\frac{1}{2}}(\vv)\ud{\gamma}=0.
\end{eqnarray}
It is easy to see if $f$ is a solution to (\ref{linear steady}), then $f+C\m^{\frac{1}{2}}$ is also a solution for arbitrary $C\in\r$. Hence, we require that the solution should satisfy the normalization condition
\begin{eqnarray}\label{linear steady normalization}
\int_{\Omega\times\r^2}f(\vx,\vv)\m^{\frac{1}{2}}(\vv)\ud{\vv}\ud{\vx}&=&0.
\end{eqnarray}
Our analysis is based on the ideas in \cite{Esposito.Guo.Kim.Marra2013, Guo2010}. Since the well-posedness of (\ref{linear steady}) is standard, we will focus on the a priori estimates here.

\subsection{$L^2$ Estimates}

\begin{lemma}\label{wellposedness prelim lemma 3}
The
solution $f(\vx,\vw)$ to the equation (\ref{linear steady}) satisfies the estimate
\begin{eqnarray}\label{wt 08}
\e\tm{\pk[f]}&\leq& C\bigg(
\e\tss{(1-\pp)[f]}{+}+\tm{(\ik-\pk)[f]}+\tm{S}+\e\tss{h}{-}\bigg).
\end{eqnarray}
\end{lemma}
\begin{proof}
Applying Green's identity in Lemma \ref{wellposedness prelim lemma
2} to the equation (\ref{linear steady}). Then for any $\psi\in L^2(\Omega\times\r^2)$
satisfying $\vv\cdot\nx\psi\in L^2(\Omega\times\r^2)$ and $\psi\in
L^2(\gamma)$, we have
\begin{eqnarray}
&&\e\int_{\gamma_+}f\psi\ud{\gamma}-\e\int_{\gamma_-}f\psi\ud{\gamma}
-\e\int_{\Omega\times\r^2}(\vv\cdot\nx\psi)f=-\int_{\Omega\times\r^2}\psi\ll[f]+\int_{\Omega\times\r^2}S\psi.
\end{eqnarray}
Considering that $f=\pk[f]+(\ik-\pk)[f]$ and $\ll\Big[\pk[f]\Big]=0$, we may simplify
\begin{eqnarray}\label{wt 11}
&&\e\int_{\gamma_+}f\psi\ud{\gamma}-\e\int_{\gamma_-}f\psi\ud{\gamma}
-\e\int_{\Omega\times\r^2}(\vv\cdot\nx\psi)f=-\int_{\Omega\times\r^2}\psi\ll\Big[(\ik-\pk)[f]\Big]+\int_{\Omega\times\r^2}S\psi.
\end{eqnarray}
Since
\begin{eqnarray}
\pk[f]=\m^{\frac{1}{2}}\bigg(a+\vv\cdot\vbb+\frac{\abs{\vv}^2-2}{2}c\bigg),
\end{eqnarray}
our goal is to choose a particular test function $\psi$ to estimate
$a$, $\vbb$ and $c$.\\
\ \\
Step 1: Estimates of $c$.\\
We choose the test function
\begin{eqnarray}\label{wt 12}
\psi=\psi_c=\m^{\frac{1}{2}}(\vv)\left(\abs{\vv}^2-\beta_c\right)\Big(\vv\cdot\nx\phi_c(\vx)\Big),
\end{eqnarray}
where
\begin{eqnarray}
\left\{
\begin{array}{rcl}
-\dx\phi_c&=&c(\vx)\ \ \text{in}\ \
\Omega,\\
\phi_c&=&0\ \ \text{on}\ \ \p\Omega,
\end{array}
\right.
\end{eqnarray}
and $\beta_c$ is a real number to be determined later. Based on the
standard elliptic estimates, we have
\begin{eqnarray}
\nm{\phi_c}_{H^2}\leq C\tm{c}.
\end{eqnarray}
With the choice of (\ref{wt 12}), the right-hand side (RHS) of
(\ref{wt 11}) is bounded by
\begin{eqnarray}
\text{RHS}\leq C\tm{c}\bigg(\tm{(\ik-\pk)[f]}+\tm{S}\bigg).
\end{eqnarray}
We have
\begin{eqnarray}
\vv\cdot\nx\psi_c&=&\m^{\frac{1}{2}}(\vv)\sum_{i,j=1}^2\left(\abs{\vv}^2-\beta_c\right)\vs_i\vs_j\p_{ij}\phi_c,
\end{eqnarray}
so the left-hand side (LHS) of (\ref{wt 11}) takes the form
\begin{eqnarray}\label{wt 13}
\text{LHS}&=&\e\int_{\p\Omega\times\r^2}f\m^{\frac{1}{2}}(\vv)\left(\abs{\vv}^2-\beta_c\right)\left(\sum_{i=1}^2\vs_i\p_i\phi_c\right)(\vv\cdot\vn)\\
&&
-\e\int_{\Omega\times\r^2}f\m^{\frac{1}{2}}(\vv)\left(\abs{\vv}^2-\beta_c\right)\left(\sum_{i,j=1}^2\vs_i\vs_j\p_{ij}\phi_c\right).\no
\end{eqnarray}
We decompose
\begin{eqnarray}
f&=&\pp
[f]+\id_{\gamma_+}(1-\pp)[f]+\id_{\gamma_-}h\ \
\text{on}\ \ \gamma, \label{wt 14}\\
f&=&\m^{\frac{1}{2}}\bigg(a+\vv\cdot
\vbb+\frac{\abs{\vv}^2-2}{2}c\bigg)+(\ik-\pk)[f]\ \ \text{in}\
\ \Omega\times\r^2.\label{wt 15}
\end{eqnarray}
We will choose $\beta_c$ such that
\begin{eqnarray}
\int_{\r^2}\m^{\frac{1}{2}}(\vv)\left(\abs{\vv}^2-\beta_c\right)\vs_i^2\ud{\vv}=0\
\ \text{for}\ \ i=1,2.
\end{eqnarray}
It is easy to check that this $\beta_c$ can always be achieved. Now substitute (\ref{wt 14})
and (\ref{wt 15}) into (\ref{wt 13}). Then based on this choice of
$\beta_c$ and the oddness in $\vv$, there is no $\pp[f]$ contribution in
the first term, and no $a$ contribution in the second term of
(\ref{wt 13}). Since $\vbb$ contribution and the off-diagonal $c$
contribution in the second term of (\ref{wt 13}) also vanish due to
the oddness in $\vv$, we can simplify (\ref{wt 13}) into
\begin{eqnarray}
\text{LHS}
&=&\e\int_{\p\Omega\times\r^2}\id_{\gamma_+}(1-\pp)[f]
\m^{\frac{1}{2}}(\vv)\left(\abs{\vv}^2-\beta_c\right)\left(\sum_{i=1}^2\vs_i\p_i\phi_c\right)(\vv\cdot\vn)\\
&&+\e\int_{\p\Omega\times\r^2}\id_{\gamma_-}h
\m^{\frac{1}{2}}(\vv)\left(\abs{\vv}^2-\beta_c\right)\left(\sum_{i=1}^2\vs_i\p_i\phi_c\right)(\vv\cdot\vn)\no\\
&&-\e\sum_{i=1}^2\int_{\r^2}\m(\vv)\abs{\vs_i}^2\left(\abs{\vv}^2-\beta_c\right)\frac{\abs{\vv}^2-2}{2}\ud{\vv}
\int_{\Omega}c(\vx)\p_{ii}\phi_c(\vx)\ud{\vx}\no\\
&&-\e\int_{\Omega\times\r^2}(\ik-\pk)[f]\m^{\frac{1}{2}}(\vv)\left(\abs{\vv}^2-\beta_c\right)\left(\sum_{i,j=1}^2\vs_i\vs_j\p_{ij}\phi_c\right).\no
\end{eqnarray}
Since
\begin{eqnarray}
\int_{\r^2}\m(\vv)\abs{\vs_i}^2\left(\abs{\vv}^2-\beta_c\right)\frac{\abs{\vv}^2-2}{2}\ud{\vv}=C,
\end{eqnarray}
we have
\begin{eqnarray}
\e\abs{\int_{\Omega}\dx\phi_c(\vx)c(\vx)\ud{\vx}}
&\leq&C\tm{c}\bigg(\e\tss{(1-\pp)[f]}{+}+\tm{(\ik-\pk)[f]}+\tm{S}+\e\tss{h}{-}\bigg),
\end{eqnarray}
where we have used the elliptic estimates and the trace estimate:
$\ts{\nx\phi_c}\leq C\nm{\phi_c}_{H^2}\leq C\tm{c}$. Since $-\dx\phi_c=c$,
we know
\begin{eqnarray}
\e\tm{c}^2
&\leq&C
\tm{c}\bigg(\e\tss{(1-\pp)[f]}{+}+\tm{(\ik-\pk)[f]}+\tm{S}+\e\tss{h}{-}\bigg),
\end{eqnarray}
which further implies
\begin{eqnarray}\label{wt 16}
&&\e\tm{c}\leq C
\bigg(\e\tss{(1-\pp)[f]}{+}+\tm{(\ik-\pk)[f]}+\tm{S}+\e\tss{h}{-}\bigg).
\end{eqnarray}
\ \\
Step 2: Estimates of $\vbb$.\\
We further divide this step into several sub-steps:\\
\ \\
Step 2.1: Estimates of $\Big(\p_{ij}\dx^{-1}b_j\Big)b_i$ for $i,j=1,2$.\\
We choose the test function
\begin{eqnarray}\label{wt 17}
\psi=\psi_b^{i,j}=\m^{\frac{1}{2}}(\vv)\left(\vs_i^2-\beta_b\right)\p_j\phi_b^j,
\end{eqnarray}
where
\begin{eqnarray}
\left\{
\begin{array}{rcl}
-\dx\phi_b^j&=&b_j(\vx)\ \ \text{in}\ \
\Omega,\\
\phi_b^j&=&0\ \ \text{on}\ \ \p\Omega,
\end{array}
\right.
\end{eqnarray}
and $\beta_b$ is a real number to be determined later. Based on the
standard elliptic estimates, we have
\begin{eqnarray}
\nm{\phi_b^j}_{H^2}\leq C \tm{\vbb}.
\end{eqnarray}
With the choice of (\ref{wt 17}), the right-hand side (RHS) of
(\ref{wt 11}) is bounded by
\begin{eqnarray}
\text{RHS}\leq C \tm{\vbb}\bigg(\tm{(\ik-\pk)[f]}+\tm{S}\bigg).
\end{eqnarray}
Hence, the left-hand side (LHS) of (\ref{wt 11}) takes the form
\begin{eqnarray}\label{wt 18}
\text{LHS}&=&\e\int_{\p\Omega\times\r^2}f\m^{\frac{1}{2}}(\vv)\left(\vs_i^2-\beta_b\right)\p_j\phi_b^j(\vv\cdot\vn)\\
&&-\e\int_{\Omega\times\r^2}f\m^{\frac{1}{2}}(\vv)\left(\vs_i^2-\beta_b\right)\left(\sum_{l=1}^2\vs_l\p_{lj}\phi_b^j\right).\no
\end{eqnarray}
Now substitute (\ref{wt 14}) and (\ref{wt 15}) into (\ref{wt 18}).
Then based on the oddness in $\vv$, there is no $\pp[f]$ contribution in the first
term, and no $a$ and $c$ contribution in the second term of (\ref{wt
18}). We can simplify (\ref{wt 18}) into
\begin{eqnarray}
\text{LHS}
&=&\e\int_{\p\Omega\times\r^2}\id_{\gamma_+}(1-\pp)[f]
\m^{\frac{1}{2}}(\vv)\left(\vs_i^2-\beta_b\right)\p_j\phi_b^j(\vv\cdot\vn)\no\\
&&+\e\int_{\p\Omega\times\r^2}\id_{\gamma_-}h
\m^{\frac{1}{2}}(\vv)\left(\vs_i^2-\beta_b\right)\p_j\phi_b^j(\vv\cdot\vn)\no\\
&&-\e\sum_{l=1}^2\int_{\Omega\times\r^2}\m(\vv)\vs_l^2\left(\vs_i^2-\beta_b\right)\p_{lj}\phi_b^jb_l\no\\
&&-\e\sum_{l=1}^2\int_{\Omega\times\r^2}(\ik-\pk)[f]\m^{\frac{1}{2}}(\vv)\left(\vs_i^2-\beta_b\right)\vs_l\p_{lj}\phi_b^j.\no
\end{eqnarray}
We will choose $\beta_b$ such that
\begin{eqnarray}
\int_{\r^2}\m(\vv)\left(\abs{\vs_i}^2-\beta_b\right)\ud{\vv}=0\ \
\text{for}\ \ i=1,2.
\end{eqnarray}
It is easy to check that this $\beta_b$ can always be achieved.
For such $\beta_b$ and any $i\neq l$, we can directly compute
\begin{eqnarray}
\int_{\r^2}\m(\vv)\left(\abs{\vs_i}^2-\beta_b\right)\vs_l^2\ud{\vv}&=&0,\\
\int_{\r^2}\m(\vv)\left(\abs{\vs_i}^2-\beta_b\right)\vs_i^2\ud{\vv}&=&C\neq0.
\end{eqnarray}
Then we deduce
\begin{eqnarray}
&&-\e\sum_{l=1}^2\int_{\Omega\times\r^2}\m(\vv)\vs_l^2\left(\vs_i^2-\beta_b\right)\p_{lj}\phi_b^jb_l\\
&=&-\e\int_{\Omega\times\r^2}\m(\vv)\vs_i^2\left(\vs_i^2-\beta_b\right)\p_{ij}\phi_b^jb_i-\e\sum_{l\neq
i}\int_{\Omega\times\r^2}\m(\vv)\vs_l^2\left(\vs_i^2-\beta_b\right)\p_{lj}\phi_b^jb_l\no\\
&=&C\int_{\Omega}(\p_{ij}\phi_b^j)b_i=C\int_{\Omega}\Big(\p_{ij}\dx^{-1}b_j\Big)b_i.\no
\end{eqnarray}
Hence, similar to (\ref{wt 16}), we may estimate
\begin{eqnarray}\label{wt 19}
\e\abs{\int_{\Omega}\Big(\p_{ij}\dx^{-1}b_j\Big)b_i}
&\leq&
C\tm{\vbb}\bigg(\e\tss{(1-\pp)[f]}{+}+\tm{(\ik-\pk)[f]}+\tm{S}+\e\tss{h}{-}\bigg).
\end{eqnarray}
\ \\
Step 2.2: Estimates of $\Big(\p_{jj}\dx^{-1}b_i\Big)b_i$ for $i\neq
j$.\\
We choose the test function
\begin{eqnarray}
\psi=\m^{\frac{1}{2}}(\vv)\abs{\vv}^2\vs_i\vs_j\p_j\phi_b^i\ \ \text{for}\ \ i\neq j.
\end{eqnarray}
The right-hand side (RHS) of (\ref{wt 11}) is still bounded by
\begin{eqnarray}
\text{RHS}\leq C\tm{\vbb}\bigg(\tm{(\ik-\pk)[f]}+\tm{S}\bigg).
\end{eqnarray}
Also, the left-hand side (LHS) of (\ref{wt 11}) takes the form
\begin{eqnarray}\label{wt 20}
\text{LHS}&=&\e\int_{\p\Omega\times\r^2}f\m^{\frac{1}{2}}(\vv)\abs{\vv}^2\vs_i\vs_j\p_j\phi_b^i(\vv\cdot\vn)\\
&&-\e\int_{\Omega\times\r^2}f\m^{\frac{1}{2}}(\vv)\abs{\vv}^2\vs_i\vs_j\left(\sum_{l=1}^2\vs_l\p_{lj}\phi_b^i\right).\no
\end{eqnarray}
Now substitute (\ref{wt 14}) and (\ref{wt 15}) into (\ref{wt 20}).
Then based on the oddness in $\vv$, there is no $\pp[f]$ contribution in
the first term, and no $a$ and $c$ contribution in the second term
of (\ref{wt 20}). We can simplify (\ref{wt 20}) into
\begin{eqnarray}
\text{LHS}
&=&\e\int_{\p\Omega\times\r^2}\id_{\gamma_+}(1-\pp)[f]
\m^{\frac{1}{2}}(\vv)\abs{\vv}^2\vs_i\vs_j\p_j\phi_b^i(\vv\cdot\vn)\\
&&+\e\int_{\p\Omega\times\r^2}\id_{\gamma_-}h
\m^{\frac{1}{2}}(\vv)\abs{\vv}^2\vs_i\vs_j\p_j\phi_b^i(\vv\cdot\vn)\no\\
&&-\e\int_{\Omega\times\r^2}\m(\vv)\abs{\vv}^2\vs_i^2\vs_j^2\Big(\p_{ij}\phi_b^ib_j+\p_{jj}\phi_b^ib_i\Big)\no\\
&&-\e\sum_{l=1}^2\int_{\Omega\times\r^2}(\ik-\pk)[f]\m^{\frac{1}{2}}(\vv)\abs{\vv}^2\vs_i\vs_j\vs_l\p_{lj}\phi_b^i.\no
\end{eqnarray}
Then we deduce
\begin{eqnarray}
&&-\e\int_{\Omega\times\r^2}\m(\vv)\abs{\vv}^2\vs_i^2\vs_j^2\Big(\p_{ij}\phi_b^ib_j+\p_{jj}\phi_b^ib_i\Big)=
C\bigg(\int_{\Omega}\Big(\p_{ij}\dx^{-1}b_i\Big)b_j+\int_{\Omega}\Big(\p_{jj}\dx^{-1}b_i\Big)b_i\bigg).
\end{eqnarray}
Hence, we may estimate that for $i\neq j$,
\begin{eqnarray}\label{wt 21}
\e\abs{\int_{\Omega}\Big(\p_{jj}\dx^{-1}b_i\Big)b_i}
&\leq&C\tm{\vbb}\bigg(\e\tss{(1-\pp)[f]}{+}+\tm{(\ik-\pk)[f]}+\tm{S}+\e\tss{h}{-}\bigg)\\
&&+C\e\abs{\int_{\Omega}\Big(\p_{ij}\dx^{-1}b_i\Big)b_j}.\no
\end{eqnarray}
Moreover, by (\ref{wt 19}), for $i=j=1,2$,
\begin{eqnarray}\label{wt 22}
\e\abs{\int_{\Omega}\Big(\p_{jj}\dx^{-1}b_j\Big)b_j}
&\leq&C\tm{\vbb}\bigg(\e\tss{(1-\pp)[f]}{+}+\tm{(\ik-\pk)[f]}+\tm{S}+\e\tss{h}{-}\bigg).
\end{eqnarray}
\ \\
Step 2.3: Synthesis.\\
Summarizing (\ref{wt 21}) and (\ref{wt 22}), we may sum up over
$j=1,2$ to obtain that for any $i=1,2$,
\begin{eqnarray}\label{wt 23}
\e\tm{b_i}^2
&\leq&C\tm{\vbb}\bigg(\e\tss{(1-\pp)[f]}{+}+\tm{(\ik-\pk)[f]}+\tm{S}+\e\tss{h}{-}\bigg).
\end{eqnarray}
which further implies
\begin{eqnarray}\label{wt 24}
\e\tm{\vbb}&\leq&C\bigg(\e\tss{(1-\pp)[f]}{+}+\tm{(\ik-\pk)[f]}+\tm{S}+\e\tss{h}{-}\bigg).
\end{eqnarray}
\ \\
Step 3: Estimates of $a$.\\
We choose the test function
\begin{eqnarray}\label{wt 25}
\psi=\psi_a=\m^{\frac{1}{2}}(\vv)\left(\abs{\vv}^2-\beta_a\right)\Big(\vv\cdot\nx\phi_a(\vx)\Big),
\end{eqnarray}
where
\begin{eqnarray}
\left\{
\begin{array}{rcl}
-\dx\phi_a&=&a(\vx)\ \ \text{in}\ \ \Omega,\\\rule{0ex}{1.5em}
\dfrac{\p\phi_a}{\p\vn}&=&0\ \ \text{on}\ \ \p\Omega,
\end{array}
\right.
\end{eqnarray}
and $\beta_a$ is a real number to be determined later. Based on the
standard elliptic estimates with the normalization condition
\begin{eqnarray}
\int_{\Omega}a(\vx)\ud{\vx}=\int_{\Omega\times\r^2}f(\vx,\vv)\m^{\frac{1}{2}}(\vv)\ud{\vv}\ud{\vx}=0,
\end{eqnarray}
we have
\begin{eqnarray}
\nm{\phi_a}_{H^2}\leq C \tm{a}.
\end{eqnarray}
With the choice of (\ref{wt 25}), the right-hand side (RHS) of
(\ref{wt 11}) is bounded by
\begin{eqnarray}
\text{RHS}\leq C \tm{a}\bigg(\tm{(\ik-\pk)[f]}+\tm{S}\bigg).
\end{eqnarray}
We have
\begin{eqnarray}
\vv\cdot\nx\psi_a&=&\m^{\frac{1}{2}}(\vv)\sum_{i,j=1}^2\left(\abs{\vv}^2-\beta_a\right)\vs_i\vs_j\p_{ij}\phi_a(\vx),
\end{eqnarray}
so the left-hand side (LHS) of (\ref{wt 11}) takes the form
\begin{eqnarray}\label{wt 26}
\text{LHS}&=&\e\int_{\p\Omega\times\r^2}f\m^{\frac{1}{2}}(\vv)\left(\abs{\vv}^2-\beta_a\right)\left(\sum_{i=1}^2\vs_i\p_i\phi_a\right)(\vv\cdot\vn)\\
&&-\e\int_{\Omega\times\r^2}f\m^{\frac{1}{2}}(\vv)\left(\abs{\vv}^2-\beta_a\right)\left(\sum_{i,j=1}^2\vs_i\vs_j\p_{ij}\phi_a\right).\no
\end{eqnarray}
We will choose $\beta_a$ such that
\begin{eqnarray}
\int_{\r^2}\m^{\frac{1}{2}}(\vv)\left(\abs{\vv}^2-\beta_a\right)\frac{\abs{\vv}^2-2}{2}\vs_i^2\ud{\vv}=0\
\ \text{for}\ \ i=1,2.
\end{eqnarray}
It is easy to check that this $\beta_a$ can always be achieved. Now substitute (\ref{wt 14})
and (\ref{wt 15}) into (\ref{wt 26}). Then based on this choice of
$\beta_a$ and the oddness in $\vv$, there is no $\vbb$ and $c$ contribution in the second
term of (\ref{wt 26}). Since the off-diagonal
$a$ contribution in the second term of (\ref{wt 26}) also vanishes due
to the oddness in $\vv$, we can simplify (\ref{wt 26}) into
\begin{eqnarray}
\text{LHS}
&=&\e\int_{\p\Omega\times\r^2}\pp[f]
\m^{\frac{1}{2}}(\vv)\left(\abs{\vv}^2-\beta_a\right)\left(\sum_{i=1}^2\vs_i\p_i\phi_a\right)(\vv\cdot\vn)\\
&&+\e\int_{\p\Omega\times\r^2}\id_{\gamma_+}(1-\pp)[f]
\m^{\frac{1}{2}}(\vv)\left(\abs{\vv}^2-\beta_a\right)\left(\sum_{i=1}^2\vs_i\p_i\phi_a\right)(\vv\cdot\vn)\no\\
&&+\e\int_{\p\Omega\times\r^2}\id_{\gamma_-}h
\m^{\frac{1}{2}}(\vv)\left(\abs{\vv}^2-\beta_a\right)\left(\sum_{i=1}^2\vs_i\p_i\phi_a\right)(\vv\cdot\vn)\no\\
&&-\sum_{i=1}^2\e\int_{\r^2}\m(\vv)\abs{\vs_i}^2\left(\abs{\vv}^2-\beta_a\right)\ud{\vv}
\int_{\Omega}a(\vx)\p_{ii}\phi_a(\vx)\ud{\vx}\no\\
&&-\e\int_{\Omega\times\r^2}(\ik-\pk)[f]\m^{\frac{1}{2}}(\vv)\left(\abs{\vv}^2-\beta_a\right)\left(\sum_{i,j=1}^2\vs_i\vs_j\p_{ij}\phi_a\right).\no
\end{eqnarray}
We make an orthogonal decomposition on the boundary
\begin{eqnarray}
\vv=(\vv\cdot\vn)\vn+\vv_{\bot}=\vs_{\n}\vn+\vv_{\bot},
\end{eqnarray}
where $\nu\perp\vv_{\bot}$. Then the contribution of
$\pp[f]=z_{\gamma}(\vx)\m^{\frac{1}{2}}(\vv)$ for a suitable function
$z_{\gamma}(\vx)$ is
\begin{eqnarray}\label{wt 27}
&&\e\int_{\p\Omega\times\r^2}\pp[f]
\m^{\frac{1}{2}}(\vv)\left(\abs{\vv}^2-\beta_a\right)\left(\sum_{i=1}^2\vs_i\p_i\phi_a\right)(\vv\cdot\vn)\\
&=&\e\int_{\p\Omega\times\r^2}
\m(\vv)\left(\abs{\vv}^2-\beta_a\right)v_n\Big(\vv\cdot\nx\phi_a\Big)z_{\gamma}(\vx)\no\\
&=&\e\int_{\p\Omega\times\r^2}
\m(\vv)\left(\abs{\vv}^2-\beta_a\right)v_n^2\frac{\p\phi_a}{\p\vn}z_{\gamma}(\vx)+\e\int_{\p\Omega\times\r^2}
\m(\vv)\left(\abs{\vv}^2-\beta_a\right)v_n\Big(\vv_{\bot}\cdot\nx\phi_a\Big)z_{\gamma}(\vx).\no
\end{eqnarray}
Based on the definition of $\phi_a$ and the oddness of
$v_n\vv_{\bot}$, we know the contribution of $\pp[f]$ in the
first term of (\ref{wt 26}) vanishes. Since
\begin{eqnarray}
\int_{\r^2}\m^{\frac{1}{2}}(\vv)\abs{\vs_i}^2\left(\abs{\vv}^2-\beta_a\right)\ud{\vv}=C,
\end{eqnarray}
we have
\begin{eqnarray}
-\e\int_{\Omega}\dx\phi_a(\vx)a(\vx)\ud{\vx}
&\leq&C
\tm{a}\bigg(\e\tss{(1-\pp)[f]}{+}+\tm{(\ik-\pk)[f]}+\tm{S}+\e\tss{h}{-}\bigg).
\end{eqnarray}
Since $-\dx\phi_a=a$, similar to (\ref{wt 16}), we know
\begin{eqnarray}\label{wt 28}
\e\tm{a} &\leq&C
\bigg(\e\tss{(1-\pp)[f]}{+}+\tm{(\ik-\pk)[f]}+\tm{S}+\e\tss{h}{-}\bigg).
\end{eqnarray}
\ \\
Step 4: Synthesis.\\
Collecting (\ref{wt 16}), (\ref{wt 24}) and (\ref{wt 28}), we deduce
\begin{eqnarray}
\e\tm{\pk[f]}
&\leq& C\bigg(\e\tss{(1-\pp)[f]}{+}+\tm{(\ik-\pk)[f]}+\tm{S}+\e\tss{h}{-}\bigg).
\end{eqnarray}
This completes our proof.
\end{proof}

\begin{theorem}\label{LT estimate}
The solution $f(\vx,\vw)$ to the equation (\ref{linear steady}) satisfies the estimate
\begin{eqnarray}
\frac{1}{\e^{\frac{1}{2}}}\tss{(1-\pp)[f]}{+}+\tm{f}&\leq&C\bigg(\frac{1}{\e^2}\tm{\pk[S]}+\frac{1}{\e}\tm{(\ik-\pk)[S]}+\frac{1}{\e}\tss{h}{-}\bigg).
\end{eqnarray}
\end{theorem}
\begin{proof}
We divide it into several steps:\\
\ \\
Step 1: Energy Estimate.\\
Multiplying $f$ on both sides of (\ref{linear steady}) and applying Green's identity imply
\begin{eqnarray}\label{wt 34}
\frac{\e}{2}\tss{f}{+}^2+\br{\ll[f],f}
&=&\frac{\e}{2}\tss{\pp[f]+h}{-}^2+\int_{\Omega\times\r^2}fS.
\end{eqnarray}
A direct computation shows that
\begin{eqnarray}
\tss{f}{+}^2-\tss{\pp[f]}{+}^2=\tss{(1-\pp)[f]}{+}^2.
\end{eqnarray}
Also, from the spectral gap of $\ll$, we know
\begin{eqnarray}
\br{\ll[f],f}\geq C\um{(\ik-\pk)[f]}^2.
\end{eqnarray}
Applying Cauchy's inequality implies that
\begin{eqnarray}
\tss{\pp[f]+h}{-}^2&=&\tss{\pp[f]}{-}^2+\tss{h}{-}^2+2\int_{\gamma_-}\pp[f]h\ud{\gamma}\\
&\leq&(1+\eta\e)\tss{\pp[f]}{+}^2+\left(1+\frac{1}{\eta\e}\right)\tss{h}{-}^2,\no
\end{eqnarray}
for some $\eta>0$. In total, we have
\begin{eqnarray}\label{wt 31}
\frac{\e}{2}\tss{(1-\pp)[f]}{+}^2+\um{(\ik-\pk)[f]}^2
&\leq&\eta\e^2\tss{\pp[f]}{}^2+\left(1+\frac{1}{\eta}\right)\tss{h}{-}^2+\int_{\Omega\times\r^2}fS.
\end{eqnarray}
\ \\
Step 2: Estimate of $\tss{\pp[f]}{}$.\\
Multiplying $f$ on both sides of the
equation (\ref{linear steady}), we have
\begin{eqnarray}\label{wt 05}
\e\vv\cdot\nx(f^2)=-2f\ll[f]+2fS.
\end{eqnarray}
Taking absolute value and integrating (\ref{wt 05}) over
$\Omega\times\r^2$, we deduce
\begin{eqnarray}
\nm{\vv\cdot\nx(f^2)}_{1}&\leq&\frac{1}{\e}\bigg(\tm{(\ik-\pk)[f]}^2+\int_{\Omega\times\r^2}fS\bigg).
\end{eqnarray}
On the other hand, by Lemma \ref{wellposedness prelim lemma 1}, for any
$\gamma\backslash\gamma^{\d}$ away from $\gamma_0$, we have
\begin{eqnarray}
\abs{\id_{\gamma\backslash\gamma^{\d}}f}_2^2
&\leq&
C(\d)\left(\tm{f}^2+\nm{\vv\cdot\nx(f^2)}_{1}\right).
\end{eqnarray}
Based on the definition, we can rewrite $\pp
f=z_{\gamma}(\vx)\m^{\frac{1}{2}}$ for a suitable function
$z_{\gamma}(\vx)$ and
for $\d$ small, we deduce
\begin{eqnarray}\label{wt 30}
\abs{\pp[\id_{\gamma\backslash\gamma^{\d}}f]}_2^2
&=&\int_{\p\Omega}\abs{z_{\gamma}(\vx)}^2
\left(\int_{\abs{\vv\cdot\vn(\vx)}\geq\d,\d\leq\abs{\vv}\leq\frac{1}{\d}}\m(\vv)\abs{\vv\cdot\vn(\vx)}\ud{\vv}\right)\ud{\vx}\\
&\geq&\half\left(\int_{\p\Omega}\abs{z_{\gamma}(\vx)}^2\ud{\vx}\right)\left(\int_{\r^2}\m(\vv)\abs{\vv\cdot\vn(\vx)}\ud{\vv}\right)\no\\
&=&\half\abs{\pp[f]}_2^2,\no
\end{eqnarray}
where we utilize the fact that
\begin{eqnarray}
\int_{\abs{\vv\cdot\vn(\vx)}\leq\d}\m(\vv)\abs{\vv\cdot\vn(\vx)}\ud{\vv}&\leq&C\d,\\
\int_{\abs{\vv}\leq\d\ or\
\abs{\vv}\geq\frac{1}{\d}}\m(\vv)\abs{\vv\cdot\vn(\vx)}\ud{\vv}&\leq&C\d.
\end{eqnarray}
Therefore, from
\begin{eqnarray}
\abs{\pp[\id_{\gamma\backslash\gamma^{\d}}f]}_2\leq C
\abs{\id_{\gamma\backslash\gamma^{\d}}f}_2,
\end{eqnarray}
we conclude
\begin{eqnarray}\label{wt 07}
\half\abs{\pp[f]}_2^2
&\leq&\abs{\pp[\id_{\gamma\backslash\gamma^{\d}}f]}_2^2\leq C\abs{\id_{\gamma\backslash\gamma^{\d}}f}_2^2\\
&\leq&C(\d)\left(\tm{f}^2+\nm{\vv\cdot\nx(f^2)}_{1}\right)\no\\
&\leq&C\left(\tm{f}^2+\frac{1}{\e}\tm{(\ik-\pk)[f]}^2+\frac{1}{\e}\int_{\Omega\times\r^2}fS\right).\no
\end{eqnarray}
Hence, we know
\begin{eqnarray}\label{wt 32}
\abs{\pp[f]}_2^2\leq C\left(\tm{f}^2+\frac{1}{\e}\tm{(\ik-\pk)[f]}^2+\frac{1}{\e}\int_{\Omega\times\r^2}fS\right),
\end{eqnarray}
which can be further simplified as
\begin{eqnarray}\label{wt 32=}
\abs{\pp[f]}_2^2\leq C\left(\tm{\pk[f]}^2+\frac{1}{\e}\tm{(\ik-\pk)[f]}^2+\frac{1}{\e}\int_{\Omega\times\r^2}fS\right).
\end{eqnarray}
\ \\
Step 3: Synthesis.\\
Plugging (\ref{wt 32=}) into (\ref{wt 31}) with $\e$ sufficiently
small to absorb $\tm{(\ik-\pk)[f]}^2$ into the left-hand side, we obtain
\begin{eqnarray}\label{wt 33}
\e\tss{(1-\pp)[f]}{+}^2+\um{(\ik-\pk)[f]}^2
&\leq&C\bigg(\eta\e^2\tm{\pk[f]}^2+\left(1+\frac{1}{\eta}\right)\tss{h}{-}^2+\int_{\Omega\times\r^2}fS\bigg).
\end{eqnarray}
We square on both sides of (\ref{wt 08}) to obtain
\begin{eqnarray}\label{wt 29}
\e^2\tm{\pk[f]}&\leq&C
\bigg(\e^2\tss{(1-\pp)[f]}{+}+\tm{(\ik-\pk)[f]}+\tm{S}+\e^2\tss{h}{-}\bigg).
\end{eqnarray}
Multiplying a small constant on both sides of (\ref{wt 29}) and
adding to (\ref{wt 33}) to absorb $\e^2\tss{(1-\pp)[f]}{+}$ and $\tm{(\ik-\pk)[f]}$ into the left-hand side, we obtain
\begin{eqnarray}
\\
\e\tss{(1-\pp)[f]}{+}^2+\e^2\tm{\pk[f]}^2+\um{(\ik-\pk)[f]}\leq C\bigg(\eta\e^2\tm{\pk[f]}^2+\left(1+\frac{1}{\eta}\right)\tss{h}{-}^2+\int_{\Omega\times\r^2}fS+\tm{S}\bigg).\no
\end{eqnarray}
Taking $\eta$ sufficiently small, we can further absorb $\eta\e^2\tm{\pk[f]}^2$ into the left-hand side to get
\begin{eqnarray}
\e\tss{(1-\pp)[f]}{+}^2+\e^2\tm{\pk[f]}^2+\um{(\ik-\pk)[f]}\leq C\bigg(\tss{h}{-}^2+\int_{\Omega\times\r^2}fS+\tm{S}\bigg).
\end{eqnarray}
Since
\begin{eqnarray}
\int_{\Omega\times\r^2}fS\leq
C\e^2\tm{f}^2+\frac{1}{4C\e^2}\tm{S}^2,
\end{eqnarray}
for $C$ sufficiently small, we have
\begin{eqnarray}
\e\tss{(1-\pp)[f]}{+}^2+\e^2\tm{\pk[f]}^2+\um{(\ik-\pk)[f]}\leq C\bigg(\tss{h}{-}^2+\frac{1}{\e^2}\tm{S}\bigg).
\end{eqnarray}
Hence, we deduce
\begin{eqnarray}
\frac{1}{\e^{\frac{1}{2}}}\tss{(1-\pp)[f]}{+}+\tm{f}\leq C\bigg(\frac{1}{\e^2}\tm{S}+\frac{1}{\e}\tss{h}{-}\bigg).
\end{eqnarray}
We may further improve the result based on the fact that
\begin{eqnarray}
\int_{\Omega\times\r^2}fS&=&\int_{\Omega\times\r^2}\bigg((\ik-\pk)[f]+\pk[f]\bigg)\bigg((\ik-\pk)[S]+\pk[S]\bigg)\\
&=&\int_{\Omega\times\r^2}\bigg((\ik-\pk)[f](\ik-\pk)[S]+\pk[f]\pk[S]\bigg)\no\\
&\leq&C\tm{(\ik-\pk)[f]}^2+\frac{1}{4C}\tm{(\ik-\pk)[S]}^2+C\e^2\tm{\pk[f]}^2+\frac{1}{4C\e^2}\tm{\pk[S]}^2.
\end{eqnarray}
Then the result is obvious.
\end{proof}

\subsection{$L^{\infty}$ Estimates - First Round}

We first define the tracking back through the characteristics and diffusive reflection.
\begin{definition}(Stochastic Cycle)
For a fixed point $(\vx,\vv)$ with $(\vx,\vv)\notin\gamma_0$, let
$(t_0,\vx_0,\vv_0)=(0,\vx,\vv)$. For $\vv_{k+1}$ such that
$\vv_{k+1}\cdot\vn(\vx_{k+1})>0$, define the $(k+1)$-component of
the back-time cycle as
\begin{eqnarray}
(t_{k+1},\vx_{k+1},\vv_{k+1})=(t_k+t_b(\vx_k,\vv_k),\vx_b(\vx_k,\vv_k),\vv_{k+1}),
\end{eqnarray}
where
\begin{eqnarray}
t_b(\vx,\vv)&=&\inf\{t>0:\vx-\e t\vv\notin\Omega\},\\
x_b(\vx,\vv)&=&\vx-\e t_b(\vx,\vv)\vv\notin\Omega.
\end{eqnarray}
Set
\begin{eqnarray}
\xc(s;\vx,\vv)&=&\sum_{k}\id_{\{t_{k}\leq s<t_{k+1}\}}\bigg(\vx_k-\e(t_k-s)\vv_k\bigg),\\
\vc(s;\vx,\vv)&=&\sum_{k}\id_{\{t_{k}\leq s<t_{k+1}\}}\vv_k.
\end{eqnarray}
\end{definition}
Define $\nn_{k}=\{\vv\in \r^2:\vv\cdot\vn(\vx_{k})>0\}$, and let
the iterated integral for $k\geq2$ be defined as
\begin{eqnarray}
\int_{\prod_{k=1}^{k-1}\nn_j}\prod_{j=1}^{k-1}\ud{\sigma_j}=\int_{\nn_1}\ldots\bigg(\int_{\nn_{k-1}}\ud{\sigma_{k-1}}\bigg)\ldots\ud{\sigma_1}
\end{eqnarray}
where $\ud{\sigma_j}=\m(\vv)\abs{\vv\cdot\vn(\vx_j)}\ud{\vv}$ is a
probability measure. We define a weight function scaled with
parameter $\xi$,
\begin{eqnarray}
\vh(\vv)=w_{\xi,\beta,\varrho}(\vv)=\left(1+\xi^2\abs{\vv}^2\right)^{\frac{\beta}{2}}\ue^{\varrho\abs{\vv}^2},
\end{eqnarray}
and
\begin{eqnarray}
\tvh(\vv)=\frac{1}{\m^{\frac{1}{2}}(\vv)\vh(\vv)}=\sqrt{2\pi}\frac{\ue^{\left(\frac{1}{4}-\varrho\right)\abs{\vv}^2}}
{\left(1+\xi^2\abs{\vv}^2\right)^{\frac{\beta}{2}}}.
\end{eqnarray}
\begin{lemma}\label{wellposedness prelim lemma 6}
For $T>0$ sufficiently large, there exists constants $C_1,C_2>0$
independent of $T_0$, such that for $k=C_1T_0^{\frac{5}{4}}$, and
$(\vx,\vv)\in\times\bar\Omega\times\r^2$,
\begin{eqnarray}\label{wt 38}
\int_{\Pi_{k=1}^{k-1}\nn_j}\id_{\{t_k(\vx,\vv,\vv_1,\ldots,\vv_{k-1})<\frac{T_0}{\e}\}}\prod_{j=1}^{k-1}\ud{\sigma_j}\leq
\left(\frac{1}{2}\right)^{C_2T_0^{\frac{5}{4}}}.
\end{eqnarray}
We also have, for $\beta>2$,
\begin{eqnarray}
\int_{\Pi_{j=1}^{k-1}\nn_j}\br{\vv_l}\tvh(\vv_l)\prod_{j=1}^{k-1}\ud{\sigma_j}&\leq&
\frac{C(\beta,\varrho)}{\xi^3},\label{wt 39}\\
\int_{\Pi_{j=1}^{k-1}\nn_j}\id_{\{t_k(\vx,\vv,\vv_1,\ldots,\vv_{k-1})<\frac{T_0}{\e}\}}\br{\vv_l}\tvh(\vv_l)\prod_{j=1}^{k-1}\ud{\sigma_j}&\leq&
\frac{C(\beta,\varrho)}{\xi^3}\bigg(\frac{1}{2}\bigg)^{C_2T_0^{\frac{5}{4}}}.
\label{wt 40}
\end{eqnarray}
\end{lemma}
\begin{proof}
See \cite[Lemma
4.1]{Esposito.Guo.Kim.Marra2013}.
\end{proof}
\begin{theorem}\label{LI estimate'}
The
solution $f(\vx,\vw)$ to the equation (\ref{linear steady}) satisfies the estimate for $\vth\geq3$ and $0\leq\varrho<\dfrac{1}{4}$,
\begin{eqnarray}
\\
\im{\bv f}
&\leq&C\bigg(\frac{1}{\e^3}\tm{\pk[S]}+\frac{1}{\e^2}\tm{(\ik-\pk)[S]}+\frac{1}{\e^{2}}\tss{h}{-}+\im{\frac{\bv S}{\nu}}+\iss{\bv h}{-}\bigg).\no
\end{eqnarray}
\end{theorem}
\begin{proof}
We divide the proof into several steps:\\
\ \\
Step 1: Mild formulation.\\
Denote
\begin{eqnarray}
g(\vx,\vv)&=&\vh(\vv) f(\vx,\vv),\\
K_{\vh(\vv)}[g](\vx,\vv)&=&\vh(\vv)
K\left[\frac{g}{\vh}\right](\vx,\vv)=\int_{\r^2}k_{\vh(\vv)}(\vv,\vuu)g(\vx,\vuu)\ud{\vuu},
\end{eqnarray}
where
\begin{eqnarray}
k_{\vh(\vv)}(\vv,\vuu)=k(\vv,\vuu)\frac{\vh(\vv)}{\vh(\vuu)}.
\end{eqnarray}
We can rewrite the solution of the equation (\ref{linear steady})
along the characteristics by Duhamel's principle as
\begin{eqnarray}
g(\vx,\vv)&=& \vh(\vv)h(\vx-\e t_1\vv,\vv)\ue^{-\nu(\vv)
t_{1}}+\int_{0}^{t_{1}}\vh(\vv) S\Big(\vx-\e(t_1-s)\vv,\vv\Big)\ue^{-\nu(\vv)
(t_{1}-s)}\ud{s}\\
&&+\int_{0}^{t_{1}}K_{\vh(\vv)}[g]\Big(\vx-\e(t_1-s)\vv,\vv\Big)\ue^{-\nu(\vv)
(t_{1}-s)}\ud{s}+\frac{\ue^{-\nu(\vv)
t_{1}}}{\tvh(\vv)}\int_{\nn_1}g(\vx_1,\vv_1)\tvh(\vv_1)\ud{\sigma_1},\no
\end{eqnarray}
where the last term refers to $\pp[f]$. We may further rewrite the equation (\ref{linear steady}) along the
stochastic cycle by applying Duhamel's principle $k$ times as
\begin{eqnarray}\label{wt 36}
g(\vx,\vv)
&=& \vh(\vv)h(\vx-\e t_1\vv,\vv)\ue^{-\nu(\vv)
t_{1}}+\int_{0}^{t_{1}}\vh(\vv) S\Big(\vx-\e(t_1-s)\vv,\vv\Big)\ue^{-\nu(\vv)
(t_{1}-s)}\ud{s}\\
&&+\int_{0}^{t_{1}}K_{\vh(\vv)}[g]\Big(\vx-\e(t_1-s)\vv,\vv\Big)\ue^{-\nu(\vv)
(t_{1}-s)}\ud{s}\no\\
&&+\frac{1}{\tvh(\vv)}\sum_{l=1}^{k-1}\int_{\prod_{j=1}^{l}\nn_j}\Big(G[\vx,\vw]+H[\vx,\vw]\Big)\tvh(\vv_l)\bigg(\prod_{j=1}^{l}\ue^{-\nu(\vv_j)(t_{j+1}-t_j)}\ud{\sigma_j}\bigg)\no\\
&&+\frac{1}{\tvh(\vv)}\int_{\prod_{j=1}^{k}\nn_j}g(\vx_k,\vv_k)\tvh(\vv_{k})
\bigg(\prod_{j=1}^{k}\ue^{-\nu(\vv_{j})(t_{j+1}-t_{j})}\ud{\sigma_j}\bigg),\no
\end{eqnarray}
where
\begin{eqnarray}\label{wt 37}
G[\vx,\vw]&=&h(\vx_{l}-\e t_{l+1}\vv_l,\vv_{l})\vh(\vv_l)+\int_{t_l}^{t_{l+1}}\bigg(
S\Big(\vx_{l}-\e(t_{l+1}-s)\vv_l,\vv_{l}\Big)\vh(\vv_l)\ue^{\nu(\vv_l)s}\bigg)\ud{s}\\
H[\vx,\vw]&=&\int_{t_l}^{t_{l+1}}\bigg(
K_{\vh(\vv_l)}[g]\Big(\vx_{l}-\e(t_{l+1}-s)\vv_l,\vv_{l}\Big)\ue^{\nu(\vv_l)s}\bigg)\ud{s}.
\end{eqnarray}
\ \\
Step 2: Estimates of source terms and boundary terms.\\
We set $k=CT_0^{\frac{5}{4}}$ and take absolute value on both sides of
(\ref{wt 36}). Then all the terms in (\ref{wt
36}) related to the source term $S$ and boundary term $h$ can be bounded as
\begin{eqnarray}\label{wt 51}
\text{Part 1}\leq C\bigg(\iss{\vh h}{-}+\im{\frac{\vh S}{\nu}}\bigg)
\end{eqnarray}
due to Lemma \ref{wellposedness prelim lemma 6} and
\begin{eqnarray}
\frac{1}{\tvh}\leq C(\beta,\varrho)\xi^{\beta}.
\end{eqnarray}
The last term in (\ref{wt 36}) can be decomposed as follows:
\begin{eqnarray}
\text{Part 2}
&=&\frac{1}{\tvh(\vv)}\int_{\Pi_{j=1}^{k}\nn_j}\id_{\left\{t_{k+1}\leq
\frac{T_0}{\e}\right\}}g(\vx_k,\vv_k)\tvh(\vv_{k})
\bigg(\prod_{j=1}^{k}\ue^{-\nu(\vv_{j})(t_{j+1}-t_{j})}\ud{\sigma_j}\bigg)\\
&&+\frac{1}{\tvh(\vv)}\int_{\Pi_{j=1}^{k}\nn_j}\id_{\left\{t_{k+1}\geq
\frac{T_0}{\e}\right\}}g(\vx_k,\vv_k)\tvh(\vv_{k})
\bigg(\prod_{j=1}^{k}\ue^{-\nu(\vv_{j})(t_{j+1}-t_{j})}\ud{\sigma_j}\bigg).\no
\end{eqnarray}
Based on Lemma \ref{wellposedness prelim lemma 6}, we have
\begin{eqnarray}
&&\frac{1}{\tvh(\vv)}\int_{\Pi_{j=1}^{k}\nn_j}\id_{\left\{t_{k+1}\leq
\frac{T_0}{\e}\right\}}g(\vx_k,\vv_k)\tvh(\vv_{k})
\bigg(\prod_{j=1}^{k}\ue^{-\nu(\vv_{j})(t_{j+1}-t_{j})}\ud{\sigma_j}\bigg)\\
&\leq&C\bigg(\frac{1}{2}\bigg)^{C_2T_0^{\frac{5}{4}}}\im{g}.\no
\end{eqnarray}
Based on Lemma \ref{wellposedness prelim lemma 6} and
$\nu_0(1+\abs{\vv})\leq \nu(\vv)\leq\nu_1(1+\abs{\vv})$, we obtain
\begin{eqnarray}
&&\frac{1}{\tvh(\vv)}\int_{\Pi_{j=1}^{k}\nn_j}\id_{\left\{t_{k+1}\geq
\frac{T_0}{\e}\right\}}g(\vx_k,\vv_k)\tvh(\vv_{k})
\bigg(\prod_{j=1}^{k}\ue^{-\nu(\vv_{j})(t_{j+1}-t_{j})}\ud{\sigma_j}\bigg)\\
&\leq&\ue^{-\frac{\nu_0T_0}{\e}}\im{g}.\no
\end{eqnarray}
For $T_0$ sufficiently large and $\e$ sufficiently small, we get
\begin{eqnarray}
&&\text{Part 2}\leq
\d \im{g}.
\end{eqnarray}
for $\d$ arbitrarily small.\\
\ \\
Step 3: Estimates of $K_{\vh}$ terms.\\
So far, the only remaining terms are related to $K_{\vh}$.
Define the back-time stochastic cycle from $(s,\xc(s;\vx,\vv),\vv')$
as $(t_i',\vx_i',\vv_i')$. Then we can rewrite $K_{\vh}$ along the
stochastic cycle as
\begin{eqnarray}\label{wt 43}
&&K_{\vh(\vv)}[g]\Big(\vx-\e(t_1-s)\vv,\vv\Big)=K_{\vh(\vv)}[g](\xc,\vv)
=\int_{\r^2}k_{\vh(\vv)}(\vv,\vv')g(\xc,\vv')\ud{\vv'}\\
&\leq&\abs{\int_{\r^2}\int_{0}^{t_1'}k_{\vh(\vv)}(\vv,\vv')K_{\vh(\vv')}[g]\Big(\xc-\e(t_1'-r)\vv',\vv'\Big)\ue^{-\nu(\vv')
(t_1'-r)}\ud{r}\ud{\vv'}}\no\\
&&+\abs{\int_{\r^2}\frac{1}{\tvh(\vv')}\sum_{l=1}^{k-1}\int_{\prod_{j=1}^{l}\nn_j'}k_{\vh(\vv)}(\vv,\vv')H[\xc,\vv']\tvh(\vv_l')
\bigg(\prod_{j=1}^{l}\ue^{-\nu(\vv_j')(t_{j+1}'-t_j')}\ud{\sigma_j'}\bigg)\ud{\vv'}}\no\\
&&+\abs{\int_{\r^2}k_{\vh(\vv)}(\vv,\vv')A\ud{\vv'}}\no\\
&=&I+II+III.\no
\end{eqnarray}
$III$ can be directly estimated as
\begin{eqnarray}
III\leq C\bigg(\iss{\vh h}{-}+\im{\vh S}+\d\im{g}\bigg).
\end{eqnarray}
We may further rewrite $I$ as
\begin{eqnarray}
I&=&\abs{\int_{\r^2}\int_{\r^2}\int_{0}^{t_1'}k_{\vh(\vv)}(\vv,\vv')k_{\vh(\vv')}(\vv',\vv'')g(\xc-\e(t_1'-r)\vv',\vv'')\ue^{-\nu(\vv')
(t_1'-r)}\ud{r}\ud{\vv'}\ud{\vv''}},
\end{eqnarray}
which will estimated in four cases:
\begin{eqnarray}
I=I_1+I_2+I_3+I_4.
\end{eqnarray}
\ \\
Case I: $\abs{\vv}\geq N$.\\
Based on Lemma \ref{wellposedness prelim lemma 8}, we
have
\begin{eqnarray}
\abs{\int_{\r^2}\int_{\r^2}k_{\vh(\vv)}(\vv,\vv')k_{\vh(\vv')}(\vv',\vv'')\ud{\vv'}\ud{\vv''}}\leq\frac{C}{1+\abs{\vv}}\leq\frac{C}{N}.
\end{eqnarray}
Hence, we get
\begin{eqnarray}
I_1\leq\frac{C}{N}\im{g}.
\end{eqnarray}
\ \\
Case II: $\abs{\vv}\leq N$, $\abs{\vv'}\geq2N$, or $\abs{\vv'}\leq
2N$, $\abs{\vv''}\geq3N$.\\
Notice this implies either $\abs{\vv'-\vv}\geq N$ or
$\abs{\vv'-\vv''}\geq N$. Hence, either of the following is valid
correspondingly:
\begin{eqnarray}
\abs{k_{\vh(\vv)}(\vv,\vv')}&\leq& C\ue^{-\d N^2}\abs{k_{\vh(\vv)}(\vv,\vv')\ue^{\d\abs{\vv-\vv'}^2}},\\
\abs{k_{\vh(\vv')}(\vv',\vv'')}&\leq& C\ue^{-\d N^2}\abs{k_{\vh(\vv')}(\vv',\vv'')\ue^{\d\abs{\vv'-\vv''}^2}}.
\end{eqnarray}
Then based on Lemma \ref{wellposedness prelim lemma 8},
\begin{eqnarray}
\int_{\r^2}\abs{k_{\vh(\vv)}(\vv,\vv')\ue^{\d\abs{\vv-\vv'}^2}}\ud{\vv'}&<&\infty,\\
\int_{\r^2}\abs{k_{\vh(\vv')}(\vv',\vv'')\ue^{\d\abs{\vv'-\vv''}^2}}\ud{\vv''}&<&\infty.
\end{eqnarray}
Hence, we have
\begin{eqnarray}
I_2\leq C\ue^{-\d N^2}\im{g}.
\end{eqnarray}
\ \\
Case III: $t_1'-r\leq\d$ and $\abs{\vv}\leq N$, $\abs{\vv'}\leq 2N$, $\abs{\vv''}\leq 3N$.\\
In this case, since the integral with respect to $r$ is restricted in a very short
interval, there is a small contribution as
\begin{eqnarray}
I_3\leq\abs{\int_{\r^2}\int_{\r^2}\int_{t_1'-\d}^{t_1'}k_{\vh(\vv)}(\vv,\vv')k_{\vh(\vv')}(\vv',\vv'')\ue^{-\nu(\vv')
(t_1'-r)}\ud{r}\ud{\vv'}\ud{\vv''}}\im{g}\leq C\d\im{g}.
\end{eqnarray}
\ \\
Case IV: $t_1'-r\geq\d$ and $\abs{\vv}\leq N$, $\abs{\vv'}\leq 2N$, $\abs{\vv''}\leq 3N$.\\
Since $k_{\vh(\vv)}(\vv,\vv')$ has
possible integrable singularity of $\dfrac{1}{\abs{\vv-\vv'}}$, we can
introduce the truncated kernel $k_N(\vv,\vv')$ which is smooth and has compact support such that
\begin{eqnarray}\label{wt 42}
\sup_{\abs{p}\leq 3N}\int_{\abs{\vv'}\leq
3N}\abs{k_N(p,\vv')-k_{\vh(p)}(p,\vv')}\ud{\vv'}\leq\frac{1}{N}.
\end{eqnarray}
Then we can split
\begin{eqnarray}\label{wt 41}
k_{\vh(\vv)}(\vv,\vv')k_{\vh(\vv')}(\vv',\vv'')
&=&k_N(\vv,\vv')k_N(\vv',\vv'')\\
&&+\bigg(k_{\vh(\vv)}(\vv,\vv')-k_N(\vv,\vv')\bigg)k_{\vh(\vv')}(\vv',\vv'')\no\\
&&+\bigg(k_{\vh(\vv')}(\vv',\vv'')-k_N(\vv',\vv'')\bigg)k_N(\vv,\vv').\no
\end{eqnarray}
This means we further split $I_4$ into
\begin{eqnarray}
I_4=I_{4,1}+I_{4,2}+I_{4,3}.
\end{eqnarray}
Based on the estimate (\ref{wt 42}), we have
\begin{eqnarray}
I_{4,2}&\leq&\frac{C}{N}\im{g},\\
I_{4,3}&\leq&\frac{C}{N}\im{g}.
\end{eqnarray}
Therefore, the only remaining term is $I_{4,1}$. Note we always have
$\xc-\e(t_1'-r)\vv'\in\Omega$. Hence, we define the change of
variable $\vv'\rt\vec y$ as $\vec y=(y_1,y_2)=\xc-\e(t_1'-r)\vv'$ such that
\begin{eqnarray}
\abs{\frac{\ud{\vec y}}{\ud{\vv'}}}=\abs{\left\vert\begin{array}{cc}
\e(t_1'-r)&0\\
0&\e(t_1'-r)
\end{array}\right\vert}=\e^2(t_1'-r)^2\geq \e^2\d^2.
\end{eqnarray}
Since $k_N$ is bounded and $\abs{\vv'},\abs{\vv''}\leq 3N$, we estimate
\begin{eqnarray}
I_{4,1}
&\leq&
C\abs{\int_{\abs{\vv'}\leq2N}\int_{\abs{\vv''}\leq3N}\int_{0}^{t_1'}
\id_{\{\xc-\e(t_1'-r)\vv'\in\Omega\}}g(\xc-\e(t_1'-r)\vv',\vv'')\ue^{-\nu(\vv')
(t_1'-r)}\ud{r}\ud{\vv'}\ud{\vv''}}\\
&\leq&
C\abs{\int_{\abs{\vv'}\leq2N}\int_{\abs{\vv''}\leq3N}\int_{0}^{t_1'}
\id_{\{\xc-\e(t_1'-r)\vv'\in\Omega\}}f(\xc-\e(t_1'-r)\vv',\vv'')\ue^{-\nu(\vv')
(t_1'-r)}\ud{r}\ud{\vv'}\ud{\vv''}}\no\\
&\leq&C\bigg(\int_{\abs{\vv'}\leq2N}\int_{\abs{\vv''}\leq3N}\int_{0}^{t_1'}
\id_{\{\xc-\e(t_1'-r)\vv'\in\Omega\}}\ue^{-\nu(\vv')
(t_1'-r)}\ud{r}\ud{\vv'}\ud{\vv''}\bigg)^{\frac{1}{2}}\no\\
&&\bigg(\int_{\abs{\vv'}\leq2N}\int_{\abs{\vv''}\leq3N}\int_{0}^{t_1'}
\id_{\{\xc-\e(t_1'-r)\vv'\in\Omega\}}f^2(\xc-\e(t_1'-r)\vv',\vv'')\ue^{-\nu(\vv')
(t_1'-r)}\ud{r}\ud{\vv'}\ud{\vv''}\bigg)^{\frac{1}{2}}\no\\
&\leq&C\abs{\int_{0}^{t_1'}\frac{1}{\e^2\d^2}\int_{\abs{\vv''}\leq3N}\int_{\Omega}\id_{\{\vec
y\in\Omega\}}f^2(\vec y,\vv'')\ue^{-\nu(\vv')
(t_1'-r)}\ud{\vec y}\ud{\vv''}\ud{r}}^{\frac{1}{2}}\no\\
&\leq&\frac{C}{\e\d}\abs{\int_{0}^{t_1'}\int_{\abs{\vv''}\leq3N}\int_{\Omega}\id_{\{\vec
y\in\Omega\}}f^2(\vec y,\vv'')\ue^{-\nu(\vv')
(t_1'-r)}\ud{\vec y}\ud{\vv''}\ud{r}}^{\frac{1}{2}}\no\\
&=&\frac{C}{\e\d}\tm{f}.\no
\end{eqnarray}
Summarize all above in Case IV, we obtain
\begin{eqnarray}
I_4\leq \frac{C}{N}\im{g}+\frac{C}{\e\d}\tm{f}.
\end{eqnarray}
Therefore, we already prove
\begin{eqnarray}
I\leq
\bigg(C\ue^{-\d N^2}+\frac{C}{N}+\d\bigg)\im{g}+\frac{C}{\e\d}\tm{f}
\end{eqnarray}
Choosing $\d$ sufficiently small and then taking $N$ sufficiently
large, we have
\begin{eqnarray}
I\leq C\d\im{g}+\frac{C}{\e\d}\tm{f}.
\end{eqnarray}
A similar technique can justify
\begin{eqnarray}
II\leq C\d\im{g}+\frac{C}{\e\d}\tm{f}.
\end{eqnarray}
All the terms related to $K_{\vh}$ can be estimated in a similar fashion.\\
\ \\
Step 4: Synthesis.\\
Collecting all above, based on the mild formulation (\ref{wt 36}) we
have shown
\begin{eqnarray}
\im{g}\leq C\d\im{g}+\frac{C}{\e\d}\tm{f}+C\im{\frac{\vh S}{\nu}}+C\iss{\vh h}{-}.
\end{eqnarray}
Taking $\d$ is sufficiently small, based on Theorem \ref{LT estimate}, we obtain
\begin{eqnarray}
\im{g}&\leq&C\bigg(\frac{1}{\e}\tm{f}+\im{\vh S}+\iss{\vh
h}{-}\bigg)\\
&\leq&C\bigg( \frac{1}{\e^3}\tm{\pk[S]}+\frac{1}{\e^2}\tm{(\ik-\pk)[S]}+\frac{1}{\e^{2}}\tss{h}{-}+\im{\frac{\vh S}{\nu}}+\iss{\vh h}{-}\bigg).\no
\end{eqnarray}
It is easy to see for $\vth=\beta$, we have
\begin{eqnarray}
C_1\bv\leq\vh\leq C_2\bv,
\end{eqnarray}
for some constant $C_1,C_2>0$.
Then we must have
\begin{eqnarray}
\\
\im{\bv f}
&\leq&C\bigg(\frac{1}{\e^3}\tm{\pk[S]}+\frac{1}{\e^2}\tm{(\ik-\pk)[S]}+\frac{1}{\e^{2}}\tss{h}{-}+\im{\frac{\bv S}{\nu}}+\iss{\bv h}{-}\bigg).\no
\end{eqnarray}
\end{proof}

\subsection{$L^{2m}$ Estimates}

Here we consider $m\geq\mathbb{N}$ and $m>2$. Let $o(1)$ denote a sufficiently small constant. Since this is very similar to the $L^2$ estimates, we will focus on the key differences.
\begin{lemma}\label{wellposedness prelim lemma 4}
The
solution $f(\vx,\vw)$ to the equation (\ref{linear steady}) satisfies the estimate
\begin{eqnarray}\label{wt 08'}
\e\nm{\pk[f]}_{L^{2m}}&\leq& C\bigg(\e\abs{(1-\pp)[f]}_{L^m_+}+\tm{(\ik-\pk)[f]}+\e\nm{(\ik-\pk)[f]}_{L^{2m}}+\tm{S}+\e\abs{h}_{L^{m}_-}\bigg).
\end{eqnarray}
\end{lemma}
\begin{proof}
Applying Green's identity in Lemma \ref{wellposedness prelim lemma
2} to the solution of the equation (\ref{linear steady}). Then for any $\psi\in L^2(\Omega\times\r^2)$
satisfying $\vv\cdot\nx\psi\in L^2(\Omega\times\r^2)$ and $\psi\in
L^2(\gamma)$, we have
\begin{eqnarray}\label{wt 11'}
&&\e\int_{\gamma_+}f\psi\ud{\gamma}-\e\int_{\gamma_-}f\psi\ud{\gamma}
-\e\int_{\Omega\times\r^2}(\vv\cdot\nx\psi)f=-\int_{\Omega\times\r^2}\psi\ll\Big[(\ik-\pk)[f]\Big]+\int_{\Omega\times\r^2}S\psi.
\end{eqnarray}
Since
\begin{eqnarray}
\pk[f]=\m^{\frac{1}{2}}\bigg(a+\vv\cdot\vbb+\frac{\abs{\vv}^2-2}{2}c\bigg),
\end{eqnarray}
our goal is to choose a particular test function $\psi$ to estimate
$a$, $\vbb$ and $c$.\\
\ \\
Step 1: Estimates of $c$.\\
We choose the test function
\begin{eqnarray}\label{wt 12'}
\psi=\psi_c=\m^{\frac{1}{2}}(\vv)\left(\abs{\vv}^2-\beta_c\right)\Big(\vv\cdot\nx\phi_c(\vx)\Big),
\end{eqnarray}
where
\begin{eqnarray}
\left\{
\begin{array}{rcl}
-\dx\phi_c&=&c^{2m-1}(\vx)\ \ \text{in}\ \
\Omega,\\
\phi_c&=&0\ \ \text{on}\ \ \p\Omega,
\end{array}
\right.
\end{eqnarray}
and $\beta_c$ is a real number to be determined later. Based on the
standard elliptic estimates, we have
\begin{eqnarray}
\nm{\phi_c}_{W^{2,\frac{2m}{2m-1}}}\leq C\nm{c^{2m-1}}_{L^{\frac{2m}{2m-1}}}\leq C\nm{c}_{L^{2m}}^{2m-1}.
\end{eqnarray}
Also, we know
\begin{eqnarray}
\nm{\psi_c}_{L^2}&\leq& C\nm{\phi_c}_{H^1}\leq C\nm{\phi_c}_{W^{2,\frac{2m}{2m-1}}}\leq C\nm{c}_{L^{2m}}^{2m-1},\\
\nm{\psi_c}_{L^\frac{2m}{2m-1}}&\leq&C\nm{\phi_c}_{W^{1,\frac{2m}{2m-1}}}\leq C\nm{c}_{L^{2m}}^{2m-1}.
\end{eqnarray}
With the choice of (\ref{wt 12'}), the right-hand side (RHS) of
(\ref{wt 11'}) is bounded by
\begin{eqnarray}
\text{RHS}\leq C\nm{c}_{L^{2m}}^{2m-1}\bigg(\tm{(\ik-\pk)[f]}+\tm{S}\bigg).
\end{eqnarray}
We will choose $\beta_c$ such that
\begin{eqnarray}
\int_{\r^2}\m^{\frac{1}{2}}(\vv)\left(\abs{\vv}^2-\beta_c\right)\vs_i^2\ud{\vv}=0\
\ \text{for}\ \ i=1,2.
\end{eqnarray}
The left-hand side (LHS) of (\ref{wt 11'}) takes the form
\begin{eqnarray}
\text{LHS}
&=&\e\int_{\p\Omega\times\r^2}\id_{\gamma_+}(1-\pp)[f]
\m^{\frac{1}{2}}(\vv)\left(\abs{\vv}^2-\beta_c\right)\left(\sum_{i=1}^2\vs_i\p_i\phi_c\right)(\vv\cdot\vn)\\
&&+\e\int_{\p\Omega\times\r^2}\id_{\gamma_-}h
\m^{\frac{1}{2}}(\vv)\left(\abs{\vv}^2-\beta_c\right)\left(\sum_{i=1}^2\vs_i\p_i\phi_c\right)(\vv\cdot\vn)\no\\
&&-\e\sum_{i=1}^2\int_{\r^2}\m(\vv)\abs{\vs_i}^2\left(\abs{\vv}^2-\beta_c\right)\frac{\abs{\vv}^2-2}{2}\ud{\vv}
\int_{\Omega}c(\vx)\p_{ii}\phi_c(\vx)\ud{\vx}\no\\
&&-\e\int_{\Omega\times\r^2}(\ik-\pk)[f]\m^{\frac{1}{2}}(\vv)\left(\abs{\vv}^2-\beta_c\right)\left(\sum_{i,j=1}^2\vs_i\vs_j\p_{ij}\phi_c\right).\no
\end{eqnarray}
Since
\begin{eqnarray}
\int_{\r^2}\m(\vv)\abs{\vs_i}^2\left(\abs{\vv}^2-\beta_c\right)\frac{\abs{\vv}^2-2}{2}\ud{\vv}=C,
\end{eqnarray}
we have
\begin{eqnarray}
\e\abs{\int_{\Omega}\dx\phi_c(\vx)c(\vx)\ud{\vx}}
&\leq&C\nm{c}_{L^{2m}}^{2m-1}\bigg(\e\abs{(1-\pp)[f]}_{L^m_+}+\tm{(\ik-\pk)[f]}+\e\nm{(\ik-\pk)[f]}_{L^{2m}}\\
&&+\tm{S}+\e\abs{h}_{L^{m}_-}\bigg),\no
\end{eqnarray}
where we have used the elliptic estimates, Sobolev embedding theorem, and the trace estimate:
\begin{eqnarray}
\abs{\nx\phi_c}_{L^{\frac{m}{m-1}}}\leq C\abs{\nx\phi_c}_{W^{\frac{1}{2m},\frac{m}{m-1}}}\leq C\nm{\nx\phi_c}_{W^{1,\frac{m}{m-1}}}\leq C\nm{\phi_c}_{W^{2,\frac{2m}{2m-1}}}\leq C\nm{c}_{L^{2m}}^{2m-1}.
\end{eqnarray}
Since $-\dx\phi_c=c^{2m-1}$,
we know
\begin{eqnarray}
\\
\e\nm{c}_{L^{2m}}^{2m}
&\leq&C
\nm{c}_{L^{2m}}^{2m-1}\bigg(\e\abs{(1-\pp)[f]}_{L^m_+}+\tm{(\ik-\pk)[f]}+\e\nm{(\ik-\pk)[f]}_{L^{2m}}+\tm{S}+\e\abs{h}_{L^{m}_-}\bigg),\no
\end{eqnarray}
which further implies
\begin{eqnarray}\label{wt 16'}
&&\e\nm{c}_{L^{2m}}\leq C
\bigg(\e\abs{(1-\pp)[f]}_{L^m_+}+\tm{(\ik-\pk)[f]}+\e\nm{(\ik-\pk)[f]}_{L^{2m}}+\tm{S}+\e\abs{h}_{L^{m}_-}\bigg).
\end{eqnarray}
\ \\
Step 2: Estimates of $\vbb$.\\
We further divide this step into several sub-steps:\\
\ \\
Step 2.1: Estimates of $\Big(\p_{ij}\dx^{-1}b_j\Big)b_i$ for $i,j=1,2$.\\
We choose the test function
\begin{eqnarray}\label{wt 17'}
\psi=\psi_b^{i,j}=\m^{\frac{1}{2}}(\vv)\left(\vs_i^2-\beta_b\right)\p_j\phi_b^j,
\end{eqnarray}
where
\begin{eqnarray}
\left\{
\begin{array}{rcl}
-\dx\phi_b^j&=&b_j^{2m-1}(\vx)\ \ \text{in}\ \
\Omega,\\
\phi_b^j&=&0\ \ \text{on}\ \ \p\Omega,
\end{array}
\right.
\end{eqnarray}
and $\beta_b$ is a real number to be determined later. Based on the
standard elliptic estimates, we have
\begin{eqnarray}
\nm{\phi_b^{i,j}}_{W^{2,\frac{2m}{2m-1}}}\leq C\nm{b_j^{2m-1}}_{L^{\frac{2m}{2m-1}}}\leq C\nm{b_j}_{L^{2m}}^{2m-1}.
\end{eqnarray}
Also, we know
\begin{eqnarray}
\nm{\psi_b^{i,j}}_{L^2}&\leq& C\nm{\phi_b^{i,j}}_{H^1}\leq C\nm{\phi_b^{i,j}}_{W^{2,\frac{2m}{2m-1}}}\leq C\nm{b_j}_{L^{2m}}^{2m-1},\\
\nm{\psi_b^{i,j}}_{L^\frac{2m}{2m-1}}&\leq&C\nm{\phi_b^{i,j}}_{W^{1,\frac{2m}{2m-1}}}\leq C\nm{b_j}_{L^{2m}}^{2m-1}.
\end{eqnarray}
With the choice of (\ref{wt 17'}), the right-hand side (RHS) of
(\ref{wt 11'}) is bounded by
\begin{eqnarray}
\text{RHS}\leq C \nm{\vbb}_{L^{2m}}^{2m-1}\bigg(\tm{(\ik-\pk)[f]}+\tm{S}\bigg).
\end{eqnarray}
We will choose $\beta_b$ such that
\begin{eqnarray}
\int_{\r^2}\m(\vv)\left(\abs{\vs_i}^2-\beta_b\right)\ud{\vv}=0\ \
\text{for}\ \ i=1,2.
\end{eqnarray}
Hence, the left-hand side (LHS) of (\ref{wt 11'}) takes the form
\begin{eqnarray}
\text{LHS}
&=&\e\int_{\p\Omega\times\r^2}\id_{\gamma_+}(1-\pp)[f]
\m^{\frac{1}{2}}(\vv)\left(\vs_i^2-\beta_b\right)\p_j\phi_b^j(\vv\cdot\vn)\no\\
&&+\e\int_{\p\Omega\times\r^2}\id_{\gamma_-}h
\m^{\frac{1}{2}}(\vv)\left(\vs_i^2-\beta_b\right)\p_j\phi_b^j(\vv\cdot\vn)\no\\
&&-\e\sum_{l=1}^2\int_{\Omega\times\r^2}\m(\vv)\vs_l^2\left(\vs_i^2-\beta_b\right)\p_{lj}\phi_b^jb_l\no\\
&&-\e\sum_{l=1}^2\int_{\Omega\times\r^2}(\ik-\pk)[f]\m^{\frac{1}{2}}(\vv)\left(\vs_i^2-\beta_b\right)\vs_l\p_{lj}\phi_b^j.\no
\end{eqnarray}
For such $\beta_b$ and any $i\neq l$, we can directly compute
\begin{eqnarray}
\int_{\r^2}\m(\vv)\left(\abs{\vs_i}^2-\beta_b\right)\vs_l^2\ud{\vv}&=&0,\\
\int_{\r^2}\m(\vv)\left(\abs{\vs_i}^2-\beta_b\right)\vs_i^2\ud{\vv}&=&C\neq0.
\end{eqnarray}
Then we deduce
\begin{eqnarray}
&&-\e\sum_{l=1}^2\int_{\Omega\times\r^2}\m(\vv)\vs_l^2\left(\vs_i^2-\beta_b\right)\p_{lj}\phi_b^jb_i\\
&=&-\e\int_{\Omega\times\r^2}\m(\vv)\vs_i^2\left(\vs_i^2-\beta_b\right)\p_{ij}\phi_b^jb_l-\e\sum_{l\neq
i}\int_{\Omega\times\r^2}\m(\vv)\vs_l^2\left(\vs_i^2-\beta_b\right)\p_{lj}\phi_b^jb_l\no\\
&=&C\int_{\Omega}\Big(\p_{ij}\dx^{-1}b_j\Big)b_i.\no
\end{eqnarray}
Hence, by (\ref{wt 16'}), we may estimate
\begin{eqnarray}\label{wt 19'}
\e\abs{\int_{\Omega}\Big(\p_{ij}\dx^{-1}b_j\Big)b_i}
&\leq&
C\nm{\vbb}_{L^{2m}}^{2m-1}\bigg(\e\abs{(1-\pp)[f]}_{L^m_+}+\tm{(\ik-\pk)[f]}+\e\nm{(\ik-\pk)[f]}_{L^{2m}}\\
&&+\tm{S}+\e\abs{h}_{L^{m}_-}\bigg).\no
\end{eqnarray}
\ \\
Step 2.2: Estimates of $(\p_{jj}\dx^{-1}b_i)b_i$ for $i\neq
j$.\\
We choose the test function
\begin{eqnarray}
\psi=\m^{\frac{1}{2}}(\vv)\abs{\vv}^2\vs_i\vs_j\p_j\phi_b^i\ \ i\neq j.
\end{eqnarray}
The right-hand side (RHS) of (\ref{wt 11'}) is still bounded by
\begin{eqnarray}
\text{RHS}\leq C\nm{\vbb}_{L^{2m}}^{2m-1}\bigg(\tm{(\ik-\pk)[f]}+\tm{S}\bigg).
\end{eqnarray}
Hence, the left-hand side (LHS) of (\ref{wt 11'}) takes the form
\begin{eqnarray}
\text{LHS}
&=&\e\int_{\p\Omega\times\r^2}\id_{\gamma_+}(1-\pp)[f]
\m^{\frac{1}{2}}(\vv)\abs{\vv}^2\vs_i\vs_j\p_j\phi_b^i(\vv\cdot\vn)\\
&&+\e\int_{\p\Omega\times\r^2}\id_{\gamma_-}h
\m^{\frac{1}{2}}(\vv)\abs{\vv}^2\vs_i\vs_j\p_j\phi_b^i(\vv\cdot\vn)\no\\
&&-\e\int_{\Omega\times\r^2}\m(\vv)\abs{\vv}^2\vs_i^2\vs_j^2\Big(\p_{ij}\phi_b^ib_j+\p_{jj}\phi_b^ib_i\Big)\no\\
&&-\e\sum_{l=1}^2\int_{\Omega\times\r^2}(\ik-\pk)[f]\m^{\frac{1}{2}}(\vv)\abs{\vv}^2\vs_i\vs_j\vs_l\p_{lj}\phi_b^i.\no
\end{eqnarray}
Then we deduce
\begin{eqnarray}
&&-\e\int_{\Omega\times\r^2}\m(\vv)\abs{\vv}^2\vs_i^2\vs_j^2\Big(\p_{ij}\phi_b^ib_j+\p_{jj}\phi_b^ib_i\Big)=
C\bigg(\int_{\Omega}\Big(\p_{ij}\dx^{-1}b_i\Big)b_j+\int_{\Omega}\Big(\p_{jj}\dx^{-1}b_i\Big)b_i\bigg).
\end{eqnarray}
Hence, we may estimate for $i\neq j$,
\begin{eqnarray}
\e\abs{\int_{\Omega}\Big(\p_{jj}\dx^{-1}b_i\Big)b_i}
&\leq&C\nm{\vbb}_{L^{2m}}^{2m-1}\bigg(\e\abs{(1-\pp)[f]}_{L^m_+}+\tm{(\ik-\pk)[f]}+\e\nm{(\ik-\pk)[f]}_{L^{2m}}\\
&&+\tm{S}+\e\abs{h}_{L^{m}_-}\bigg)+C\e\abs{\int_{\Omega}\Big(\p_{ij}\dx^{-1}b_i\Big)b_j},\no
\end{eqnarray}
which implies
\begin{eqnarray}\label{wt 21'}
\e\abs{\int_{\Omega}\Big(\p_{jj}\dx^{-1}b_i\Big)b_i}
&\leq&C\nm{\vbb}_{L^{2m}}^{2m-1}\bigg(\e\abs{(1-\pp)[f]}_{L^m_+}+\tm{(\ik-\pk)[f]}+\e\nm{(\ik-\pk)[f]}_{L^{2m}}\\
&&+\tm{S}+\e\abs{h}_{L^{m}_-}\bigg).\no
\end{eqnarray}
Moreover, by (\ref{wt 19'}), for $i=j=1,2$,
\begin{eqnarray}\label{wt 22'}
\e\abs{\int_{\Omega}\Big(\p_{jj}\dx^{-1}b_j\Big)b_j}
&\leq&C\nm{\vbb}_{L^{2m}}^{2m-1}\bigg(\e\abs{(1-\pp)[f]}_{L^m_+}+\tm{(\ik-\pk)[f]}+\e\nm{(\ik-\pk)[f]}_{L^{2m}}\\
&&+\tm{S}+\e\abs{h}_{L^{m}_-}\bigg).\no
\end{eqnarray}
\ \\
Step 2.3: Synthesis.\\
Summarizing (\ref{wt 21'}) and (\ref{wt 22'}), we may sum up over
$j=1,2$ to obtain, for any $i=1,2$,
\begin{eqnarray}\label{wt 23'}
\\
\e\nm{b_i}_{L^{2m}}^{2m}
&\leq&C\nm{\vbb}_{L^{2m}}^{2m-1}\bigg(\e\abs{(1-\pp)[f]}_{L^m_+}+\tm{(\ik-\pk)[f]}+\e\nm{(\ik-\pk)[f]}_{L^{2m}}+\tm{S}+\e\abs{h}_{L^{m}_-}\bigg).\no
\end{eqnarray}
which further implies
\begin{eqnarray}
\\
\e\nm{\vbb}_{L^{2m}}^{2m}&\leq&C\nm{\vbb}_{L^{2m}}^{2m-1}\bigg(\e\abs{(1-\pp)[f]}_{L^m_+}+\tm{(\ik-\pk)[f]}+\e\nm{(\ik-\pk)[f]}_{L^{2m}}+\tm{S}+\e\abs{h}_{L^{m}_-}\bigg).\no
\end{eqnarray}
Then we have
\begin{eqnarray}\label{wt 24'}
\e\nm{\vbb}_{L^{2m}}&\leq&C\bigg(\e\abs{(1-\pp)[f]}_{L^m_+}+\tm{(\ik-\pk)[f]}+\e\nm{(\ik-\pk)[f]}_{L^{2m}}+\tm{S}+\e\abs{h}_{L^{m}_-}\bigg).
\end{eqnarray}
\ \\
Step 3: Estimates of $a$.\\
We choose the test function
\begin{eqnarray}\label{wt 25'}
\psi=\psi_a=\m^{\frac{1}{2}}(\vv)\left(\abs{\vv}^2-\beta_a\right)\Big(\vv\cdot\nx\phi_a(\vx)\Big),
\end{eqnarray}
where
\begin{eqnarray}
\left\{
\begin{array}{rcl}
-\dx\phi_a&=&a^{2m-1}(\vx)-\dfrac{1}{\abs{\Omega}}\ds\int_{\Omega}a^{2m-1}(\vx)\ud{\vx}\ \ \text{in}\ \ \Omega,\\\rule{0ex}{1.5em}
\dfrac{\p\phi_a}{\p\vn}&=&0\ \ \text{on}\ \ \p\Omega,
\end{array}
\right.
\end{eqnarray}
and $\beta_a$ is a real number to be determined later. Based on the
standard elliptic estimates with
\begin{eqnarray}
\int_{\Omega}\bigg(a^{2m-1}(\vx)-\dfrac{1}{\abs{\Omega}}\ds\int_{\Omega}a^{2m-1}(\vx)\ud{\vx}\bigg)\ud{\vx}=\int_{\Omega\times\r^2}f(\vx,\vv)\ud{\vv}\ud{\vx}=0,
\end{eqnarray}
we have
\begin{eqnarray}
\nm{\phi_a}_{W^{2,\frac{2m}{2m-1}}}\leq C\nm{a^{2m-1}}_{L^{\frac{2m}{2m-1}}}\leq C\nm{a}_{L^{2m}}^{2m-1}.
\end{eqnarray}
Also, we know
\begin{eqnarray}
\nm{\psi_a}_{L^2}&\leq& C\nm{\phi_a}_{H^1}\leq C\nm{\phi_a}_{W^{2,\frac{2m}{2m-1}}}\leq C\nm{a}_{L^{2m}}^{2m-1},\\
\nm{\psi_a}_{L^\frac{2m}{2m-1}}&\leq&C\nm{\phi_a}_{W^{1,\frac{2m}{2m-1}}}\leq C\nm{a}_{L^{2m}}^{2m-1}.
\end{eqnarray}
With the choice of (\ref{wt 25'}), the right-hand side (RHS) of
(\ref{wt 11'}) is bounded by
\begin{eqnarray}
\text{RHS}\leq C \nm{a}_{L^{2m}}^{2m-1}\bigg(\tm{(\ik-\pk)[f]}+\tm{S}\bigg).
\end{eqnarray}
We will choose $\beta_a$ such that
\begin{eqnarray}
\int_{\r^2}\m^{\frac{1}{2}}(\vv)\left(\abs{\vv}^2-\beta_a\right)\frac{\abs{\vv}^2-2}{2}\vs_i^2\ud{\vv}=0\
\ \text{for}\ \ i=1,2.
\end{eqnarray}
The left-hand side (LHS) of (\ref{wt 11'}) takes the form
\begin{eqnarray}
\text{LHS}
&=&\e\int_{\p\Omega\times\r^2}\id_{\gamma_+}(1-\pp)[f]
\m^{\frac{1}{2}}(\vv)\left(\abs{\vv}^2-\beta_a\right)\left(\sum_{i=1}^2\vs_i\p_i\phi_a\right)(\vv\cdot\vn)\\
&&+\e\int_{\p\Omega\times\r^2}\id_{\gamma_-}h
\m^{\frac{1}{2}}(\vv)\left(\abs{\vv}^2-\beta_a\right)\left(\sum_{i=1}^2\vs_i\p_i\phi_a\right)(\vv\cdot\vn)\no\\
&&-\sum_{i=1}^2\e\int_{\r^2}\m(\vv)\abs{\vs_i}^2\left(\abs{\vv}^2-\beta_a\right)\ud{\vv}
\int_{\Omega}a(\vx)\p_{ii}\phi_a(\vx)\ud{\vx}\no\\
&&-\e\int_{\Omega\times\r^2}(\ik-\pk)[f]\m^{\frac{1}{2}}(\vv)\left(\abs{\vv}^2-\beta_a\right)\left(\sum_{i,j=1}^2\vs_i\vs_j\p_{ij}\phi_a\right).\no
\end{eqnarray}
Since
\begin{eqnarray}
\int_{\r^2}\m^{\frac{1}{2}}(\vv)\abs{\vs_i}^2\left(\abs{\vv}^2-\beta_a\right)\ud{\vv}=C,
\end{eqnarray}
we have
\begin{eqnarray}
-\e\int_{\Omega}\dx\phi_a(\vx)a(\vx)\ud{\vx}
&\leq&C
\nm{a}_{L^{2m}}^{2m-1}\bigg(\e\abs{(1-\pp)[f]}_{L^m_+}+\tm{(\ik-\pk)[f]}+\e\nm{(\ik-\pk)[f]}_{L^{2m}}\\
&&+\tm{S}+\e\abs{h}_{L^{m}_-}\bigg).\no
\end{eqnarray}
Since $-\dx\phi_a=a^{2m-1}-\ds\frac{1}{\abs{\Omega}}\int_{\Omega}a^{2m-1}$, by (\ref{wt 24}), we know
\begin{eqnarray}
\\
\e\nm{a}_{L^{2m}}^{2m-1} &\leq&C
\nm{a}_{L^{2m}}^{2m-1}\bigg(\e\abs{(1-\pp)[f]}_{L^m_+}+\tm{(\ik-\pk)[f]}+\e\nm{(\ik-\pk)[f]}_{L^{2m}}+\tm{S}+\e\abs{h}_{L^{m}_-}\bigg).\no
\end{eqnarray}
This implies
\begin{eqnarray}\label{wt 28'}
\e\nm{a}_{L^{2m}} &\leq&C
\bigg(\e\abs{(1-\pp)[f]}_{L^m_+}+\tm{(\ik-\pk)[f]}+\e\nm{(\ik-\pk)[f]}_{L^{2m}}+\tm{S}+\e\abs{h}_{L^{m}_-}\bigg).\no
\end{eqnarray}
\ \\
Step 4: Synthesis.\\
Collecting (\ref{wt 16'}), (\ref{wt 24'}) and (\ref{wt 28'}), we deduce
\begin{eqnarray}
\e\nm{\pk[f]}_{L^{2m}}
&\leq& C\bigg(\e\abs{(1-\pp)[f]}_{L^m_+}+\tm{(\ik-\pk)[f]}+\e\nm{(\ik-\pk)[f]}_{L^{2m}}+\tm{S}+\e\abs{h}_{L^{m}_-}\bigg).
\end{eqnarray}
This completes our proof.
\end{proof}

\begin{theorem}\label{LN estimate}
The solution $f(\vx,\vw)$ to the equation (\ref{linear steady}) satisfies the estimate
\begin{eqnarray}
&&\frac{1}{\e^{\frac{1}{2}}}\tss{(1-\pp)[f]}{+}+\frac{1}{\e}\um{(\ik-\pk)[f]}+\nm{\pk[f]}_{L^{2m}}\\
&\leq& C\bigg(o(1)\e^{\frac{1}{m}}\nm{f}_{L^{\infty}}+\frac{1}{\e^{2}}\nm{\pk[S]}_{L^{\frac{2m}{2m-1}}}+\frac{1}{\e}\tm{S}+\abs{h}_{L^{m}_-}+\frac{1}{\e}\tss{h}{-}\bigg).\no
\end{eqnarray}
\end{theorem}
\begin{proof}
We divide it into several steps:\\
\ \\
Step 1: Energy Estimate.\\
Multiplying $f$ on both sides of (\ref{linear steady}) and applying Green's identity imply
\begin{eqnarray}\label{wt 34'}
\frac{\e}{2}\tss{f}{+}^2+\br{\ll[f],f}
&=&\frac{\e}{2}\tss{\pp[f]+h}{-}^2+\int_{\Omega\times\r^2}fS.
\end{eqnarray}
Considering the fact that
\begin{eqnarray}
\tss{f}{+}^2-\tss{\pp[f]}{+}^2=\tss{(1-\pp)[f]}{+}^2,
\end{eqnarray}
We deduce from the spectral gap of $\ll$ and Cauchy's inequality that
\begin{eqnarray}\label{wt 31'}
\frac{\e}{2}\tss{(1-\pp)[f]}{+}^2+\um{(\ik-\pk)[f]}^2
&\leq&\eta\e^2\tss{\pp[f]}{-}^2+\left(1+\frac{1}{\eta}\right)\tss{h}{-}^2+\int_{\Omega\times\r^2}fS.
\end{eqnarray}
Following the same argument as in $L^2$ estimate, we obtain
\begin{eqnarray}\label{wt 32'}
\abs{\pp[f]}_2^2\leq C\left(\tm{f}^2+\frac{1}{\e}\tm{(\ik-\pk)[f]}^2+\frac{1}{\e}\int_{\Omega\times\r^2}fS\right).
\end{eqnarray}
Plugging (\ref{wt 32'}) into (\ref{wt 31'}) with $\e$ sufficiently
small to absorb $\tm{(\ik-\pk)[f]}^2$ into the left-hand side, we obtain
\begin{eqnarray}\label{wt 33'}
\e\tss{(1-\pp)[f]}{+}^2+\um{(\ik-\pk)[f]}^2
&\leq&C\bigg(\eta\e^2\tm{\pk[f]}^2+\left(1+\frac{1}{\eta}\right)\tss{h}{-}^2+\int_{\Omega\times\r^2}fS\bigg).
\end{eqnarray}
We square on both sides of (\ref{wt 08'}) to obtain
\begin{eqnarray}\label{wt 29'}
\\
\e^2\nm{\pk[f]}_{L^{2m}}^2
&\leq& C\bigg(\e^2\abs{(1-\pp)[f]}_{L^m_+}^2+\tm{(\ik-\pk)[f]}^2+\e^2\nm{(\ik-\pk)[f]}_{L^{2m}}^2+\tm{S}^2+\e^2\abs{h}_{L^{m}_-}^2\bigg).\no
\end{eqnarray}
Multiplying a small constant on both sides of (\ref{wt 29'}) and
adding to (\ref{wt 33'}) with $\eta>0$ sufficiently small to absorb $\e^2\nm{\pk[f]}_{L^{2m}}^2$ and $\tm{(\ik-\pk)[f]}$ into the left-hand side, we obtain
\begin{eqnarray}\label{wt 35'}
&&\e\tss{(1-\pp)[f]}{+}^2+\um{(\ik-\pk)[f]}^2+\e^2\nm{\pk[f]}_{L^{2m}}^2\\
&\leq& C\bigg(\e^2\abs{(1-\pp)[f]}_{L^m_+}^2+\e^2\nm{(\ik-\pk)[f]}_{L^{2m}}^2+\tm{S}^2+\e^2\abs{h}_{L^{m}_-}^2+\tss{h}{-}^2+\int_{\Omega\times\r^2}fS\bigg).\no
\end{eqnarray}
\ \\
Step 2: Interpolation Argument.\\
By interpolation estimate and Young's inequality, we have
\begin{eqnarray}
\abs{(1-\pp)[f]}_{L^{m}}&\leq&\abs{(1-\pp)[f]}_{L^2}^{\frac{2}{m}}\abs{(1-\pp)[f]}_{L^{\infty}}^{\frac{m-2}{m}}\\
&=&\bigg(\frac{1}{\e^{\frac{m-2}{m^2}}}\abs{(1-\pp)[f]}_{L^2}^{\frac{2}{m}}\bigg)
\bigg(\e^{\frac{m-2}{m^2}}\abs{(1-\pp)[f]}_{L^{\infty}}^{\frac{m-2}{m}}\bigg)\no\\
&\leq&C\bigg(\frac{1}{\e^{\frac{m-2}{m^2}}}\abs{(1-\pp)[f]}_{L^2}^{\frac{2}{m}}\bigg)^{\frac{m}{2}}+o(1)
\bigg(\e^{\frac{m-2}{m^2}}\abs{(1-\pp)[f]}_{L^{\infty}}^{\frac{m-2}{m}}\bigg)^{\frac{m}{m-2}}\no\\
&\leq&\frac{C}{\e^{\frac{m-2}{2m}}}\abs{(1-\pp)[f]}_{L^2}+o(1)\e^{\frac{1}{m}}\abs{(1-\pp)[f]}_{L^{\infty}}\no\\
&\leq&\frac{C}{\e^{\frac{m-2}{2m}}}\abs{(1-\pp)[f]}_{L^2}+o(1)\e^{\frac{1}{m}}\abs{f}_{L^{\infty}(\Omega\times\s^1)}.\no
\end{eqnarray}
Similarly, we have
\begin{eqnarray}
\nm{(\ik-\pk)[f]}_{L^{2m}}&\leq&\nm{(\ik-\pk)[f]}_{L^2}^{\frac{1}{m}}\nm{(\ik-\pk)[f]}_{L^{\infty}}^{\frac{m-1}{m}}\\
&=&\bigg(\frac{1}{\e^{\frac{m-1}{m^2}}}\nm{(\ik-\pk)[f]}_{L^2}^{\frac{1}{m}}\bigg)
\bigg(\e^{\frac{m-1}{m^2}}\nm{(\ik-\pk)[f]}_{L^{\infty}}^{\frac{m-1}{m}}\bigg)\no\\
&\leq&C\bigg(\frac{1}{\e^{\frac{m-1}{m^2}}}\nm{(\ik-\pk)[f]}_{L^2}^{\frac{1}{m}}\bigg)^{m}+o(1)
\bigg(\e^{\frac{m-1}{m^2}}\nm{(\ik-\pk)[f]}_{L^{\infty}}^{\frac{m-1}{m}}\bigg)^{\frac{m}{m-1}}\no\\
&\leq&\frac{C}{\e^{\frac{m-1}{m}}}\nm{(\ik-\pk)[f]}_{L^2}+o(1)\e^{\frac{1}{m}}\nm{(\ik-\pk)[f]}_{L^{\infty}}.\no
\end{eqnarray}
We need this extra $\e^{\frac{1}{m}}$ for the convenience of $L^{\infty}$ estimate.
Then we know for sufficiently small $\e$,
\begin{eqnarray}
\e^2\abs{(1-\pp)[f]}_{L^{m}}^2
&\leq&C\e^{2-\frac{m-2}{m}}\abs{(1-\pp)[f]}_{L^2}^2+o(1)\e^{2+\frac{2}{m}}\abs{f}_{L^{\infty}}^2\\
&\leq&o(1)\e\abs{(1-\pp)[f]}_{L^2}^2+o(1)\e^{2+\frac{2}{m}}\abs{f}_{L^{\infty}}^2.\no
\end{eqnarray}
Similarly, we have
\begin{eqnarray}
\e^2\nm{(\ik-\pk)[f]}_{L^{2m}}^2&\leq&\e^{2-\frac{2m-2}{m}}\nm{(\ik-\pk)[f]}_{L^2}^2+o(1)\e^{2+\frac{2}{m}}\nm{u}_{L^{\infty}}^2\\
&\leq& o(1)\nm{(\ik-\pk)[f]}_{L^2}^2+o(1)\e^{2+\frac{2}{m}}\nm{f}_{L^{\infty}}^2.\no
\end{eqnarray}
In (\ref{wt 35'}), we can absorb $\e\abs{(1-\pp)[f]}_{L^2}^2$ and $\nm{(\ik-\pk)[f]}_{L^2}^2$ into left-hand side to obtain
\begin{eqnarray}\label{wt 36'}
&&\e\tss{(1-\pp)[f]}{+}^2+\um{(\ik-\pk)[f]}^2+\e^2\nm{\pk[f]}_{L^{2m}}^2\\
&\leq& C\bigg(o(1)\e^{2+\frac{2}{m}}\nm{f}_{L^{\infty}}^2+\tm{S}^2+\e^2\abs{h}_{L^{m}_-}^2+\tss{h}{-}^2+\int_{\Omega\times\r^2}fS\bigg).\no
\end{eqnarray}
We can decompose
\begin{eqnarray}
\int_{\Omega\times\r^2}fS=\iint_{\Omega\times\r^2}\pk[S]\pk[f]+\iint_{\Omega\times\r^2}(\ik-\pk)S(\ik-\pk)[f].
\end{eqnarray}
H\"older's inequality and Cauchy's inequality imply
\begin{eqnarray}
\int_{\Omega\times\r^2}\pk[S]\pk[f]\leq\nm{\pk[S]}_{L^{\frac{2m}{2m-1}}}\nm{\pk[f]}_{L^{2m}}
\leq\frac{C}{\e^{2}}\nm{\pk[S]}_{L^{\frac{2m}{2m-1}}}^2+o(1)\e^2\nm{\pk[f]}_{L^{2m}}^2,
\end{eqnarray}
and
\begin{eqnarray}
\int_{\Omega\times\r^2}(\ik-\pk)S(\ik-\pk)[f]\leq C\nm{\nu^{-\frac{1}{2}}(\ik-\pk)S}_{L^{2}}^2+o(1)\um{(\ik-\pk)[f]}^2.
\end{eqnarray}
Hence, absorbing $\e^2\nm{\pk[f]}_{L^{2m}(\Omega\times\s^1)}^2$ and $\um{(\ik-\pk)[f]}$ into left-hand side of (\ref{wt 36'}), we get
\begin{eqnarray}
&&\e\tss{(1-\pp)[f]}{+}^2+\um{(\ik-\pk)[f]}^2+\e^2\nm{\pk[f]}_{L^{2m}}^2\\
&\leq& C\bigg(o(1)\e^{2+\frac{2}{m}}\nm{f}_{L^{\infty}}^2+\frac{1}{\e^{2}}\nm{\pk[S]}_{L^{\frac{2m}{2m-1}}}^2+\tm{S}^2+\e^2\abs{h}_{L^{m}_-}^2+\tss{h}{-}^2\bigg).\no
\end{eqnarray}
Therefore, we have
\begin{eqnarray}
&&\frac{1}{\e^{\frac{1}{2}}}\tss{(1-\pp)[f]}{+}+\frac{1}{\e}\um{(\ik-\pk)[f]}+\nm{\pk[f]}_{L^{2m}}\\
&\leq& C\bigg(o(1)\e^{\frac{1}{m}}\nm{f}_{L^{\infty}}+\frac{1}{\e^{2}}\nm{\pk[S]}_{L^{\frac{2m}{2m-1}}}+\frac{1}{\e}\tm{S}+\abs{h}_{L^{m}_-}+\frac{1}{\e}\tss{h}{-}\bigg).\no
\end{eqnarray}
\end{proof}

\subsection{$L^{\infty}$ Estimates - Second Round}

\begin{theorem}\label{LI estimate}
The
solution $f(\vx,\vw)$ to the equation (\ref{linear steady}) satisfies the
estimate for $\vth\geq3$ and $0\leq\varrho<\dfrac{1}{4}$,
\begin{eqnarray}
\im{\bv f}&\leq& C\bigg(\frac{1}{\e^{2+\frac{1}{m}}}\nm{\pk[S]}_{L^{\frac{2m}{2m-1}}}+\frac{1}{\e^{1+\frac{1}{m}}}\tm{S}+\im{\frac{\bv S}{\nu}}\\
&&+\frac{1}{\e^{1+\frac{1}{m}}}\tss{h}{-}+\frac{1}{\e^{\frac{1}{m}}}\abs{h}_{L^{m}_-}+\iss{\bv h}{-}\bigg).\no
\end{eqnarray}
\end{theorem}
\begin{proof}
Following the argument in the proof of Theorem \ref{LI estimate'}, by double Duhamel's principle along the characteristics,
the key step is to decompose $f=\pk[f]+(\ik-\pk)[f]$ and utilize $L^{2m}$ estimates and $L^2$ estimates separately as
\begin{eqnarray}
I_{4,1}
&\leq&
C\abs{\int_{\abs{\vv'}\leq2N}\int_{\abs{\vv''}\leq3N}\int_{0}^{t_1'}
\id_{\{\xc-\e(t_1'-r)\vv'\in\Omega\}}f(\xc-\e(t_1'-r)\vv',\vv'')\ue^{-\nu(\vv')
(t_1'-r)}\ud{r}\ud{\vv'}\ud{\vv''}}\\
&\leq&C\bigg(\int_{\abs{\vv'}\leq2N}\int_{\abs{\vv''}\leq3N}\int_{0}^{t_1'}
\id_{\{\xc-\e(t_1'-r)\vv'\in\Omega\}}\ue^{-\nu(\vv')
(t_1'-r)}\ud{r}\ud{\vv'}\ud{\vv''}\bigg)^{\frac{2m-1}{2m}}\no\\
&&\bigg(\int_{\abs{\vv'}\leq2N}\int_{\abs{\vv''}\leq3N}\int_{0}^{t_1'}
\id_{\{\xc-\e(t_1'-r)\vv'\in\Omega\}}\Big(\pk[f]\Big)^{2m}(\xc-\e(t_1'-r)\vv',\vv'')\ue^{-\nu(\vv')
(t_1'-r)}\ud{r}\ud{\vv'}\ud{\vv''}\bigg)^{\frac{1}{2m}}\no\\
&&+C\bigg(\int_{\abs{\vv'}\leq2N}\int_{\abs{\vv''}\leq3N}\int_{0}^{t_1'}
\id_{\{\xc-\e(t_1'-r)\vv'\in\Omega\}}\ue^{-\nu(\vv')
(t_1'-r)}\ud{r}\ud{\vv'}\ud{\vv''}\bigg)^{\frac{1}{2}}\no\\
&&\bigg(\int_{\abs{\vv'}\leq2N}\int_{\abs{\vv''}\leq3N}\int_{0}^{t_1'}
\id_{\{\xc-\e(t_1'-r)\vv'\in\Omega\}}\Big((\ik-\pk)[f]\Big)^{2}(\xc-\e(t_1'-r)\vv',\vv'')\ue^{-\nu(\vv')
(t_1'-r)}\ud{r}\ud{\vv'}\ud{\vv''}\bigg)^{\frac{1}{2}}\no\\
&\leq&C\abs{\int_{0}^{t_1'}\frac{1}{\e^2\d^2}\int_{\abs{\vv''}\leq3N}\int_{\Omega}\id_{\{\vec
y\in\Omega\}}\Big(\pk[f]\Big)^{2m}(\vec y,\vv'')\ue^{-\nu(\vv')
(t_1'-r)}\ud{\vec y}\ud{\vv''}\ud{r}}^{\frac{1}{2m}}\no\\
&&+C\abs{\int_{0}^{t_1'}\frac{1}{\e^2\d^2}\int_{\abs{\vv''}\leq3N}\int_{\Omega}\id_{\{\vec
y\in\Omega\}}\Big((\ik-\pk)[f]\Big)^{2}(\vec y,\vv'')\ue^{-\nu(\vv')
(t_1'-r)}\ud{\vec y}\ud{\vv''}\ud{r}}^{\frac{1}{2}}\no\\
&=&\frac{C}{\e^{\frac{1}{m}}\d^{\frac{1}{m}}}\nm{\pk[f]}_{L^{2m}}+\frac{C}{\e\d}\nm{(\ik-\pk)[f]}_{L^{2}}.\no
\end{eqnarray}
All the other terms can be estimated in the similar fashion. In summary, we have
\begin{eqnarray}
\\
\im{\bv f}\leq C\bigg(\frac{1}{\e^{\frac{1}{m}}}\nm{\pk[f]}_{L^{2m}}+\frac{1}{\e}\nm{(\ik-\pk)[f]}_{L^{2}}+\im{\frac{\bv S}{\nu}}+\iss{\bv h}{-}\bigg).\no
\end{eqnarray}
Then using $L^{2m}$ and $L^2$ estimates in Theorem \ref{LI estimate}, we know
\begin{eqnarray}
\\
\im{\bv f}&\leq& C\bigg(o(1)\nm{\bv f}_{L^{\infty}}+\frac{1}{\e^{2+\frac{1}{m}}}\nm{\pk[S]}_{L^{\frac{2m}{2m-1}}}+\frac{1}{\e^{1+\frac{1}{m}}}\tm{S}+\im{\frac{\bv S}{\nu}}\no\\
&&+\frac{1}{\e^{1+\frac{1}{m}}}\tss{h}{-}+\frac{1}{\e^{\frac{1}{m}}}\abs{h}_{L^{m}_-}+\iss{\bv h}{-}\bigg).\no
\end{eqnarray}
Since $o(1)$ is small, we can absorb $\nm{\bv f}_{L^{\infty}}$ into the left-hand side to obtain
\begin{eqnarray}
\im{\bv f}&\leq& C\bigg(\frac{1}{\e^{2+\frac{1}{m}}}\nm{\pk[S]}_{L^{\frac{2m}{2m-1}}}+\frac{1}{\e^{1+\frac{1}{m}}}\tm{S}+\im{\frac{\bv S}{\nu}}\\
&&+\frac{1}{\e^{1+\frac{1}{m}}}\tss{h}{-}+\frac{1}{\e^{\frac{1}{m}}}\abs{h}_{L^{m}_-}+\iss{\bv h}{-}\bigg).\no
\end{eqnarray}
\end{proof}

\newpage

\section{Well-Posedness of $\e$-Milne Problem with Geometric Correction}

We consider the $\e$-Milne problem with geometric correction for $\g(\eta,\theta,\vww)$ in
the domain
$(\eta,\theta,\vww)\in[0,L]\times[-\pi,\pi)\times\r^2$ as
\begin{eqnarray}\label{Milne.}
\left\{
\begin{array}{l}\displaystyle
\va\dfrac{\p\g }{\p\eta}-\dfrac{\e}{\rk-\e\eta}\bigg(\vb^2\dfrac{\p\g }{\p\va}-\va\vb\dfrac{\p\g }{\p\vb}\bigg)+\ll[\g ]
=0,\\\rule{0ex}{1.5em}
\g (0,\theta,\vww)=h (\theta,\vww)\ \ \text{for}\ \
\va>0,\\\rule{0ex}{1.5em}
\displaystyle\g (L,\theta,\vww)=\g (L,\theta,\rr[\vww]),
\end{array}
\right.
\end{eqnarray}
where $\rr[\vvv]=(-\va,\vb)$ and $L=\e^{-\frac{1}{2}}$. For simplicity, we temporarily ignore the dependence of $\theta$, i.e. consider
the $\e$-Milne problem with geometric correction for $\g(\eta,\vww)$ in
the domain
$(\eta,\vww)\in[0,L]\times\r^2$ as
\begin{eqnarray}\label{Milne}
\left\{
\begin{array}{l}\displaystyle
\va\dfrac{\p\g }{\p\eta}-\dfrac{\e}{\rk-\e\eta}\bigg(\vb^2\dfrac{\p\g }{\p\va}-\va\vb\dfrac{\p\g }{\p\vb}\bigg)+\ll[\g ]
=0,\\\rule{0ex}{1.5em}
\g (0,\vww)=h (\vww)\ \ \text{for}\ \
\va>0,\\\rule{0ex}{1.5em}
\displaystyle\g (L,\vww)=\g (L,\rr[\vww]).
\end{array}
\right.
\end{eqnarray}
Since the null space of the operator $\ll$ is spanned by
$\nk=\m^{\frac{1}{2}}\bigg\{1,\va,\vb,\dfrac{\abs{\vvv}^2-2}{2}\bigg\}=\{\psi_0,\psi_1,\psi_2,\psi_3\}$,
we can decompose the solution as
\begin{eqnarray}
\g=w_g+q_g,
\end{eqnarray}
where
\begin{eqnarray}
q_g&=&\m^{\frac{1}{2}}\bigg(q_{g,0}+q_{g,1}\va+q_{g,2}\vb+q_{g,3}\dfrac{\abs{\vvv}^2-2}{2}\bigg)=q_{g,0}\psi_0+q_{g,1}\psi_1+q_{g,2}\psi_2+q_{g,3}\psi_3\in\nk,
\end{eqnarray}
and
\begin{eqnarray}
w_g\in\nk^{\perp},
\end{eqnarray}
where $\nk^{\perp}$ is the orthogonal space of $\nk$. When there is no confusion, we will simply write $\g=w+q$.
Our main goal is to find
\begin{eqnarray}\label{Milne transform compatibility}
\tilde h(\vvv)=\sum_{i=0}^3\tilde D_i\psi_i\in\nk,
\end{eqnarray}
with $\tilde D_1=0$ such that the $\e$-Milne problem with geometric correction for
$\gg(\eta,\vvv)$ in the domain
$(\eta,\vvv)\in[0,L]\times\r^2$ as
\begin{eqnarray}\label{Milne transform}
\left\{
\begin{array}{l}\displaystyle
\va\dfrac{\p\gg }{\p\eta}-\dfrac{\e}{\rk-\e\eta}\bigg(\vb^2\dfrac{\p\gg }{\p\va}-\va\vb\dfrac{\p\gg }{\p\vb}\bigg)+\ll[\gg]
=0,\\\rule{0ex}{1.5em}
\gg (0,\vww)=h (\vww)-\tilde h(\vvv)\ \ \text{for}\ \
\va>0,\\\rule{0ex}{1.5em}
\displaystyle\gg (L,\vww)=\gg (L,\rr[\vww]),
\end{array}
\right.
\end{eqnarray}
is well-posed, and $\gg$ decays exponentially fast as $\eta$ becomes larger and larger. The estimates and decaying rate should be uniform in $\e$.

Let $G(\eta)=-\dfrac{\e}{\rk-\e\eta}$. We define a potential function $W(\eta)$ as
$G(\eta)=-\dfrac{\ud W}{\ud\eta}$ with $W(0)=0$. It is easy to check that
\begin{eqnarray}
W(\eta)=\ln\left(\frac{\rk}{\rk-\e\eta}\right).
\end{eqnarray}

In this section, we introduce some special notation to describe the
norms in the space $(\eta,\vvv)\in[0,L]\times\r^2$. Define
the $L^2$ norm as follows:
\begin{eqnarray}
\tnm{f(\eta)}&=&\bigg(\int_{\r^2}\abs{f(\eta,\vvv)}^2\ud{\vvv}\bigg)^{1/2},\\
\tnnm{f}&=&\bigg(\int_0^{L}\int_{\r^2}\abs{f(\eta,\vvv)}^2\ud{\vvv}\ud{\eta}\bigg)^{1/2}.
\end{eqnarray}
Define the inner product in $\vvv$ space
\begin{eqnarray}
\br{f,g}(\eta)=\int_{\r^2}
f(\eta,\vvv)g(\eta,\vvv)\ud{\vvv}.
\end{eqnarray}
Define the weighted $L^{\infty}$ norm as follows:
\begin{eqnarray}
\lnm{f(\eta)}{\vth,\varrho}&=&\sup_{\vvv\in\r^2}\bigg(\bvv\abs{f(\eta,\vvv)}\bigg),\\
\lnnm{f}{\vth,\varrho}&=&\sup_{(\eta,\vvv)\in[0,L]\times\r^2}\bigg(\bvv\abs{f(\eta,\vvv)}\bigg),
\end{eqnarray}
Define the mixed $L^2$ and weighted $L^{\infty}$ norm as follows:
\begin{eqnarray}
\ltnm{f}{\varrho}&=&\sup_{\eta\in[0,L]}\bigg(\int_{\r^2}\abs{\ue^{2\varrho\abs{\vvv}^2}f(\eta,\vvv)}^2\ud{\vvv}\bigg)^{1/2},
\end{eqnarray}
for $0\leq\varrho<\dfrac{1}{4}$ and an integer $\vth\geq3$.
Since the boundary data $h(\vvv)$ is only defined on $\va>0$, we
naturally extend above definitions on this half-domain as follows:
\begin{eqnarray}
\tnm{h}&=&\bigg(\int_{\va>0}\abs{h(\vvv)}^2\ud{\vvv}\bigg)^{1/2},\\
\lnm{h}{\vth,\varrho}&=&\sup_{\va>0}\bigg(\bvv\abs{h(\vvv)}\bigg).
\end{eqnarray}
We assume
\begin{eqnarray}\label{Milne bound}
\lnm{h}{\vth,\varrho}\leq C,
\end{eqnarray}
for some $C>0$ uniform in $\e$.

Here, we mainly refer to the procedure in \cite{Cercignani.Marra.Esposito1998} and \cite{AA004}, where $\eta\in[0,\infty)$. Since our domain for $\eta$ is bounded, we have to start from scratch and prove each result in the new settings.

\subsection{$L^2$ Estimates}

We write $\g=w+q$ with $q=q_0\psi_0+q_1\psi_1+q_2\psi_2+q_3\psi_3$.
\begin{lemma}\label{Milne lemma 1}
There exists a unique solution of the equation (\ref{Milne})
satisfying the estimates
\begin{eqnarray}
\tnnm{\sn{w}}&\leq&C,\\
\tnnm{{q}-{q}_{L}}&\leq& C,
\end{eqnarray}
for some $q_L$ satisfying $\abs{q_L}\leq C$,
where $C$ is a constant independent of $\e$. Also, the solution satisfies the orthogonal relation
\begin{eqnarray}
\br{\va\psi_i,w}(\eta)=0,\ \ \text{for}\ \ i=0,2,3.
\end{eqnarray}
\end{lemma}
\begin{proof}
The existence and uniqueness follow from a standard argument by adding penalty term $\lambda\g$ on the left-hand side of the equation for $0<\lambda<<1$ and estimate along the characteristics (see \cite{AA004}). Hence, we concentrate on the a priori estimates. We divide the proof into several steps:\\
\ \\
Step 1: Estimate of $w$.\\
Multiplying $\g$ on both sides of (\ref{Milne}) and
integrating over $\vvv\in\r^2$, we have
\begin{eqnarray}
\half\frac{\ud{}}{\ud{\eta}}\br{\va\g,\g}+
G(\eta)\br{\vb^2\dfrac{\p
\g}{\p\va}-\va\vb\dfrac{\p \g}{\p\vb},\g}=-\br{\g,\ll[\g]}.
\end{eqnarray}
An integration by parts implies
\begin{eqnarray}
\br{\vb^2\dfrac{\p
\g}{\p\va}-\va\vb\dfrac{\p \g}{\p\vb},\g}=\half\br{\vb^2,\dfrac{\p
(\g^2)}{\p\va}}-\half\br{\va\vb,\dfrac{\p (\g^2)}{\p\vb}}=\half\br{\va\g,\g}.
\end{eqnarray}
Therefore, using Lemma \ref{Milne property}, we obtain
\begin{eqnarray}
\half\frac{\ud{}}{\ud{\eta}}\br{\va\g,\g}+\half
G(\eta)\br{\va\g,\g}=-\br{w,\ll[w]}.
\end{eqnarray}
Define
\begin{eqnarray}
\alpha(\eta)=\half\br{\va\g,\g}(\eta),
\end{eqnarray}
which implies
\begin{eqnarray}
\frac{\ud{\alpha}}{\ud{\eta}}+G(\eta)\alpha=-\br{w,\ll[w]}.
\end{eqnarray}
Then we have
\begin{eqnarray}
\alpha(\eta)&=&\alpha(L)\exp\bigg(\int_{\eta}^LG(y)\ud{y}\bigg)+
\int_{\eta}^L\exp\bigg(-\int_{\eta}^yG(z)\ud{z}\bigg)\bigg(\br{w,\ll[w]}(y)\bigg)\ud{y},\label{mt 01}\\
\alpha(\eta)&=&\alpha(0)\exp\bigg(-\int_{0}^{\eta}G(y)\ud{y}\bigg)+
\int_{0}^{\eta}\exp\bigg(\int_{y}^{\eta}G(z)\ud{z}\bigg)\bigg(-\br{w,\ll[w]}(y)\bigg)\ud{y}.\label{mt 02}
\end{eqnarray}
Since $\alpha(L)=0$ due to the reflexive boundary condition and the coercivity $\br{w,\ll[w]}(\eta)\geq\tnm{\sn{w(\eta)}}$, (\ref{mt 01}) implies that
\begin{eqnarray}
\alpha(\eta)\geq0.
\end{eqnarray}
Considering
\begin{eqnarray}
\alpha(0)&=&\half\int_{\va>0}\va\g^2(0,\vvv)\ud{\vvv}+\half\int_{\va<0}\va\g^2(0,\vvv)\ud{\vvv}\leq\half\int_{\va>0}\va\g^2(0,\vvv)\ud{\vvv}\\
&=&\half\int_{\va>0}\va h^2(\vvv)\ud{\vvv}\leq C,\no
\end{eqnarray}
and (\ref{mt 02}), we obtain
\begin{eqnarray}\label{mt 03}
\alpha(\eta)\leq C.
\end{eqnarray}
Hence, (\ref{mt 01}) and (\ref{mt 03}) lead to
\begin{eqnarray}
\int_{0}^{L}\exp\bigg(-\int_{0}^{y}G(z)\ud{z}\bigg)\bigg(\br{w,\ll[w]}(y)\bigg)\ud{y}\leq
C,
\end{eqnarray}
which, by Lemma \ref{Milne property}, further yields
\begin{eqnarray}\label{mt 04}
\int_0^L\tnm{\sn{w}(\eta)}^2\ud{\eta}\leq C
\end{eqnarray}
\ \\
Step 2: Estimate of ${q}$.\\
Multiplying $\va\psi_j$ with $j\neq1$ on both sides of (\ref{Milne}) and integrating over
$\vvv\in\r^2$, we obtain
\begin{eqnarray}
\frac{\ud{}}{\ud{\eta}}\br{\va^2\psi_j,\g}+G(\eta)\br{\va\psi_j,
\vb^2\dfrac{\p \g}{\p\va}-\va\vb\dfrac{\p
\g}{\p\vb}}=-\br{\va\psi_j,\ll[w]}.
\end{eqnarray}
Define $\tilde q ={q}-q_1 \psi_1$ and
\begin{eqnarray}
\beta_j(\eta)&=&\br{\va^2\psi_j,\tilde q }(\eta),\\
\beta(\eta)&=&\bigg(\beta_0(\eta),\beta_1(\eta),\beta_2(\eta),\beta_3(\eta)\bigg)^T\\
\tilde\beta(\eta)&=&\bigg(\beta_0(\eta),\beta_2(\eta),\beta_3(\eta)\bigg)^T.
\end{eqnarray}
Due to symmetry, it is easy to check that $\beta_1=0$. For $j\neq1$, using integration by parts, we have
\begin{eqnarray}
\frac{\ud{}}{\ud{\eta}}\br{\va^2\psi_j,\g}&=&G(\eta)\br{\frac{\p}{\p\va}(\va\vb^2\psi_j)-\frac{\p}{\p\vb}(\va^2\vb\psi_j),
\g}-\br{\va\psi_j,\ll[w]},
\end{eqnarray}
which further implies
\begin{eqnarray}\label{mt 05}
\frac{\ud{\beta_j}}{\ud{\eta}}&=&G(\eta)\br{\frac{\p}{\p\va}(\va\vb^2\psi_j)-\frac{\p}{\p\vb}(\va^2\vb\psi_j),
\tilde q +q_1 \psi_1+w}-\br{\va\psi_j,\ll[w]}-\frac{\ud{}}{\ud{\eta}}\br{\va^2\psi_j,w}.
\end{eqnarray}
Then we can write
\begin{eqnarray}
\br{\frac{\p}{\p\va}(\va\vb^2\psi_j)-\frac{\p}{\p\vb}(\va^2\vb\psi_j),
\tilde q }(\eta)=\sum_{i}B_{ji} q_i (\eta),
\end{eqnarray}
for $i,j=0,2,3$, where
\begin{eqnarray}
B_{ji}=\br{\frac{\p}{\p\va}(\va\vb^2\psi_j)-\frac{\p}{\p\vb}(\va^2\vb\psi_j),
\psi_i}.
\end{eqnarray}
Moreover,
\begin{eqnarray}
\beta_j(\eta)=\sum_{k}A_{jk} q_k (\eta),
\end{eqnarray}
for $k,j=0,2,3$, where
\begin{eqnarray}
A_{jk}=\br{\va^2\psi_j,\psi_k},
\end{eqnarray}
constitutes a non-singular matrix $A$ such that we can express back
\begin{eqnarray}
q_j (\eta)=\sum_{k}A_{jk}^{-1}\beta_k(\eta).
\end{eqnarray}
Hence, (\ref{mt 05}) can be rewritten as
\begin{eqnarray}
\frac{\ud{\tilde\beta}}{\ud{\eta}}=G(BA^{-1})\tilde\beta+D,
\end{eqnarray}
where
\begin{eqnarray}\label{tt 01}
D_j&=&G(\eta)\br{\frac{\p}{\p\va}(\va\vb^2\psi_j)-\frac{\p}{\p\vb}(\va^2\vb\psi_j),
q_1 \psi_1+w}-\br{\va\psi_j,\ll[w]}-\frac{\ud{}}{\ud{\eta}}\br{\va^2\psi_j,w}.
\end{eqnarray}
We can solve for $\tilde\beta$ as
\begin{eqnarray}\label{mt 06}
\tilde\beta(\eta)&=&\exp\bigg(-W(\eta)BA^{-1}\bigg)\tilde\beta(0)+\int_0^{\eta}\exp\bigg(\Big(W(\eta)-W(y)\Big)BA^{-1}\bigg)D(y)\ud{y}.
\end{eqnarray}
We may further simplify the second term on the right-hand side. Consider the last term in (\ref{tt 01}). Define
\begin{eqnarray}
\zeta_j(\eta)=\br{\va^2\psi_j, w}(\eta).
\end{eqnarray}
We may directly integrate by parts to obtain
\begin{eqnarray}
\int_0^{\eta}\exp\bigg(\Big(W(\eta)-W(y)\Big)BA^{-1}\bigg)\frac{\ud{\zeta}}{\ud{y}}\ud{y}
&=&\zeta(\eta)-\exp\bigg(-W(\eta)BA^{-1}\bigg)\zeta(0)\\
&&-\int_0^{\eta}\exp\bigg(\Big(W(\eta)-W(y)\Big)BA^{-1}\bigg)G(y)(BA^{-1})\zeta(y)\ud{y}.\no
\end{eqnarray}
Hence, we may rewrite (\ref{mt 06}) as
\begin{eqnarray}\label{mt 07}
\tilde\beta(\eta)&=&\exp\bigg(-W(\eta)BA^{-1}\bigg)\theta-\zeta(\eta)+\int_0^{\eta}\exp\bigg(\Big(W(\eta)-W(y)\Big)BA^{-1}\bigg)Z(y)\ud{y},
\end{eqnarray}
where
\begin{eqnarray}
\theta_j=\br{\va^2\psi_j, \g}(0),\ \ j=0,2,3,
\end{eqnarray}
and
\begin{eqnarray}
Z=D+\frac{\ud{\zeta}}{\ud{\eta}}+G(BA^{-1})\zeta.
\end{eqnarray}
Hence, using the boundedness of $W(\eta)$ and $BA^{-1}$, we can directly estimate (\ref{mt 07}) to get
\begin{eqnarray}\label{mt 08}
\abs{\beta_j(\eta)}\leq
C\abs{\th_j}+\abs{\zeta_j(\eta)}+C\int_0^{\eta}\abs{Z_j(y)}\ud{y}\ \ \text{for}\ \ i=0,2,3.
\end{eqnarray}
By Cauchy's inequality and Lemma \ref{Milne property}, we obtain
\begin{eqnarray}
\abs{\zeta_j(\eta)}&\leq& \tnm{\sn{w}(\eta)}\label{mt 09}\\
\abs{Z_j(\eta)}&\leq&
C\bigg(\tnm{\sn{w}(\eta)}+q_1(\eta)\bigg)\label{mt 10}.
\end{eqnarray}
Multiplying $\psi_0$ on both sides of (\ref{Milne}) and
integrating over $\vvv\in\r^2$, we have
\begin{eqnarray}
\frac{\ud{}}{\ud{\eta}}\br{\psi_0\va,\g}&=&G(\eta)\br{\frac{\p}{\p\va}(\psi_0\vb^2)-\frac{\p}{\p\vb}(\psi_0\va\vb),
\g}=-G(\eta)\br{\psi_0\va,\g},
\end{eqnarray}
which is actually
\begin{eqnarray}
\frac{\ud{q_1 }}{\ud{\eta}}=-G(\eta)q_1 .
\end{eqnarray}
Since $q_1 (L)=0$, we have for any $\eta\in[0,L]$,
\begin{eqnarray}\label{mt 11}
q_1 (\eta)=0.
\end{eqnarray}
Also,
\begin{eqnarray}
\theta_j=\br{\va^2\psi_j, \g}(0)\leq C\br{\abs{\va}\g(0),
\g(0)}^{1/2}\br{\abs{\va}^3,\psi_j^2}^{1/2}\leq C\br{\abs{\va}\g(0),
\g(0)}^{1/2},
\end{eqnarray}
\begin{eqnarray}
\br{\abs{\va}\g(0), \g(0)}=\int_{\va>0}\va
h^2(\vvv)\ud{\vvv}-\int_{\va<0}\va\g^2(0,\vvv)\ud{\vvv}.
\end{eqnarray}
Since
\begin{eqnarray}
\int_{\va>0}\va h^2(\vvv)\ud{\vvv}+\int_{\va<0}\va\g^2(0,\vvv)\ud{\vvv}=2\alpha(0)\geq0,
\end{eqnarray}
we have
\begin{eqnarray}\label{mt 12}
\theta_j=\br{\va^2\psi_j, \g}(0)\leq 2C\int_{\va>0}\va h^2(\vvv)\ud{\vvv}\leq C.
\end{eqnarray}
In conclusion, collecting (\ref{mt 08}), (\ref{mt 09}), (\ref{mt 10}), (\ref{mt 11}), and (\ref{mt 12}), we have
\begin{eqnarray}
\abs{\beta_j(\eta)}\leq C\bigg(
1+\tnm{\sn{w}(\eta)}+\int_0^{\eta}\tnm{\sn{w}(y)}\ud{y}\bigg)\ \ \text{for}\ \ j=0,2,3,
\end{eqnarray}
which further implies
\begin{eqnarray}
\abs{{q}_j(\eta)}\leq C\bigg(
1+\tnm{\sn{w}(\eta)}+\int_0^{\eta}\tnm{\sn{w}(y)}\ud{y}\bigg)\ \ \text{for}\ \ j=0,2,3,
\end{eqnarray}
and ${q}_1(\eta)=0$. An application of Cauchy's inequality leads to our desired result.\\
\ \\
Step 3: Orthogonal Properties.\\
It is easy to check that $q_1=0$ and
\begin{eqnarray}
\br{\va\psi_i,{q}}&=&0,\ \ i=0,2,3.
\end{eqnarray}
In the equation (\ref{Milne}), multiplying $\psi_i$ for $i=0,2,3$ on both
sides and integrating over $\vvv\in\r^2$, we have
\begin{eqnarray}
\frac{\ud{}}{\ud{\eta}}\br{\psi_i\va,\g}&=&G(\eta)\br{\frac{\p}{\p\va}(\psi_i\vb^2)-\frac{\p}{\p\vb}(\psi_0\va\vb),
\g}=-G\br{\psi_i\va,\g}.
\end{eqnarray}
Since $\br{\psi_i\va,\g}(L)=0$ due to reflexive boundary condition, we have
\begin{eqnarray}
\br{\va\psi_i,\g}(\eta)=\br{\va\psi_i,w}(\eta)=0.
\end{eqnarray}
\ \\
Step 4: Estimate of ${q}_{L}$.\\
The estimates in Step 2 is not strong enough to bound $q$, so we need a different setting to further bound $w$ and $q$.
Considering $q_1(\eta)=0$ for any $\eta\in[0,L]$, we do not need to bother with it.

Since $\ll: L^2(\r^2)\rt\nk^{\bot}$ with null space $\nk$ and image $\nk^{\bot}$, we have $\tilde\ll: L^2/\nk\rt\nk^{\bot}$ is bijective, where $L^2/\nk=\nk^{\bot}$ is the quotient space. Then we can define its inverse, i.e. the pseudo-inverse of $\ll$ as
$\ll^{-1}: \nk^{\bot}\rt\nk^{\bot}$ satisfying $\ll\ll^{-1}[f]=f$ for any $f\in\nk^{\bot}$.

We intend to multiply
$\ll^{-1}[\va\psi_{i}]$ for $i=2,3$ on both sides of
(\ref{Milne}) and integrating over $\vvv\in\r^2$. Notice that $\va\psi_{2}\in \nk^{\bot}$, but $\va\psi_{3}\notin \nk^{\bot}$. Actually, it is easy to verify $\va(\psi_{3}-\psi_0)\in \nk^{\bot}$. To avoid introducing new notation, we still use $\psi_3$ to denote $\psi_{3}-\psi_0$ in the following proof and it is easy to see that there is no confusion. Then we get
\begin{eqnarray}
&&\frac{\ud{}}{\ud{\eta}}\br{\ll^{-1}[\psi_i\va],\va\g}+G(\eta)\br{\ll^{-1}[\psi_i\va],\bigg(\vb^2\dfrac{\p
\g}{\p\va}-\va\vb\dfrac{\p
\g}{\p\vb}\bigg)}=-\br{\ll^{-1}[\psi_i\va],\ll[{w}]}.
\end{eqnarray}
Since $\ll$ is self-adjoint, combining with the orthogonal properties, we have
\begin{eqnarray}
\br{\ll^{-1}[\psi_i\va],\ll[{w}]}(\eta)=\br{\ll\bigg[\ll^{-1}[\psi_i\va]\bigg],{w}}(\eta)=\br{\psi_i\va,{w}}(\eta)=0.
\end{eqnarray}
Therefore, we have
\begin{eqnarray}
&&\frac{\ud{}}{\ud{\eta}}\br{\va\ll^{-1}[\psi_i\va],\g}+G(\eta)\br{\ll^{-1}[\psi_i\va],\bigg(\vb^2\dfrac{\p
\g}{\p\va}-\va\vb\dfrac{\p \g}{\p\vb}\bigg)}=0.
\end{eqnarray}
We may integrate by parts to obtain
\begin{eqnarray}\label{mt 13}
&&\frac{\ud{}}{\ud{\eta}}\br{\va\ll^{-1}[\psi_i\va],\g}-G(\eta)\br{\bigg(\vb^2\dfrac{\p
}{\p\va}-\va\vb\dfrac{\p}{\p\vb}-\va\bigg)\ll^{-1}[\psi_i\va],g}=0.
\end{eqnarray}
Since $\va\psi_0=\psi_1\in\nk$ and $\ll^{-1}[\psi_i\va]\in\nk^{\bot}$, we
have
\begin{eqnarray}
\br{\va\ll^{-1}[\psi_i\va],\psi_0}=0.
\end{eqnarray}
For $i,k=2,3$, put
\begin{eqnarray}
N_{i,k}&=&\br{\va\ll^{-1}[\psi_i\va],\psi_k},\\
P_{i,k}&=&\br{\bigg(\vb^2\dfrac{\p }{\p\va}-\va\vb\dfrac{\p
}{\p\vb}-\va\bigg)\ll^{-1}[\psi_i\va],\psi_k}.
\end{eqnarray}
Thus,
\begin{eqnarray}
\Omega_i&=&\br{\va\ll^{-1}[\psi_i\va],{q}}=\sum_{k=2}^3N_{i,k}{q}_k(\eta),
\end{eqnarray}
and
\begin{eqnarray}
\br{\bigg(\vb^2\dfrac{\p }{\p\va}-\va\vb\dfrac{\p
}{\p\vb}-\va\bigg)\ll^{-1}[\psi_i\va],{q}}&=&\sum_{k=2}^3P_{i,k}{q}_k(\eta).
\end{eqnarray}
Since matrix $N$ is invertible (see \cite{Golse.Poupaud1989}), from (\ref{mt 13}) and integration
by parts, we have for $i=2,3$,
\begin{eqnarray}
\frac{\ud{\Omega_i}}{\ud{\eta}}&=&-\frac{\ud{}}{\ud{\eta}}\br{\va\ll^{-1}[\psi_i\va],{w}}\\
&&+\sum_{k=2}^3G(\eta)(PN^{-1})_{ik}\Omega_k+G(\eta)\br{\bigg(\vb^2\dfrac{\p
}{\p\va}-\va\vb\dfrac{\p
}{\p\vb}-\va\bigg)\ll^{-1}[\hat\psi_i\va],{w}}.\no
\end{eqnarray}
Denote
\begin{eqnarray}
\hat\Omega=\exp\bigg(W(\eta)PN^{-1}\bigg)\Omega.
\end{eqnarray}
This is an ordinary differential equation for $\Omega$. Let $\hat\psi=(\psi_2,\psi_3)^T$, we can solve
\begin{eqnarray}
\hat\Omega(\eta)
&=&\br{\va\ll^{-1}[\hat\psi\va],\g}(0)-\exp\bigg(W(\eta)PN^{-1}\bigg)\br{\va\ll^{-1}[\hat\psi\va],{w}}(\eta)\\
&&+\int_0^{\eta}\exp\bigg(W(y)PN^{-1}\bigg)G(y)\Bigg(\br{\bigg(\vb^2\dfrac{\p
}{\p\va}-\va\vb\dfrac{\p
}{\p\vb}-\va\bigg)\ll^{-1}[\hat\psi\va],{w}}(y)\no\\
&&+\sum_{k=2}^3PN^{-1}\br{\va\ll^{-1}[\hat\psi\va],{w}}(y)\Bigg)\ud{y}.\no
\end{eqnarray}
By a similar method as in Step 2 to bound $\theta_i(0)$, we can show
\begin{eqnarray}
\br{\va\ll^{-1}[\hat\psi\va],\g}(0)=\br{\hat\psi\va,\ll[\va\g]}(0)\leq C.
\end{eqnarray}
Since ${w}\in L^2([0,L]\times\r^2)$,
considering $W(\eta)$ and $PN^{-1}$ are bounded, and $G(\eta)\in L^{\infty}$,
we define
\begin{eqnarray}
\\
\hat\Omega_{L}&=&\br{\va\ll^{-1}[\hat\psi\va],\g}(0)
+\int_0^{L}\exp\bigg(W(y)PN^{-1}\bigg)G(y)\Bigg(\br{\bigg(\vb^2\dfrac{\p
}{\p\va}-\va\vb\dfrac{\p
}{\p\vb}-\va\bigg)\ll^{-1}[\hat\psi\va],{w}}(y)\no\\
&&+\sum_{k=2}^3PN^{-1}\br{\va\ll^{-1}[\hat\psi\va],{w}}(y)\Bigg)\ud{y}.\no
\end{eqnarray}
Let $\hat q_L=(q_{2,L},q_{3,L})^T$. Then we can define
\begin{eqnarray}
\hat q_{L}=N^{-1}\exp\bigg(-W(L)PN^{-1}\bigg)\hat\Omega_{L}.
\end{eqnarray}
Finally, we consider $q_{0,L}$. Multiplying $\psi_1$ on both sides of
(\ref{Milne}) and integrating over $\vvv\in\r^2$, we obtain
\begin{eqnarray}
\frac{\ud{}}{\ud{\eta}}\br{\va\psi_1,\g}=-G\br{\psi_1,\bigg(\vb^2\dfrac{\p
\g}{\p\va}-\va\vb\dfrac{\p
\g}{\p\vb}\bigg)}=G\br{g,(\vb^2-\va^2)\m^{\frac{1}{2}}}.
\end{eqnarray}
Then integrating over $[0,\eta]$, we obtain
\begin{eqnarray}\label{mt 14}
\br{\va\psi_1,\g}(\eta)&=&\br{\va\psi_1,\g}(0)+\int_0^{\eta}G(y)\br{g,(\vb^2-\va^2)\m^{\frac{1}{2}}}(y)\ud{y}\\
&=&\br{\psi_1\g,\va}(0)+\int_0^{\eta}G(y)\br{w,(\vb^2-\va^2)\m^{\frac{1}{2}}}(y)\ud{y}.\no
\end{eqnarray}
Since ${w}\in L^2([0,L]\times\r^2)$ and we can also bound $\br{\va\g,\va}(0)$, we have
\begin{eqnarray}
\br{\va\psi_1,\g}(L)&=&\br{\psi_1\g,\va}(0)+\int_0^{\eta}G(y)\br{w,(\vb^2-\va^2)\m^{\frac{1}{2}}}(y)\ud{y}\leq C.
\end{eqnarray}
Note that
\begin{eqnarray}
\br{\va\psi_1,\psi_1}(\eta)=\br{\va\psi_1,\psi_2}(\eta)=0.
\end{eqnarray}
Then we define
\begin{eqnarray}
{q}_{0,L}=\frac{\br{\va\psi_1,\g}(L)-{q}_{3,L}\br{\va\psi_1,\psi_3}}{\br{\va\psi_1,\psi_0}}.
\end{eqnarray}
Naturally, we define $q_{1,L}=0$. Then to summarize all above, we have defined
\begin{eqnarray}
{q}_{L}=q_{0,L}\psi_0+q_{1,L}\psi_1+q_{2,L}\psi_2+q_{3,L}\psi_3,
\end{eqnarray}
which satisfies $\abs{{q}_{i,L}}\leq C$ for $i=0,1,2,3$.\\
\ \\
Step 5: $L^2$ Decay of ${w}$.\\
The orthogonal property and $q_1=0$ imply
\begin{eqnarray}
\br{\va{q},{w}}(\eta)=\sum_{k=0}^3\br{\va\psi_k,{w}}(\eta)=0.
\end{eqnarray}
Also, we may directly verify that
\begin{eqnarray}
\br{\va{q},{q}}(\eta)=0.
\end{eqnarray}
Therefore, we deduce that
\begin{eqnarray}
\br{\va\g,\g}(\eta)=\br{\va{w},{w}}(\eta).
\end{eqnarray}
Multiplying $\ue^{2K_0\eta}\g$ on both sides of (\ref{Milne}) and integrating over $\vvv\in\r^2$, we obtain

\begin{eqnarray}
\half\frac{\ud{}}{\ud{\eta}}\br{\va\g,\ue^{2K_0\eta}\g}+
G(\eta)\br{\vb^2\dfrac{\p
\g}{\p\va}-\va\vb\dfrac{\p \g}{\p\vb},\ue^{2K_0\eta}\g}=K_0\ue^{2K_0\eta}\br{\va{w},{w}}-\ue^{2K_0\eta}\br{\g,\ll[\g]}.
\end{eqnarray}
An integration by parts implies
\begin{eqnarray}
\br{\vb^2\dfrac{\p
\g}{\p\va}-\va\vb\dfrac{\p \g}{\p\vb},\ue^{2K_0\eta}\g}=\half\br{\vb^2,\dfrac{\p
(\ue^{2K_0\eta}\g^2)}{\p\va}}-\half\br{\va\vb,\dfrac{\p (\ue^{2K_0\eta}\g^2)}{\p\vb}}=\half\br{\va\g,\ue^{2K_0\eta}\g}.
\end{eqnarray}
Therefore, using Lemma \ref{Milne property}, we obtain
\begin{eqnarray}
\half\frac{\ud{}}{\ud{\eta}}\br{\va\g,\ue^{2K_0\eta}\g}+\half
G(\eta)\br{\va\g,\ue^{2K_0\eta}\g}=K_0\ue^{2K_0\eta}\br{\va{w},{w}}-\ue^{2K_0\eta}\br{\g,\ll[\g]}.
\end{eqnarray}
Hence, we can rewrite it as
\begin{eqnarray}
\half\frac{\ud{}}{\ud{\eta}}\bigg(\ue^{2K_0\eta+W(\eta)}\br{\va{w},{w}}\bigg)-\ue^{2K_0\eta+W(\eta)}\bigg(K_0\br{\va{w},{w}}
-\br{{w},\ll[w]}\bigg)=0.
\end{eqnarray}
Since
\begin{eqnarray}
\br{\ll[{w}],{w}}\geq \br{\nu{w},{w}},
\end{eqnarray}
and $W(\eta)$ is bounded,
for $K_0$ sufficiently small, we have
\begin{eqnarray}
\br{\ll[{w}],{w}}-K_0\br{\va{w},{w}}\geq C\tnm{\sn{w}}.
\end{eqnarray}
Then by a similar argument as in Step 1,
we can show that
\begin{eqnarray}
\int_0^{L}\ue^{2K_0\eta}\tnm{\sn{w}(\eta)}\ud{\eta}\leq C.
\end{eqnarray}
\ \\
Step 6: Estimate of ${q}-{q}_{L}$.\\
We first consider $\hat{q}=({q}_2,{q}_3)^T$, which satisfies
\begin{eqnarray}
\hat{q}(\eta)=N^{-1}\exp\bigg(-W(\eta)PN^{-1}\bigg)\hat\Omega(\eta),
\end{eqnarray}
Let
\begin{eqnarray}
\delta&=&\br{\va\ll^{-1}[\hat\psi\va],\g}(0)\\
\Delta&=&G\Bigg(\br{\bigg(\vb^2\dfrac{\p
}{\p\va}-\va\vb\dfrac{\p
}{\p\vb}-\va\bigg)\ll^{-1}[\hat\psi\va],{w}}
+\sum_{k=2}^3PN^{-1}\br{\va\ll^{-1}[\hat\psi\va],{w}}\Bigg)
\end{eqnarray}
Then we have
\begin{eqnarray}
\hat{q}(\eta)
&=&N^{-1}\exp\bigg(-W(\eta)PN^{-1}\bigg)\delta-N^{-1}\br{\va\ll^{-1}[\hat\psi\va],{w}}(\eta)\\
&&+\int_0^{\eta}\exp\bigg(\Big(W(y)-W(\eta)\Big)PN^{-1}\bigg)\Delta(y)\ud{y}.\no
\end{eqnarray}
Also, we know
\begin{eqnarray}
\hat{q}_{L}&=&N^{-1}\exp\bigg(-W(L)PN^{-1}\bigg)\delta+\int_0^{L}\exp\bigg(\Big(W(y)-W(\infty)\Big)PN^{-1}\bigg)\Delta(y)\ud{y}.
\end{eqnarray}
Then we have
\begin{eqnarray}
\hat{q}(\eta)-\hat{q}_{L}&=&N^{-1}\Bigg(\exp\bigg(-W(\eta)PN^{-1}\bigg)-\exp\bigg(-W(L)PN^{-1}\bigg)\Bigg)\delta-N^{-1}\br{\va\ll^{-1}[\hat\psi\va],{w}}(\eta)\\
&&+N^{-1}\Bigg(\exp\bigg(-W(\eta)PN^{-1}\bigg)-\exp\bigg(-W(L)PN^{-1}\bigg)\Bigg)\int_0^{L}\exp\bigg(W(y)PN^{-1}\bigg)\Delta(y)\ud{y}\no\\
&&+N^{-1}\int_{\eta}^{L}\exp\bigg(\Big(W(y)-W(\eta)\Big)PN^{-1}\bigg)\Delta(y)\ud{y}.\no
\end{eqnarray}
Then we have
\begin{eqnarray}\label{mt 15}
&&\tnnm{\hat{q}-\hat{q}_{L}}\\
&\leq&\tnnm{N^{-1}\Bigg(\exp\bigg(-W(\eta)PN^{-1}\bigg)-\exp\bigg(-W(L)PN^{-1}\bigg)\Bigg)\delta}
+\tnnm{N^{-1}\br{\va\ll^{-1}[\hat\psi\va],{w}}}\no\\
&&+\tnnm{N^{-1}\Bigg(\exp\bigg(-W(\eta)PN^{-1}\bigg)-\exp\bigg(-W(L)PN^{-1}\bigg)\Bigg)\int_0^{L}\exp\bigg(W(y)PN^{-1}\bigg)\Delta(y)\ud{y}}\no\\
&&+\tnnm{N^{-1}\int_{\eta}^{L}\exp\bigg(\Big(W(y)-W(\eta)\Big)PN^{-1}\bigg)\Delta(y)\ud{y}}.\no
\end{eqnarray}
We need to estimate each term on the right-hand side of (\ref{mt 15}). We have
\begin{eqnarray}
&&\tnnm{N^{-1}\Bigg(\exp\bigg(-W(\eta)PN^{-1}\bigg)-\exp\bigg(-W(L)PN^{-1}\bigg)\Bigg)\delta}^2\\
&\leq&C\delta\tnnm{\ue^{-W(\eta)}-\ue^{-W(L)}}\leq C.\no
\end{eqnarray}
Since ${w}\in L^2([0,L]\times\r^2)$, we have
\begin{eqnarray}
\tnnm{N^{-1}\br{\va\ll^{-1}[\hat\psi\va],{w}}}\leq C\tnnm{w}\leq C.
\end{eqnarray}
Similarly, we can show
\begin{eqnarray}
&&\tnnm{N^{-1}\Bigg(\exp\bigg(-W(\eta)PN^{-1}\bigg)-\exp\bigg(-W(L)PN^{-1}\bigg)\Bigg)\int_0^{L}\exp\bigg(W(y)PN^{-1}\bigg)\Delta(y)\ud{y}}\\
&\leq&C\tnnm{\ue^{-W(\eta)}-\ue^{-W(L)}}\tnnm{r}\leq C.\no
\end{eqnarray}
For the last term, we have to resort to the exponential decay of ${w}$ in Step 5. We estimate
\begin{eqnarray}
&&\tnnm{N^{-1}\int_{\eta}^{L}\exp\bigg(\Big(W(y)-W(\eta)\Big)PN^{-1}\bigg)\Delta(y)\ud{y}}\\
&\leq&C\int_0^{L}\bigg(\int_{\eta}^{L}\Delta(y)\ud{y}\bigg)^2\ud{\eta}\leq \int_0^{L}\bigg(\int_{\eta}^{\infty}\ue^{-2K_0y}\ud{y}\bigg)\bigg(\int_{\eta}^{L}{w}^2(y)\ue^{2K_0y}\ud{y}\bigg)\ud{\eta}\no\\
&\leq&\int_0^{L}C\ue^{-2K_0\eta}\ud{\eta}\leq C.\no
\end{eqnarray}
Collecting all above, we have
\begin{eqnarray}
\tnnm{\hat{q}-\hat{q}_{L}}\leq C.
\end{eqnarray}
Then we turn to ${q}_0$. We have
\begin{eqnarray}
{q}_{0}(\eta)=\frac{\br{\va\psi_1,\g}(\eta)-\br{\va\psi_1,w}(\eta)-{q}_{3}(\eta)\br{\va\psi_1,\psi_3}}{\br{\va\psi_1,\psi_1}},
\end{eqnarray}
where
\begin{eqnarray}
\br{\psi_1\g,\va}(\eta)&=&\br{\psi_1\g,\va}(0)+\int_0^{\eta}G(y)\br{w,(\vb^2-\va^2)\m^{\frac{1}{2}}}(y)\ud{y}.
\end{eqnarray}
Also, we have
\begin{eqnarray}
{q}_{0,L}=\frac{\br{\va\psi_1,\g}(L)-{q}_{3,L}\br{\va\psi_1,\psi_3}}{\br{\va\psi_1,\psi_0}}.
\end{eqnarray}
Therefore, we have
\begin{eqnarray}
{q}_0(\eta)-{q}_{0,L}=\frac{\displaystyle\int_{\eta}^{L}G(y)\br{w,(\vb^2-\va^2)\m^{\frac{1}{2}}}(y)\ud{y}-\br{\va\psi_1,w}(\eta)
-({q}_3(\eta)-{q}_{3,L})\br{\va\psi_1,\psi_3}}{\br{\va\psi_0,\va}}
\end{eqnarray}
Then we can naturally estimate
\begin{eqnarray}
\\
\tnnm{{q}_0-{q}_{0,L}}\leq C\left(\tnnm{\int_{\eta}^{L}G(y)\br{w,(\vb^2-\va^2)\m^{\frac{1}{2}}}(y)\ud{y}}+\tnnm{w}+\tnnm{{q}_3(\eta)-{q}_{3,L}}\right).\no
\end{eqnarray}
$\tnnm{{q}_3(\eta)-{q}_{3,L}}$ is bounded due to the estimate of $\tnnm{\hat{q}(\eta)-\hat{q}_{L}}$. Then by Cauchy's inequality, we obtain
\begin{eqnarray}
\tnnm{\int_{\eta}^{L}G(y)\br{w,(\vb^2-\va^2)\m^{\frac{1}{2}}}(y)\ud{y}}&\leq& \int_0^{L}\bigg(\int_{\eta}^{L}G(y)\tnm{w(y)}\ud{y}\bigg)^2\ud{\eta}\\
&\leq&\tnnm{w}\int_0^{L}\int_{\eta}^{L}G^2(y)\ud{y}\ud{\eta}\leq C.\no
\end{eqnarray}
Therefore, we have shown
\begin{eqnarray}
\tnnm{{q}_0-{q}_{0,L}}\leq C.
\end{eqnarray}
In summary, we prove that
\begin{eqnarray}
\tnnm{{q}-{q}_{L}}\leq C.
\end{eqnarray}
\end{proof}

\begin{lemma}\label{Milne lemma 2}
There exists a unique solution $\g(\eta,\vvv)$ to the $\e$-Milne problem with geometric correction
(\ref{Milne}) satisfying
\begin{eqnarray}
\tnnm{\g-\g_{L}}\leq C,
\end{eqnarray}
for some $g_L\in\nk$ satisfying $\abs{g_L}\leq C$.
\end{lemma}
\begin{proof}
Taking $\g_{L}={q}_{L}$, we can naturally obtain the
desired result.
\end{proof}

Then we turn to the construction of $\tilde h$ and the well-posedness of the equation (\ref{Milne transform}).
\begin{theorem}\label{Milne theorem 1}
There
exists $\tilde h$ satisfying the condition (\ref{Milne transform
compatibility}) such that there exists a unique solution
$\gg(\eta,\vvv)$ to the $\e$-Milne problem (\ref{Milne transform}) with geometric correction
satisfying
\begin{eqnarray}
\tnnm{\gg}\leq C.
\end{eqnarray}
\end{theorem}
\begin{proof}
We want to find $\gg$ such that $\gg_L=0$. The key part is the construction of $\tilde h$. Our main idea is to find $\tilde h\in\nk$ such that the equation
\begin{eqnarray}
\left\{
\begin{array}{l}\displaystyle
\va\frac{\p \tilde\g}{\p\eta}+G(\eta)\bigg(\vb^2\dfrac{\p
\tilde\g}{\p\va}-\va\vb\dfrac{\p
\tilde\g}{\p\vb}\bigg)+\ll[\tilde\g]
=0,\\\rule{0ex}{1.5em}
\tilde\g(0,\vvv)=\tilde h(\vvv)\ \ \text{for}\ \ \va>0,\\\rule{0ex}{1.5em}
\tilde\g(L,\vvv)=\tilde\g(L,\rr[\vvv]),\\\rule{0ex}{1.5em}
q_{\tilde\g,L}=q_{g,L},
\end{array}
\right.
\end{eqnarray}
for $\tilde g(\eta,\vvv)$ is well-posed, where
\begin{eqnarray}
\tilde\g_{L}(\vvv)=\g_{L}(\vvv)={q}_{0,L}\psi_0+{q}_{1,L}\psi_1+{q}_{2,L}\psi_2+{q}_{3,L}\psi_3,
\end{eqnarray}
is given by the equation (\ref{Milne}) for $\g$. Note that
\begin{eqnarray}
\tilde h(\vvv)=\tilde D_0\psi_0+\tilde D_1\psi_1+\tilde
D_2\psi_2+\tilde D_3\psi_3,
\end{eqnarray}
with $\tilde D_1=0$. We consider the endomorphism $\mathcal{T}$ in $\tilde\nk\{\psi_0,\psi_2,\psi_3\}$ defined as $\mathcal{T}:\tilde h\rt \mathcal{T}[\tilde h]=\tilde g_{L}$. Therefore, we only need to study the matrix of $\mathcal{T}$ at the basis $\{\psi_0,\psi_2,\psi_3\}$.
It is easy to check when $\tilde
h=\psi_0$ and $\tilde h=\psi_3$, $\mathcal{T}$ is an identity mapping, i.e.
\begin{eqnarray}
\mathcal{T}[\psi_0]&=&\psi_0\\
\mathcal{T}[\psi_3]&=&\psi_3
\end{eqnarray}
The main obstacle is when $\tilde h=\psi_2$. In this case, define $\tilde\g'=\tilde\g-\psi_2$. Then $\tilde\g'$ satisfies the equation
\begin{eqnarray}
\left\{
\begin{array}{l}\displaystyle
\va\frac{\p \tilde\g'}{\p\eta}+G(\eta)\bigg(\vb^2\dfrac{\p \tilde\g'}{\p\va}-\va\vb\dfrac{\p \tilde\g'}{\p\vb}\bigg)+\ll[\tilde\g']
=G(\eta)\m^{\frac{1}{2}}\va\vb,\\\rule{0ex}{1.5em}
\tilde\g'(0,\vvv)=0\ \ \text{for}\ \ \va>0,\\\rule{0ex}{1.5em}
\tilde\g'(L,\vvv)=\tilde\g'(L,\rr[\vvv]).
\end{array}
\right.
\end{eqnarray}
Although $G(\eta)\m^{\frac{1}{2}}\va\vb$ does not decay exponentially, it is easy to check that the $L^1$ and $L^2$ norm of
$G$ can be sufficiently small as $\e\rt0$ and $G(\eta)\m^{\frac{1}{2}}\va\vb\in\nk$. Using a natural extension of Lemma \ref{Milne lemma 1}, we know
$\abs{\tilde{q}'_{L}}$ is also sufficiently small, where $\tilde{q}'\in\nk$ is defined in the decomposition $\tilde\g'=\tilde w'+\tilde q'$. Note that we do not need exponential decay of source term in order to show the bound of $\tilde{q}'_{L}$. This means that
\begin{eqnarray}
\mathcal{T}[\psi_0,\psi_2,\psi_3]=[\psi_0,\psi_2,\psi_3]\left(
\begin{array}{ccc}
1&\tilde{q}'_{0,L}&0\\
0&1+\tilde{q}'_{2,L}&0\\
0&\tilde{q}'_{3,L}&1\\
\end{array}
\right)
\end{eqnarray}
For $\e$ sufficiently small, this matrix is invertible, which means $\mathcal{T}$ is bijective. Therefore, we can always find $\tilde h$ such that
$\tilde\g_{L}=\g_{L}$, which is desired. Then by Lemma
\ref{Milne lemma 2} and the superposition property, when define $\gg=\g-\tilde\g$, the theorem naturally follows.
\end{proof}

\subsection{$L^{\infty}$ Estimates}

Consider the $\e$-transport problem for $\g(\eta,\vvv)$
\begin{eqnarray}\label{transport}
\left\{
\begin{array}{l}\displaystyle
\va\frac{\p \g}{\p\eta}+G(\eta)\bigg(\vb^2\dfrac{\p
\g}{\p\va}-\va\vb\dfrac{\p \g}{\p\vb}\bigg)+\nu\g
=Q(\eta,\vvv),\\\rule{0ex}{1.5em} \g(0,\vvv)=h(\vvv)\ \
\text{for}\ \ \va>0,\\\rule{0ex}{1.5em}
\g(L,\vvv)=\g(L,R[\vvv]),
\end{array}
\right.
\end{eqnarray}
We define the characteristics starting from
$\Big(\eta(0),\va(0),\vb(0)\Big)$ as $\Big(\eta(s),\va(s),\vb(s)\Big)$
satisfying
\begin{eqnarray}
\frac{\ud{\eta}}{\ud{s}}=\va,\ \
\frac{\ud{\va}}{\ud{s}}=G(\eta)\vb^2,\ \
\frac{\ud{\vb}}{\ud{s}}=-G(\eta)\va\vb,
\end{eqnarray}
which leads to
\begin{eqnarray}
\va^2(s)+\vb^2(s)&=&C_1,\\
\vb(s)\ue^{-W(\eta(s))}&=&C_2,
\end{eqnarray}
where $C_1$ and $C_2$ are two constants depending on the starting
point. Along the characteristics, $\va^2+\vb^2$ and $\vb\ue^{-W(\eta)}$ are conserved quantities and the equation (\ref{transport}) can be
rewritten as
\begin{eqnarray}
\va\frac{\p \g}{\p\eta}+\nu\g&=&Q.
\end{eqnarray}
Define the energy
\begin{eqnarray}
E_1&=&\va^2+\vb^2,\\
E_2&=&\vb\ue^{-W(\eta)}.
\end{eqnarray}
Let
\begin{eqnarray}
\vb'(\eta,\vvv;\eta')&=&\vb e^{W(\eta')-W(\eta)}.
\end{eqnarray}
For $E_1\geq\vb'^2$, define
\begin{eqnarray}
\va'(\eta,\vvv;\eta')&=&\sqrt{E_1-\vb'^2(\eta,\vvv;\eta')},\\
\vvv'(\eta,\vvv;\eta')&=&\Big(\va'(\eta,\vvv;\eta'),\vb'(\eta,\vvv;\eta')\Big),\\
\rr[\vvv'(\eta,\vvv;\eta')]&=&\Big(-\va'(\eta,\vvv;\eta'),\vb'(\eta,\vvv;\eta')\Big).
\end{eqnarray}
Basically, this means $(\eta,\va,\vb)$ and $(\eta',\va',\vb')$, $(\eta',-\va',\vb')$ are on the same characteristics. Also, this implies $\va'\geq0$.
Moreover, define an implicit function $\eta^{+}(\eta,\vvv)$ by the
equation
\begin{eqnarray}
E_1(\eta,\vvv)=\vb'^2(\eta,\vvv;\eta^+).
\end{eqnarray}
We know $(\eta^+,0,\vb')$ at the axis $\va=0$ is on the same characteristics as $(\eta,\vvv)$. Finally put
\begin{eqnarray}
H_{\eta,\eta'}&=&\int_{\eta'}^{\eta}\frac{\nu\Big(\vvv'(\eta,\vvv;y)\Big)}{\va'(\eta,\vvv;y)}\ud{y},\\
\rr[H_{\eta,\eta'}]&=&\int_{\eta'}^{\eta}\frac{\nu\Big(\rr[\vvv'(\eta,\vvv;y)]\Big)}{\va'(\eta,\vvv;y)}\ud{y}.
\end{eqnarray}
Actually, since $\nu$ only depends on $\abs{\vvv}$, we must have $H_{\eta,\eta'}=\rr[H_{\eta,\eta'}]$. This distinction is purely for clarify and does not play a role in the estimates.
We can rewrite the solution to the equation (\ref{transport})
along the characteristics
as
\begin{eqnarray}
\g(\eta,\vvv)=\k\Big[h(\vvv)\Big]+\t\Big[Q(\eta,\vvv)\Big],
\end{eqnarray}
where\\
\ \\
Case I:\\
For $\va>0$,
\begin{eqnarray}\label{mt 16}
\k\Big[h(\vvv)\Big]&=&h\Big(\vvv'(\eta,\vvv; 0)\Big)\exp(-H_{\eta,0}),\\
\t\Big[Q(\eta,\vvv)\Big]&=&\int_0^{\eta}\frac{Q\Big(\eta',\vvv'(\eta,\vvv;\eta')\Big)}{\va'(\eta,\vvv;\eta')}\exp(-H_{\eta,\eta'})\ud{\eta'}.
\end{eqnarray}
\ \\
Case II:\\
For $\va<0$ and $\va^2+\vb^2\geq \vb'^2(\eta,\vvv;L)$,
\begin{eqnarray}\label{mt 17}
\k\Big[h(\vvv)\Big]&=&h\Big(\vvv'(\eta,\vvv; 0)\Big)\exp(-H_{L,0}-\rr[H_{L,\eta}]),\\
\t\Big[Q(\eta,\vvv)\Big]&=&\bigg(\int_0^{L}\frac{Q\Big(\eta',\vvv'(\eta,\vvv;\eta')\Big)}{\va'(\eta,\vvv;\eta')}
\exp(-H_{L,\eta'}-\rr[H_{L,\eta}])\ud{\eta'}\\
&&+\int_{\eta}^{L}\frac{Q\Big(\eta',\rr[\vvv'(\eta,\vvv;\eta')]\Big)}{\va'(\eta,\vvv;\eta')}\exp(\rr[H_{\eta,\eta'}])\ud{\eta'}\bigg).\no
\end{eqnarray}
\ \\
Case III:\\
For $\va<0$ and $\va^2+\vb^2\leq \vb'^2(\eta,\vvv;L)$,
\begin{eqnarray}\label{mt 18}
\k\Big[h(\vvv)\Big]&=&h\Big(\vvv'(\eta,\vvv; 0)\Big)\exp(-H_{\eta^+,0}-\rr[H_{\eta^+,\eta}]),\\
\t\Big[Q(\eta,\vvv)\Big]&=&\bigg(\int_0^{\eta^+}\frac{Q\Big(\eta',\vvv'(\eta,\vvv;\eta')\Big)}{\va'(\eta,\vvv;\eta')}
\exp(-H_{\eta^+,\eta'}-\rr[H_{\eta^+,\eta}])\ud{\eta'}\\
&&+\int_{\eta}^{\eta^+}\frac{Q\Big(\eta',\rr[\vvv'(\eta,\vvv;\eta')]\Big)}{\va'(\eta,\vvv;\eta')}\exp(\rr[H_{\eta,\eta'}])\ud{\eta'}\bigg).\no
\end{eqnarray}
In order to achieve the estimate of $\g$, we need to control $\k[h]$ and $\t[Q]$. Since we always assume that $(\eta,\vvv)$ and $(\eta',\vvv')$ are on the same characteristics, in the following, we will simply write $\vvv'(\eta')$ or even $\vvv'$ instead of $\vvv'(\eta,\vvv;\eta')$ when there is no confusion. We can use this notation interchangeably when necessary.
\begin{lemma}\label{Milne lemma 3}
There is a positive $0<\beta<\nu_0$ such that for any $\vth\geq0$ and
$\varrho\geq0$,
\begin{eqnarray}
\lnm{\ue^{\beta\eta}\k[h]}{\vth,\varrho}\leq C\lnm{h}{\vth,\varrho}.
\end{eqnarray}
\end{lemma}
\begin{proof}
Based on Lemma \ref{Milne property},
we know
\begin{eqnarray}
\frac{\nu(\vvv')}{\va'}\geq\nu_0,\ \
\frac{\nu(\rr[\vvv'])}{\va'}\geq\nu_0.
\end{eqnarray}
It follows that
\begin{eqnarray}
\exp(-H_{\eta,0})&\leq&\ue^{-\beta\eta}\\
\exp(-H_{\eta^+,0}-\rr[H_{\eta^+,\eta}])&\leq&\ue^{-\beta\eta}
\end{eqnarray}
Then our results are obvious.
\end{proof}
\begin{lemma}\label{Milne lemma 4}
For any $\vth\geq0$, $\varrho\geq0$ and
$0\leq\beta\leq\dfrac{\nu_0}{2}$, there is a constant $C$ such that
\begin{eqnarray}
\lnnm{\t[Q]}{\vth,\varrho}\leq C\lnnm{\frac{Q}{\nu}}{\vth,\varrho}.
\end{eqnarray}
Moreover, we have
\begin{eqnarray}
\lnnm{\ue^{\beta\eta}\t[Q]}{\vth,\varrho}\leq
C\lnnm{\frac{\ue^{\beta\eta}Q}{\nu}}{\vth,\varrho}.
\end{eqnarray}
\end{lemma}
\begin{proof}
The first inequality is a special case of the second one, so we only need to prove the second inequality.
For $\va>0$ case, we have
\begin{eqnarray}
\beta(\eta-\eta')-H_{\eta,\eta'}\leq\beta(\eta-\eta')-\frac{\nu_0(\eta-\eta')}{2}-\frac{H_{\eta,\eta'}}{2}\leq-\frac{H_{\eta,\eta'}}{2}.
\end{eqnarray}
It is natural that
\begin{eqnarray}
\int_0^{\eta}\frac{\nu\Big(\vvv'(\eta')\Big)}{\va'(\eta')}
\exp\Big(\beta(\eta-\eta')-H_{\eta,\eta'}\Big)\ud{\eta'}\leq\int_0^{\infty}\exp\bigg(-\frac{z}{2}\bigg)\ud{z}=2,
\end{eqnarray}
for $z=H_{\eta,\eta'}$. Then we estimate
\begin{eqnarray}
\abs{\bvv\ue^{\beta\eta}\t[Q]}&\leq& \ue^{\beta\eta}\int_0^{\eta}\bvv\frac{\abs{Q\Big(\eta',\vvv'(\eta')\Big)}}{\va'(\eta')}\exp(-H_{\eta,\eta'})\ud{\eta'}\\
&\leq&\lnnm{\frac{\ue^{\beta\eta}Q}{\nu}}{\vth,\varrho}\int_0^{\eta}\frac{\nu\Big(\vvv'(\eta')\Big)}{\va'(\eta')}
\exp\Big(\beta(\eta-\eta')-H_{\eta,\eta'}\Big)\ud{\eta'}\no\\
&\leq&C\lnnm{\frac{\ue^{\beta\eta}Q}{\nu}}{\vth,\varrho}.\no
\end{eqnarray}
The $\va<0$ case can be proved in a similar fashion, so we omit it here.
\end{proof}
\begin{lemma}\label{Milne lemma 5}
For any $\d>0$, $\vth>2$ and $\varrho\geq0$, there is a
constant $C(\d)$ such that
\begin{eqnarray}
\ltnm{\t[Q]}{\varrho}\leq C(\d)\tnnm{\nu^{-\frac{1}{2}}Q}+\d\lnnm{Q}{\vth,\varrho}.
\end{eqnarray}
\end{lemma}
\begin{proof}
We divide the proof into several cases:\\
\ \\
Case I: For $\va>0$,
\begin{eqnarray}
\t\Big[Q(\eta,\vvv)\Big]&=&\int_0^{\eta}\frac{Q\Big(\eta',\vvv(\eta,\vvv;\eta')\Big)}{\va'(\eta,\vvv;\eta')}\exp(-H_{\eta,\eta'})\ud{\eta'}.
\end{eqnarray}
We need to estimate
\begin{eqnarray}
\int_{\r^2}\ue^{2\varrho\abs{\vvv}^2}\bigg(\int_0^{\eta}\frac{Q\Big(\eta',\vvv(\eta,\vvv;\eta')\Big)}{\va'(\eta,\vvv;\eta')}\exp(-H_{\eta,\eta'})\ud{\eta'}
\bigg)^2\ud{\vvv}.
\end{eqnarray}
Assume $m>0$ is sufficiently small, $M>0$ is sufficiently large and $\sigma>0$ is sufficiently small, which will be determined in the following. We can split the integral into four parts
\begin{eqnarray}
I=I_1+I_2+I_3+I_4.
\end{eqnarray}
In the following, we use $\chi_i$ for $i=1,2,3,4$ to represent the indicator function of each case.\\
\ \\
Case I - Type I: $\chi_1$: $M\leq\va'(\eta,\vvv;\eta')$ or $M\leq\vb'(\eta,\vvv;\eta')$.\\
By Lemma \ref{Milne property}, we have
\begin{eqnarray}
\abs{\vvv(\eta')}+1\leq C\nu\Big(\vvv(\eta')\Big).
\end{eqnarray}
Then for $\vth>2$, since $\abs{\vvv}$ is conserved along the characteristics,  we have
\begin{eqnarray}
I_1&\leq&C\lnnm{Q}{\vth,\varrho}^2\int_{\r^2}\chi_1\bigg(\int_0^{\eta}\frac{1}{\br{\vvv'}^{\vth}}\frac{\exp(-H_{\eta,\eta'})}{\va'(\eta,\vvv;\eta')}\ud{\eta'}
\bigg)^2\ud{\vvv}\\
&\leq&\frac{C}{M^{\vth}}\lnnm{Q}{\vth,\varrho}^2\int_{\r^2}\frac{1}{\br{\vvv}^{\vth}}\bigg(\int_0^{\eta}\frac{\exp(-H_{\eta,\eta'})}{\va'(\eta,\vvv;\eta')}\ud{\eta'}
\bigg)^2\ud{\vvv}\no\\
&\leq&\frac{C}{M^{\vth}}\lnnm{Q}{\vth,\varrho}^2\int_{\r^2}\frac{1}{\br{\vvv}^{\vth}}\ud{\vvv}\no\\
&\leq&\frac{C}{M^{\vth}}\lnnm{Q}{\vth,\varrho}^2,\no
\end{eqnarray}
since for $y=H_{\eta,\eta'}$,
\begin{eqnarray}
\abs{\int_0^{\eta}\frac{\exp(-H_{\eta,\eta'})}{\va'(\eta,\vvv;\eta')}\ud{\eta'}}&\leq& \abs{\int_0^{\eta}\frac{\nu\Big(\vvv'(\eta,\vvv;\eta')\Big)\exp(-H_{\eta,\eta'})}{\va'(\eta,\vvv;\eta')}\ud{\eta'}}\\
&\leq&\int_0^{\infty}\ue^{-y}\ud{y}=1.\no
\end{eqnarray}
\ \\
Case I - Type II: $\chi_2$: $\va\geq\sigma$, $m\leq\va'(\eta,\vvv;\eta')\leq M$ and $\vb'(\eta,\vvv;\eta')\leq M$.\\
Since along the characteristics, $\abs{\vvv}^2$ can be bounded by
$2M^2$ and the integral domain for $\vvv$ is finite, by Cauchy's inequality, we have
\begin{eqnarray}
I_2&\leq&C\ue^{4\varrho M^2}\int_{\r^2}\bigg(\int_0^{\eta}\frac{Q^2}{\nu}\Big(\eta',\vvv'(\eta,\vvv;\eta')\Big)\ud{\eta'}\bigg)
\bigg(\int_0^{\eta}\frac{\nu\Big(\vvv'(\eta,\vvv;\eta')\Big)\exp(-2H_{\eta,\eta'})}{\va'^2(\eta,\vvv;\eta')}\ud{\eta'}\bigg)\ud{\vvv}\\
&\leq&C\frac{\ue^{4\varrho M^2}}{m}\int_{\r^2}\bigg(\int_0^{\eta}\frac{Q^2}{\nu}\Big(\eta',\vvv'(\eta,\vvv;\eta')\Big)\ud{\eta'}\bigg)
\bigg(\int_0^{\eta}\frac{\nu\Big(\vvv'(\eta,\vvv;\eta')\Big)\exp(-2H_{\eta,\eta'})}{\va'(\eta,\vvv;\eta')}\ud{\eta'}\bigg)\ud{\vvv}\no\\
&\leq&C\frac{\ue^{4\varrho M^2}}{m}\bigg(\int_{\r^2}\int_0^{\eta}\frac{Q^2}{\nu}\Big(\eta',\vvv'(\eta,\vvv;\eta')\Big)\ud{\eta'}\ud{\vvv}\bigg)\no\\
&\leq&C\frac{M\ue^{4\varrho M^2}}{m\sigma}\bigg(\int_{\r^2}\int_0^{\eta}\frac{Q}{\nu^2}\Big(\eta',\vvv'\Big)\ud{\eta'}\ud{\vvv'}\bigg)\no\\
&\leq& C\frac{M\ue^{4\varrho M^2}}{m\sigma}\tnnm{\nu^{-\frac{1}{2}}Q}^2,\no
\end{eqnarray}
where for $y=H_{\eta,\eta'}$,
\begin{eqnarray}
\int_0^{\eta}\frac{\nu\Big(\vvv'(\eta,\vvv;\eta')\Big)}{\va'(\eta,\vvv;\eta')}\exp(-2H_{\eta,\eta'})\ud{\eta'}
\ud{\vvv}&\leq&\int_0^{\infty}\ue^{-2y}\ud{y}=\half,
\end{eqnarray}
and the Jacobian
\begin{eqnarray}
\abs{\dfrac{\ud{\vvv}}{\ud{\vvv'}}}=\abs{\dfrac{\rk-\e\eta}{\rk-\e\eta'}\dfrac{\va'}{\va}}\leq C\dfrac{\va'}{\va}\leq \dfrac{CM}{\sigma}.
\end{eqnarray}
\ \\
Case I - Type III: $\chi_3$: $\va\geq\sigma$, $0\leq\va'(\eta,\vvv;\eta')\leq m$ and $\vb'(\eta,\vvv;\eta')\leq M$.\\
We can directly verify the fact that
\begin{eqnarray}
0\leq\va\leq\va'(\eta,\vvv;\eta'),
\end{eqnarray}
for $\eta'\leq\eta$. Then we know the integral of $\va$ is always in a small domain.
We have for $y=H_{\eta,\eta'}$,
\begin{eqnarray}
I_3&\leq&C\ue^{4\varrho M^2}\lnnm{Q}{\vth,\varrho}^2\int_{\r^2}\frac{\chi_3}{\br{\vvv}^{\vth}}\bigg(\int_0^{\eta}\frac{\exp(-H_{\eta,\eta'})}{\va'(\eta,\vvv;\eta')}\ud{\eta'}
\bigg)^2\ud{\vvv}\\
&\leq&C\ue^{4\varrho M^2}\lnnm{Q}{\vth,\varrho}^2\int_{\r^2}\frac{\chi_3}{\br{\vvv}^{\vth}}\bigg(\int_0^{\infty}\ue^{-y}\ud{y}
\bigg)^2\ud{\vvv}\no\\
&\leq&C\ue^{4\varrho M^2}\lnnm{Q}{\vth,\varrho}^2\int_{\r^2}\frac{\chi_3}{\br{\vvv}^{\vth}}\ud{\vvv}\no\\
&\leq&C\ue^{4\varrho M^2}m\lnnm{Q}{\vth,\varrho}^2.\no
\end{eqnarray}
\ \\
Case I - Type IV: $\chi_4$: $\va\leq\sigma$, $\va'(\eta,\vvv;\eta')\leq M$ and $\vb'(\eta,\vvv;\eta')\leq M$.\\
Similar to Case I - Type III , we know the integral of $\va$ is always in a small domain.
We have for $y=H_{\eta,\eta'}$,
\begin{eqnarray}
I_3&\leq&C\ue^{4\varrho M^2}\lnnm{Q}{\vth,\varrho}^2\int_{\r^2}\frac{\chi_4}{\br{\vvv}^{\vth}}\bigg(\int_0^{\eta}\frac{\exp(-H_{\eta,\eta'})}{\va'(\eta,\vvv;\eta')}\ud{\eta'}
\bigg)^2\ud{\vvv}\\
&\leq&C\ue^{4\varrho M^2}\lnnm{Q}{\vth,\varrho}^2\int_{\r^2}\frac{\chi_4}{\br{\vvv}^{\vth}}\bigg(\int_0^{\infty}\ue^{-y}\ud{y}
\bigg)^2\ud{\vvv}\no\\
&\leq&C\ue^{4\varrho M^2}\lnnm{Q}{\vth,\varrho}^2\int_{\r^2}\frac{\chi_4}{\br{\vvv}^{\vth}}\ud{\vvv}\no\\
&\leq&C\ue^{4\varrho M^2}\sigma\lnnm{Q}{\vth,\varrho}^2.\no
\end{eqnarray}
\ \\
Collecting all three types, we have
\begin{eqnarray}
I\leq C\frac{M\ue^{4\varrho
M^2}}{m\sigma}\tnnm{\nu^{-\frac{1}{2}}Q}^2+C\bigg(\frac{1}{M^{\vth}}+\ue^{4\varrho M^2}(m+\sigma)\bigg)\lnnm{Q}{\vth,\varrho}^2.
\end{eqnarray}
Taking $M$ sufficiently large, $m<<\ue^{-4\varrho M^2}$ and
$\sigma<<\ue^{-4\varrho M^2}$ sufficiently small, we obtain
the desired result.\\
\ \\
Case II: \\
For $\va<0$ and $\va^2+\vb^2\geq \vb'^2(\eta,\vvv;L)$,
\begin{eqnarray}
\t\Big[Q(\eta,\vvv)\Big]&=&\bigg(\int_0^{L}\frac{Q\Big(\eta',\vvv(\eta,\vvv;\eta')\Big)}{\va'(\eta,\vvv;\eta')}
\exp(-H_{L,\eta'}-\rr[H_{L,\eta}])\ud{\eta'}\\
&&+\int_{\eta}^{L}\frac{Q\Big(\eta',\rr[\vvv(\eta,\vvv;\eta')]\Big)}{\va'(\eta,\vvv;\eta')}\exp(\rr[H_{\eta,\eta'}])\ud{\eta'}\bigg).\no
\end{eqnarray}
We first estimate
\begin{eqnarray}
\int_{\r^2}\ue^{2\varrho\abs{\vvv}^2}\bigg(\int_{\eta}^{L}\frac{Q\Big(\eta',\rr[\vvv(\eta,\vvv;\eta')]\Big)}{\va'(\eta,\vvv;\eta')}
\exp(\rr[H_{\eta,\eta'}])\ud{\eta'}
\bigg)^2\ud{\vvv}.
\end{eqnarray}
We can split the integral into four parts:
\begin{eqnarray}
II=II_1+II_2+II_3+II_4.
\end{eqnarray}
\ \\
Case II - Type I: $\chi_1$: $M\leq\va'(\eta,\vvv;\eta')$ or $M\leq\vb'(\eta,\vvv;\eta')$.\\
Similar to Case I - Type I, we have
\begin{eqnarray}
II_1&\leq&C\lnnm{Q}{\vth,\varrho}^2\int_{\r^2}\chi_1\bigg(\int_{\eta}^{L}\frac{1}{\br{\vvv'}^{\vth}}\frac{\exp(\rr[H_{\eta,\eta'}])}{\va'(\eta,\vvv;\eta')}\ud{\eta'}
\bigg)^2\ud{\vvv}\\
&\leq&\frac{C}{M^{\vth}}\lnnm{Q}{\vth,\varrho}^2\int_{\r^2}\frac{1}{\br{\vvv}^{\vth}}\bigg(\int_{\eta}^{L}\frac{\exp(\rr[H_{\eta,\eta'}])}{\va'(\eta,\vvv;\eta')}\ud{\eta'}
\bigg)^2\ud{\vvv}\no\\
&\leq&\frac{C}{M^{\vth}}\lnnm{Q}{\vth,\varrho}^2\int_{\r^2}\frac{1}{\br{\vvv}^{\vth}}\ud{\vvv}\no\\
&\leq&\frac{C}{M^{\vth}}\lnnm{Q}{\vth,\varrho}^2,\no
\end{eqnarray}
since for $y=H_{\eta,\eta'}$,
\begin{eqnarray}
\abs{\int_{\eta}^{L}\frac{\exp(\rr[H_{\eta,\eta'}])}{\va'(\eta,\vvv;\eta')}\ud{\eta'}}&\leq& \abs{\int_{\eta}^{L}\frac{\nu\Big(\vvv'(\eta,\vvv;\eta')\Big)\exp(\rr[H_{\eta,\eta'}])}{\va'(\eta,\vvv;\eta')}\ud{\eta'}}\\
&\leq&\int^0_{-\infty}\ue^{y}\ud{y}=1.\no
\end{eqnarray}
\ \\
Case II - Type II: $\chi_2$: $m\leq\va'(\eta,\vvv;\eta')\leq M$ and $\vb'(\eta,\vvv;\eta')\leq M$.\\
We can directly verify the fact that
\begin{eqnarray}
0\leq\va'(\eta,\vvv;\eta')\leq\abs{\va},
\end{eqnarray}
for $\eta'\geq\eta$. Similar to Case I - Type II, by Cauchy's inequality, we have
\begin{eqnarray}
II_2&\leq&C\ue^{4\varrho M^2}\int_{\r^2}\bigg(\int_0^{\eta}\frac{Q^2}{\nu}\Big(\eta',\vvv(\eta,\vvv;\eta')\Big)\ud{\eta'}\bigg)
\bigg(\int_{\eta}^{L}\frac{\nu\Big(\vvv'(\eta,\vvv;\eta')\Big)\exp(2\rr[H_{\eta,\eta'}])}{\va'^2(\eta,\vvv;\eta')}\ud{\eta'}\bigg)\ud{\vvv}\\
&\leq&C\frac{\ue^{4\varrho M^2}}{m}\int_{\r^2}\bigg(\int_{\eta}^{L}\frac{Q^2}{\nu}\Big(\eta',\vvv(\eta,\vvv;\eta')\Big)\ud{\eta'}\bigg)
\bigg(\int_{\eta}^{L}\frac{\nu\Big(\vvv'(\eta,\vvv;\eta')\Big)\exp(2\rr[H_{\eta,\eta'}])}{\va'(\eta,\vvv;\eta')}\ud{\eta'}\bigg)\ud{\vvv}\no\\
&\leq&C\frac{\ue^{4\varrho M^2}}{m}\bigg(\int_{\r^2}\int_{\eta}^{L}\frac{Q^2}{\nu}\Big(\eta',\vvv(\eta,\vvv;\eta')\Big)\ud{\eta'}\ud{\vvv}\bigg)\no\\
&\leq&C\frac{\ue^{4\varrho M^2}}{m}\bigg(\int_{\r^2}\int_{\eta}^{L}\frac{Q^2}{\nu}\Big(\eta',\vvv(\eta,\vvv;\eta')\Big)\ud{\eta'}\ud{\vvv'}\bigg)\no\\
&\leq& C\frac{\ue^{4\varrho M^2}}{m}\tnnm{\nu^{-\frac{1}{2}}Q}^2,\no
\end{eqnarray}
where for $y=H_{\eta,\eta'}$,
\begin{eqnarray}
\int_0^{\eta}\frac{\nu\Big(\vvv'(\eta,\vvv;\eta')\Big)}{\va'(\eta,\vvv;\eta')}\exp(-2H_{\eta,\eta'})\ud{\eta'}
\ud{\vvv}&\leq&\int_0^{\infty}\ue^{-2y}\ud{y}=\half,
\end{eqnarray}
and the Jacobian
\begin{eqnarray}
\abs{\dfrac{\ud{\vvv}}{\ud{\vvv'}}}=\abs{\dfrac{\rk-\e\eta}{\rk-\e\eta'}\dfrac{\va'}{\va}}\leq C\abs{\dfrac{\va'}{\va}}\leq C.
\end{eqnarray}
\ \\
Case II - Type III: $\chi_3$: $0\leq\va'(\eta,\vvv;\eta')\leq m$, $\vb'(\eta,\vvv;\eta')\leq M$ and $\eta'-\eta\geq\sigma$.\\
We know
\begin{eqnarray}
H_{\eta,\eta'}\leq-\frac{\sigma}{m}.
\end{eqnarray}
Then after substitution $y=H_{\eta,\eta'}$, the integral is not from zero, but from
$-\dfrac{\sigma}{m}$. In detail, we have
\begin{eqnarray}
I_3&\leq&C\ue^{4\varrho M^2}\lnnm{Q}{\vth,\varrho}^2\int_{\r^2}\frac{\chi_3}{\br{\vvv}^{\vth}}\bigg(\int_{\eta}^{L}\frac{\exp(\rr[H_{\eta,\eta'}])}{\va'(\eta,\vvv;\eta')}\ud{\eta'}
\bigg)^2\ud{\vvv}\\
&\leq&C\ue^{4\varrho M^2}\lnnm{Q}{\vth,\varrho}^2\int_{\r^2}\frac{\chi_3}{\br{\vvv}^{\vth}}\bigg(\int_{\eta}^{L}\frac{\nu\Big(\vvv'(\eta,\vvv;\eta')\Big)
\exp(\rr[H_{\eta,\eta'}])}{\va'(\eta,\vvv;\eta')}\ud{\eta'}
\bigg)^2\ud{\vvv}\no\\
&\leq&C\ue^{4\varrho M^2}\lnnm{Q}{\vth,\varrho}^2\int_{\r^2}\frac{\chi_3}{\br{\vvv}^{\vth}}\bigg(\int^{-\frac{\sigma}{m}}_{-\infty}\ue^{y}\ud{y}\bigg)^2\ud{\vvv}\no\\
&\leq&C\ue^{4\varrho M^2}\ue^{-\frac{\sigma}{m}}\lnnm{Q}{\vth,\varrho}^2\int_{\r^2}\frac{\chi_3}{\br{\vvv}^{\vth}}\ud{\vvv}\no\\
&\leq& C\ue^{4\varrho M^2}\ue^{-\frac{\sigma}{m}}\lnnm{Q}{\vth,\varrho}^2.\no
\end{eqnarray}
\ \\
Case II - Type IV: $\chi_4$: $0\leq\va'(\eta,\vvv;\eta')\leq m$, $\vb'(\eta,\vvv;\eta')\leq M$ and $\eta'-\eta\leq\sigma$.\\
For $\eta'\leq\eta$ and $\eta'-\eta\leq\sigma$, we have
\begin{eqnarray}
\va&=&\sqrt{\va'^2(\eta,\vvv;\eta')+\vb'^2(\eta,\vvv;\eta')-\vb^2}\\
&=&\sqrt{\va'^2(\eta,\vvv;\eta')+\vb'^2(\eta,\vvv;\eta')-\vb'^2\ue^{2W(\eta)-2W(\eta')}}\no\\
&=&\sqrt{\va'^2(\eta,\vvv;\eta')+\vb'^2(\eta,\vvv;\eta')-\vb'^2\left(\frac{\rk-\e\eta'}{\rk-\e\eta}\right)^2}\no\\
&\leq&\sqrt{\va'^2(\eta,\vvv;\eta')+2\rk M^2\e(\eta'-\eta)}\no\\
&\leq&
C\sqrt{m^2+\e M^2\sigma}\leq C(m+M\sqrt{\e\sigma}).\no
\end{eqnarray}
Therefore, the integral domain for $\va$ is very small. We have the estimate for $y=H_{\eta,\eta'}$
\begin{eqnarray}
I_4&\leq&C\ue^{4\varrho M^2}\lnnm{Q}{\vth,\varrho}^2\int_{\r^2}\frac{\chi_4}{\br{\vvv}^{\vth}}\bigg(\int_{\eta}^{L}\frac{\exp(\rr[H_{\eta,\eta'}])}{\va'(\eta,\vvv;\eta')}\ud{\eta'}
\bigg)^2\ud{\vvv}\\
&\leq&C\ue^{4\varrho M^2}\lnnm{Q}{\vth,\varrho}^2\int_{\r^2}\frac{\chi_4}{\br{\vvv}^{\vth}}\bigg(\int^0_{-\infty}\ue^{y}\ud{y}
\bigg)^2\ud{\vvv}\no\\
&\leq&C\ue^{4\varrho M^2}\lnnm{Q}{\vth,\varrho}^2\int_{\r^2}\frac{\chi_4}{\br{\vvv}^{\vth}}\ud{\vvv}\no\\
&\leq&C\ue^{4\varrho M^2}(m+\sqrt{\e\sigma})\lnnm{Q}{\vth,\varrho}^2.\no
\end{eqnarray}
\ \\
Collecting all four types, we have
\begin{eqnarray}
II\leq C\frac{\ue^{4\varrho
M^2}}{m}\tnnm{\nu^{-\frac{1}{2}}Q}^2+C\bigg(\frac{1}{M^{\vth}}+\ue^{4\varrho M^2}m\bigg)\lnnm{Q}{\vth,\varrho}^2.
\end{eqnarray}
Taking $M$ sufficiently large, $\sigma<<\ue^{-8\varrho M^2}$ sufficiently small and
$m<<\min\{\sigma,\ue^{-4\varrho M^2}\}$ sufficiently small, we obtain
the desired result.

Note that we have the decomposition
\begin{eqnarray}
&&\int_0^{L}\frac{Q\Big(\eta',\vvv(\eta,\vvv;\eta')\Big)}{\va'(\eta,\vvv;\eta')}
\exp(-H_{L,\eta'}-\rr[H_{L,\eta}])\ud{\eta'}\\
&=&\int_0^{\eta}\frac{Q\Big(\eta',\vvv(\eta,\vvv;\eta')\Big)}{\va'(\eta,\vvv;\eta')}
\exp(-H_{L,\eta'}-\rr[H_{L,\eta}])\ud{\eta'}+\int_{\eta}^{L}\frac{Q\Big(\eta',\vvv(\eta,\vvv;\eta')\Big)}{\va'(\eta,\vvv;\eta')}
\exp(-H_{L,\eta'}-\rr[H_{L,\eta}])\ud{\eta'}.\no
\end{eqnarray}
Then this term can actually be bounded using the techniques in Case I and Case II.\\
\ \\
Case III: \\
For $\va<0$ and $\va^2+\vb^2\leq \vb'^2(\eta,\vvv;L)$,
\begin{eqnarray}
\t\Big[Q(\eta,\vvv)\Big]&=&\bigg(\int_0^{\eta^+}\frac{Q\Big(\eta',\vvv(\eta,\vvv;\eta')\Big)}{\va'(\eta,\vvv;\eta')}
\exp(-H_{\eta^+,\eta'}-\rr[H_{\eta^+,\eta}])\ud{\eta'}\\
&&+\int_{\eta}^{\eta^+}\frac{Q\Big(\eta',\rr[\vvv(\eta,\vvv;\eta')]\Big)}{\va'(\eta,\vvv;\eta')}\exp(\rr[H_{\eta,\eta'}])\ud{\eta'}\bigg).\no
\end{eqnarray}
This is a combination of Case I and Case II, so it naturally holds.
\end{proof}

\begin{lemma}\label{Milne lemma 6}
The solution $\g(\eta,\vvv)$ to the $\e$-Milne problem (\ref{Milne})
satisfies for $\varrho\geq0$ and an integer $\vth\geq3$,
\begin{eqnarray}
\lnnm{\g-\g_{L}}{\vth,\varrho}\leq
C+C\tnnm{\g-\g_{L}}.
\end{eqnarray}
\end{lemma}
\begin{proof}
Define $u=\g-\g_{L}$. Then $u$ satisfies the equation
\begin{eqnarray}
\left\{
\begin{array}{l}\displaystyle
\va\frac{\p u}{\p\eta}+G(\eta)\bigg(\vb^2\dfrac{\p
u}{\p\va}-\va\vb\dfrac{\p u}{\p\vb}\bigg)+\nu u-K[u]
=S=\g_{2,L}G(\eta)\m^{\frac{1}{2}}\va\vb,\\\rule{0ex}{1.5em} u(0,\vvv)=p(\vvv)=h(\vvv)-\g_{L}(\vvv)\ \
\text{for}\ \ \va>0,\\\rule{0ex}{1.5em}
u(L,\vvv)=u(L,\rr[\vvv]).
\end{array}
\right.
\end{eqnarray}
Since $u=\k[p]+\t\Big[K[u]\Big]+\t[S]$,
based on Lemma \ref{Milne lemma 5}, we have
\begin{eqnarray}
\ltnm{u}{\varrho}
&\leq&\lnmv{\k[p]}+\ltnm{\t\Big[K[u]\Big]}{\varrho}+\lnnmv{\t[S]}\\
&\leq&\lnmv{p}+
C(\d)\tnnm{\nu^{-\frac{1}{2}}K[u]}+\d\lnnm{K[u]}{\vth,\varrho}+\lnnm{\nu^{-1}S}{\vth,\varrho}\no\\
&\leq&
\lnmv{p}+
C(\d)\tnnm{u}+\d\lnnm{K[u]}{\vth,\varrho}+\lnnm{\nu^{-1}S}{\vth,\varrho},\no
\end{eqnarray}
where we can directly verify
\begin{eqnarray}
\tnnm{\nu^{-\frac{1}{2}}K[u]}&\leq&\tnnm{u}.
\end{eqnarray}
In \cite[Lemma 3.3.1]{Glassey1996}, it is shown that
\begin{eqnarray}
\lnnm{K[u]}{\vth,\varrho}&\leq&\lnnm{u}{\vth-1,\varrho},\\
\lnnm{K[u]}{0,\varrho}&\leq&\ltnm{u}{\varrho}.
\end{eqnarray}
Since $u=\k[p]+\t\Big[K[u]+ S\Big]$, by Lemma \ref{Milne lemma 3} and Lemma \ref{Milne lemma 4}, for $\e$ and $\d$ sufficiently
small, we can estimate
\begin{eqnarray}
\lnnm{u}{\vth,\varrho}&\leq&C\bigg(\lnnm{\t\Big[K[u]\Big]}{\vth,\varrho}+\lnnm{\t[S]}{\vth,\varrho}+\lnm{\k[p]}{\vth,\varrho}\bigg)\\
&\leq&C\bigg(\lnnm{K[u]}{\vth,\varrho}+\lnnm{\nu^{-1}S}{\vth,\varrho}+\lnm{p}{\vth,\varrho}\bigg)\no\\
&\leq&C\bigg(\lnnm{u}{\vth-1,\varrho}+\lnnm{\nu^{-1}S}{\vth,\varrho}+\lnm{p}{\vth,\varrho}\bigg)\no\\
&\leq&\ldots\no\\
&\leq&C\bigg(\lnnm{K[u]}{0,\varrho}+\lnnm{\nu^{-1}S}{\vth,\varrho}+\lnm{p}{\vth,\varrho}\bigg)\no\\
&\leq&C\bigg(\ltnm{u}{\varrho}+\lnnm{\nu^{-1}S}{\vth,\varrho}+\lnm{p}{\vth,\varrho}\bigg)\no\\
&\leq&C(\d)\tnnm{u}+\d\lnnm{K[u]}{\vth,\varrho}+C\bigg(\lnnm{\nu^{-1}S}{\vth,\varrho}+\lnm{p}{\vth,\varrho}\bigg).\no
\end{eqnarray}
Therefore, absorbing $\d\lnnm{K[u]}{\vth,\varrho}$ into the right-hand side of the second inequality implies
\begin{eqnarray}
\lnnm{K[u]}{\vth,\varrho}&\leq& C\bigg(\tnnm{u}+\lnnm{\nu^{-1}S}{\vth,\varrho}+\lnm{p}{\vth,\varrho}\bigg).
\end{eqnarray}
Therefore, we have
\begin{eqnarray}
\lnnm{u}{\vth,\varrho}&\leq&\lnmv{\k[p]}+\lnnmv{\t\Big[K[u]\Big]}+\lnnmv{\t[S]}\\
&\leq&C\bigg(\lnmv{p}+\lnnmv{\nu^{-1}K[u]}+\lnnmv{\nu^{-1}S}\bigg)\no\\
&\leq&C\bigg(\lnmv{p}+\lnnmv{K[u]}+\lnnmv{\nu^{-1}S}\bigg)\no\\
&\leq&C\bigg(\tnnm{u}+\lnnm{\nu^{-1}S}{\vth,\varrho}+\lnm{p}{\vth,\varrho}\bigg).\no
\end{eqnarray}
Then our result naturally follows.
\end{proof}
\begin{lemma}\label{Milne lemma 7}
There
exists a unique solution $\g(\eta,\vvv)$ to the $\e$-Milne problem
(\ref{Milne}) satisfying for $\varrho\geq0$ and an integer $\vth\geq3$,
\begin{eqnarray}
\lnnm{\g-\g_{L}}{\vth,\varrho}\leq C.
\end{eqnarray}
\end{lemma}
\begin{proof}
Based on Lemma \ref{Milne lemma 2} and Lemma \ref{Milne lemma 6}, this is obvious.
\end{proof}
\begin{theorem}\label{Milne theorem 2}
There exists a unique solution $\gg(\eta,\vvv)$ to the $\e$-Milne
problem (\ref{Milne transform}) satisfying for
$\varrho\geq0$ and an integer $\vth\geq3$,
\begin{eqnarray}
\lnnm{\gg}{\vth,\varrho}\leq C.
\end{eqnarray}
\end{theorem}
\begin{proof}
Based on Theorem \ref{Milne theorem 1} and Lemma \ref{Milne lemma
7}, this is obvious.
\end{proof}

\subsection{Exponential Decay}

\begin{theorem}\label{Milne theorem 3}
For sufficiently small $K_0$, there
exists a unique solution $\gg(\eta,\vvv)$ to the $\e$-Milne
problem (\ref{Milne transform}) satisfying for
$\varrho\geq0$ and an integer $\vth\geq3$,
\begin{eqnarray}
\lnnm{\ue^{K_0\eta}\gg}{\vth,\varrho}\leq C.
\end{eqnarray}
\end{theorem}
\begin{proof}
Define $U=\ue^{K_0\eta}\gg$. Then $U$ satisfies the equation
\begin{eqnarray}\label{exponential}
\left\{
\begin{array}{l}\displaystyle
\va\frac{\p U}{\p\eta}+G(\eta)\bigg(\vb^2\dfrac{\p
U}{\p\va}-\va\vb\dfrac{\p U}{\p\vb}\bigg)+\ll[U]
=K_0\va U,\\\rule{0ex}{1.5em}
U(0,\vvv)=h(\vvv)-\tilde h(\vvv)\ \ \text{for}\ \
\va>0,\\\rule{0ex}{1.5em}
U(L,\vvv)=U(L,\rr[\vvv])
\end{array}
\right.
\end{eqnarray}
We divide the proof into several steps:\\
\ \\
Step 1: $L^2$ Estimates.\\
In the proof of Lemma \ref{Milne lemma 1},
we already show
\begin{eqnarray}
\int_0^{L}\ue^{2K_0\eta}\br{{w}_{\gg},{w}_{\gg}}(\eta)\ud{\eta}\leq C,
\end{eqnarray}
where we decompose $\gg={w}_{\gg}+{q}_{\gg}$ with $q_{\gg}\in\nk$ and $w_{\gg}\in\nk^{\perp}$. Based on the construction of $\gg$, we naturally have ${q}_{\gg,L}=0$.
Then using the orthogonal relation, we have
\begin{eqnarray}
&&\int_0^{L}\ue^{2K_0\eta}\int_{\r^2}\gg^2(\eta,\vvv)\ud{\vvv}\ud{\eta}=
\int_0^{L}\ue^{2K_0\eta}\br{{q}_{\gg},{q}_{\gg}}(\eta)\ud{\eta}+\int_0^{L}\ue^{2K_0\eta}\br{{w}_{\gg},{w}_{\gg}}(\eta)\ud{\eta}.
\end{eqnarray}
Similar to Step 6 in the proof of Lemma \ref{Milne lemma 1}, using the exponential decay of ${w}_{\gg}$, we have
\begin{eqnarray}
\int_0^{L}\ue^{2K_0\eta}\br{{q}_{\gg},{q}_{\gg}}(\eta)\ud{\eta}\leq C+C\int_0^{L}\ue^{2K_0\eta}\br{{w}_{\gg},{w}_{\gg}}(\eta)\ud{\eta}.
\end{eqnarray}
This shows that
\begin{eqnarray}
\tnnm{U}<C.
\end{eqnarray}
\ \\
Step 2: $L^{\infty}$ Estimates.\\
Since $U=\k[p]+\t\Big[K[U]\Big]+\t[K_0\va U]$, similar to the proof of Lemma \ref{Milne lemma 6}, we have
\begin{eqnarray}
\lnnm{U}{\vth,\varrho}&\leq& C\bigg(\tnnm{U}+\lnnm{\nu^{-1}K_0\va U}{\vth,\varrho}+\lnm{p}{\vth,\varrho}\bigg)\\
&\leq&C\bigg(\tnnm{U}+K_0\lnnm{U}{\vth,\varrho}+\lnm{p}{\vth,\varrho}\bigg).\no
\end{eqnarray}
When $K_0>0$ is sufficiently small, we may absorb $K_0\lnnm{U}{\vth,\varrho}$ into the left-hand side to obtain
\begin{eqnarray}
\lnnm{U}{\vth,\varrho}&\leq&C\bigg(\tnnm{U}+\lnm{p}{\vth,\varrho}\bigg).\no
\end{eqnarray}
Then we naturally obtain the result.
\end{proof}

\newpage

\section{Regularity of $\e$-Milne Problem with Geometric Correction}

In this section, we consider the regularity of the $\e$-Milne problem with geometric correction for $\gg(\eta,\theta,\vww)$ in
the domain
$(\eta,\theta,\vww)\in[0,L]\times[-\pi,\pi)\times\r^2$ as
\begin{eqnarray}\label{Milne difference problem.}
\left\{
\begin{array}{l}\displaystyle
\va\dfrac{\p\gg }{\p\eta}-\dfrac{\e}{\rk-\e\eta}\bigg(\vb^2\dfrac{\p\gg }{\p\va}-\va\vb\dfrac{\p\gg }{\p\vb}\bigg)+\ll[\gg ]
=0,\\\rule{0ex}{1.5em}
\gg (0,\theta,\vww)=p (\theta,\vww)\ \ \text{for}\ \
\va>0,\\\rule{0ex}{1.5em}
\displaystyle\gg (L,\theta,\vww)=\gg (L,\theta,\rr[\vww]),
\end{array}
\right.
\end{eqnarray}
where $\rr[\vvv]=(-\va,\vb)$ and $L=\e^{-\frac{1}{2}}$. For simplicity, we temporarily ignore the dependence of $\theta$, i.e. consider
the $\e$-Milne problem with geometric correction for $\gg(\eta,\vww)$ in
the domain
$(\eta,\vww)\in[0,L]\times\r^2$ as
\begin{eqnarray}\label{Milne difference problem}
\left\{
\begin{array}{l}\displaystyle
\va\dfrac{\p\gg }{\p\eta}-\dfrac{\e}{\rk-\e\eta}\bigg(\vb^2\dfrac{\p\gg }{\p\va}-\va\vb\dfrac{\p\gg }{\p\vb}\bigg)+\ll[\gg ]
=0,\\\rule{0ex}{1.5em}
\gg (0,\vww)=p (\vww)\ \ \text{for}\ \
\va>0,\\\rule{0ex}{1.5em}
\displaystyle\gg (L,\vww)=\gg (L,\rr[\vww]).
\end{array}
\right.
\end{eqnarray}
Define a weight function
\begin{eqnarray}
\zeta(\eta,\va,\vb)=\left(\left(\va^2+\vb^2\right)-\left(\frac{\rk-\e\eta}{\rk}\right)^2\vb^2\right)^{\frac{1}{2}}.
\end{eqnarray}
It is easy to see that the closer a point $(\eta,\va,\vb)$ is to the grazing set $(\eta,\va,\vb)=(0,0,\vb)$, the smaller $\zeta$ is. In particular, at the grazing set, $\zeta(0,0,\vb)=0$.
\begin{lemma}\label{weight lemma}
We have
\begin{eqnarray}
\va\dfrac{\p\zeta}{\p\eta}-\dfrac{\e}{\rk-\e\eta}\bigg(\vb^2\dfrac{\p\zeta}{\p\va}-\va\vb\dfrac{\p\zeta}{\p\vb}\bigg)=0.
\end{eqnarray}
\end{lemma}
\begin{proof}
We may directly compute
\begin{eqnarray}
\frac{\p\zeta}{\p\eta}=\frac{1}{\zeta}\frac{\rk-\e\eta}{\rk}\e\vb^2,\ \
\frac{\p\zeta}{\p\va}=\frac{1}{\zeta}\va,\ \
\frac{\p\zeta}{\p\vb}=\frac{1}{\zeta}\left(\vb-\left(\frac{\rk-\e\eta}{\rk}\right)^2\vb\right)
\end{eqnarray}
Then we know
\begin{eqnarray}
&&\va\dfrac{\p\zeta}{\p\eta}-\dfrac{\e}{\rk-\e\eta}\bigg(\vb^2\dfrac{\p\zeta}{\p\va}-\va\vb\dfrac{\p\zeta}{\p\vb}\bigg)\\
&=&\frac{1}{\zeta}\left(\frac{\rk-\e\eta}{\rk}\e\va\vb^2
-\frac{\e}{\rk-\e\eta}\left(\va\vb^2-\va\vb^2+\va\vb^2\left(\frac{\rk-\e\eta}{\rk}\right)^2\right)\right)=0.\no
\end{eqnarray}
\end{proof}

\subsection{Mild Formulation}

Consider the $\e$-transport problem for $\a=\zeta\dfrac{\p\v}{\p\eta}$ as
\begin{eqnarray}\label{normal derivative equation}
\left\{
\begin{array}{l}\displaystyle
\va\frac{\p\a}{\p\eta}+G(\eta)\bigg(\vb^2\dfrac{\p\a }{\p\va}-\va\vb\dfrac{\p\a }{\p\vb}\bigg)+\nu\a=\tilde\a+S_{\a},\\\rule{0ex}{1.5em}
\a(0,\vvv)=p_{\a}(\vvv)\ \ \text{for}\ \ \va>0,\\\rule{0ex}{1.5em}
\a(L,\vvv)=\a(L,\rr[\vvv]),
\end{array}
\right.
\end{eqnarray}
where $p_{\a}$ and $S_{\a}$ will be specified later with
\begin{eqnarray}
\tilde\a(\eta,\vvv)=\int_{\r^2}\frac{\zeta(\eta,\vvv)}{\zeta(\eta,\vuu)}k(\vuu,\vvv)\a(\eta,\vuu)\ud{\vuu}.
\end{eqnarray}
Here we utilize Lemma \ref{weight lemma}. We need to derive the a priori estimate of $\a$.
Define the energy as before
\begin{eqnarray}
E_1&=&\va^2+\vb^2,\\
E_2&=&\vb\ue^{-W(\eta)}.
\end{eqnarray}
We can easily check that the weight function $\zeta=\sqrt{E_1-E_2^2}$. Along the characteristics, where $E_1$, $E_2$ and $\zeta$ are constants, the equation (\ref{normal derivative equation}) can be simplified as follows:
\begin{eqnarray}
\va\frac{\ud{\a}}{\ud{\eta}}+\a=\tilde\a+S_{\a}.
\end{eqnarray}
Similar to the $\e$-Milne problem with geometric correction, we can define the solution along the characteristics as follows:
\begin{eqnarray}
\a(\eta,\vvv)=\k[p_{\a}]+\t[\tilde\a+S_{\a}],
\end{eqnarray}
where\\
\ \\
Region I:\\
For $\va>0$,
\begin{eqnarray}\label{mt 16}
\k[p_{\a}]&=&p_{\a}\Big(\vvv'(\eta,\vvv; 0)\Big)\exp(-H_{\eta,0}),\\
\t[\tilde\a+S_{\a}]&=&\int_0^{\eta}\frac{(\tilde\a+S_{\a})\Big(\eta',\vvv'(\eta,\vvv;\eta')\Big)}{\va'(\eta,\vvv;\eta')}\exp(-H_{\eta,\eta'})\ud{\eta'}.
\end{eqnarray}
\ \\
Region II:\\
For $\va<0$ and $\va^2+\vb^2\geq \vb'^2(\eta,\vvv;L)$,
\begin{eqnarray}\label{mt 17}
\k[p_{\a}]&=&p_{\a}\Big(\vvv'(\eta,\vvv; 0)\Big)\exp(-H_{L,0}-\rr[H_{L,\eta}]),\\
\t[\tilde\a+S_{\a}]&=&\bigg(\int_0^{L}\frac{(\tilde\a+S_{\a})\Big(\eta',\vvv'(\eta,\vvv;\eta')\Big)}{\va'(\eta,\vvv;\eta')}
\exp(-H_{L,\eta'}-\rr[H_{L,\eta}])\ud{\eta'}\\
&&+\int_{\eta}^{L}\frac{(\tilde\a+S_{\a})\Big(\eta',\rr[\vvv'(\eta,\vvv;\eta')]\Big)}{\va'(\eta,\vvv;\eta')}\exp(\rr[H_{\eta,\eta'}])\ud{\eta'}\bigg).\no
\end{eqnarray}
\ \\
Region III:\\
For $\va<0$ and $\va^2+\vb^2\leq \vb'^2(\eta,\vvv;L)$,
\begin{eqnarray}\label{mt 18}
\k[p_{\a}]&=&p_{\a}\Big(\vvv'(\eta,\vvv; 0)\Big)\exp(-H_{\eta^+,0}-\rr[H_{\eta^+,\eta}]),\\
\t[\tilde\a+S_{\a}]&=&\bigg(\int_0^{\eta^+}\frac{(\tilde\a+S_{\a})\Big(\eta',\vvv'(\eta,\vvv;\eta')\Big)}{\va'(\eta,\vvv;\eta')}
\exp(-H_{\eta^+,\eta'}-\rr[H_{\eta^+,\eta}])\ud{\eta'}\\
&&+\int_{\eta}^{\eta^+}\frac{(\tilde\a+S_{\a})\Big(\eta',\rr[\vvv'(\eta,\vvv;\eta')]\Big)}{\va'(\eta,\vvv;\eta')}\exp(\rr[H_{\eta,\eta'}])\ud{\eta'}\bigg).\no
\end{eqnarray}
Then we need to estimate $\k[p_{\a}]$ and $\t[\tilde\a+S_{\a}]$ in each region. We assume $0<\d<<1$ and $0<\d_0<<1$ are small quantities which will be determined later.

\subsection{Region I: $\va>0$}

Based on Lemma \ref{Milne lemma 3} and Lemma \ref{Milne lemma 4},
we can directly obtain
\begin{eqnarray}
\lnmv{\k[p_{\a}]}&\leq&\lnmv{p_{\a}},\\
\lnnmv{\t[S_{\a}]}&\leq&\lnnmv{\frac{S_{\a}}{\nu}}.
\end{eqnarray}
Hence, we only need to estimate
\begin{eqnarray}
I=\t[\tilde\a]=\int_0^{\eta}\frac{\tilde\a\Big(\eta',\vvv'(\eta,\vvv;\eta')\Big)}{\va'(\eta,\vvv;\eta')}\exp(-H_{\eta,\eta'})\ud{\eta'}.
\end{eqnarray}
Since we always assume that $(\eta,\vvv)$ and $(\eta',\vvv')$ are on the same characteristics, in the following, we will simply write $\vvv'(\eta')$ or even $\vvv'$ instead of $\vvv'(\eta,\vvv;\eta')$ when there is no confusion. We can use this notation interchangeably when necessary.\\
\ \\
We divide it into several steps:\\
\ \\
Step 0: Preliminaries.\\
We have
\begin{eqnarray}
E_2(\eta',\vb')=\frac{\rk-\e\eta'}{\rk}\vb'.
\end{eqnarray}
Then we can directly obtain
\begin{eqnarray}\label{pt 01}
\zeta(\eta',\vvv')&=&\frac{1}{\rk}\sqrt{\rk^2(\va'^2+\vb'^2)-\bigg((\rk-\e\eta')\vb'\bigg)^2}
=\frac{1}{\rk}\sqrt{\rk^2\va'^2+\Big(\rk^2-(\rk-\e\eta')^2\Big)\vb'^2},\\
&\leq& \frac{1}{\rk}\sqrt{\rk^2\va'^2}+\frac{1}{\rk}\sqrt{\Big(\rk^2-(\rk-\e\eta')^2\Big)\vb'^2}\leq C\bigg(\va'+\sqrt{\e\eta'}\vb'\bigg)\leq C\nu(\vvv'),\no
\end{eqnarray}
and
\begin{eqnarray}\label{pt 02}
\zeta(\eta',\vvv')\geq\frac{1}{2}\left(\frac{1}{\rk}\sqrt{\rk^2\va'^2}+\frac{1}{\rk}\sqrt{\Big(\rk^2-(\rk-\e\eta')^2\Big)\vb'^2}\right)\geq C\bigg(\va'+\sqrt{\e\eta'}\vb'\bigg)\geq C\sqrt{\e\eta'}\abs{\vvv}.
\end{eqnarray}
Also, we know for $0\leq\eta'\leq\eta$,
\begin{eqnarray}
\va'&=&\sqrt{\va^2+\vb^2-\vb'^2}=\sqrt{\va^2+\vb^2-\vb^2\bigg(\frac{\rk-\e\eta}{\rk-\e\eta'}\bigg)^2}\\
&=&\frac{\sqrt{(\rk-\e\eta')^2\va^2+(2\rk-\e\eta-\e\eta')(\e\eta-\e\eta')\vb^2}}{\rk-\e\eta'}.
\end{eqnarray}
Since
\begin{eqnarray}
0\leq(2\rk-\e\eta-\e\eta')(\e\eta-\e\eta')\vb^2\leq 2\rk\e(\eta-\eta')\vb^2,
\end{eqnarray}
we have
\begin{eqnarray}
\va\leq\va'
\leq2\sqrt{\va^2+\e(\eta-\eta')\vb^2},
\end{eqnarray}
which means
\begin{eqnarray}
\frac{1}{2\sqrt{\va^2+\e(\eta-\eta')\vb^2}}\leq\frac{1}{\va'}
\leq\frac{1}{\va}.
\end{eqnarray}
Therefore,
\begin{eqnarray}\label{pt 03}
-\int_{\eta'}^{\eta}\frac{1}{\va'(\eta,\vvv;y)}\ud{y}&\leq& -\int_{\eta'}^{\eta}\frac{1}{2\sqrt{\va^2+\e(\eta-y)\vb^2}}\ud{y}=\frac{1}{\e\vb^2}\bigg(\va-\sqrt{\va^2+\e(\eta-\eta')\vb^2}\bigg)\\
&=&-\frac{\eta-\eta'}{\va+\sqrt{\va^2+\e(\eta-\eta')\vb^2}}\leq-\frac{\eta-\eta'}{2\sqrt{\va^2+\e(\eta-\eta')\vb^2}}.\no
\end{eqnarray}
Define a cut-off function $\chi\in C^{\infty}[0,\infty)$ satisfying
\begin{eqnarray}
\chi(\va)=\left\{
\begin{array}{ll}
1&\text{for}\ \ \abs{\va}\leq\d,\\
0&\text{for}\ \ \abs{\va}\geq2\d,
\end{array}
\right.
\end{eqnarray}
In the following, we will divide the estimate of $I$ into several cases based on the value of $\va$, $\va'$, $\e\eta'$ and $\e(\eta-\eta')$. Let $\id$ denote the indicator function. Assume the dummy variable $\vuu=(\ua,\ub)$. We write
\begin{eqnarray}
I&=&\int_0^{\eta}\id_{\{\va\geq\d_0\}}+\int_0^{\eta}\id_{\{0\leq\va\leq\d_0\}}\id_{\{\chi(\ua)<1\}}
+\int_0^{\eta}\id_{\{0\leq\va\leq\d_0\}}\id_{\{\chi(\ua)=1\}}\id_{\{ \sqrt{\e\eta'}\abs{\vb'}\geq\va'\}}\\
&&+\int_0^{\eta}\id_{\{0\leq\va\leq\d_0\}}\id_{\{\chi(\ua)=1\}}\id_{\{\sqrt{\e\eta'}\abs{\vb'}\leq\va'\}}\id_{\{\va^2\leq\e(\eta-\eta')\vb^2\}}\no\\
&&+\int_0^{\eta}\id_{\{0\leq\va\leq\d_0\}}\id_{\{\chi(\ua)=1\}}\id_{\{\sqrt{\e\eta'}\abs{\vb'}\leq\va'\}}\id_{\{\va^2\geq\e(\eta-\eta')\vb^2\}}\no\\
&=&I_1+I_2+I_3+I_4+I_5.\no
\end{eqnarray}
\ \\
Step 1: Estimate of $I_1$ for $\va\geq\d_0$.\\
In this step, we will prove estimates based on the characteristics of $\v$ itself instead of $\a$.
Here, we rewrite the equation (\ref{Milne difference problem}) along the characteristics as
\begin{eqnarray}
\va\frac{\ud{\gg}}{\ud{\eta}}+\nu\gg=K[\v].
\end{eqnarray}
In the following, we will repeatedly use simple facts (SF):
\begin{itemize}
\item
Based on the well-posedness and decay theorem for $\gg$, we know $\lnmv{\gg}\leq C$.
\item
Based on Lemma \ref{prelim 1}, we get $\lnmv{K[\gg]}\leq\lnmv{\gg}\leq C$ and $\lnmv{\nabla_vK[\gg]}\leq\lnmv{\gg}\leq C$.
\item
Since $E_1$ is conserved along the characteristics, we must have $\abs{\vvv}=\abs{\vvv'}$.
\item
For $\eta'\leq\eta$, we must have $\va'\geq\va\geq\d_0$.
\item
Using substitution $y=H_{\eta,\eta'}$, we know
\begin{eqnarray}\label{inte 1}
\abs{\int_0^{\eta}\frac{\nu\Big(\vvv'(\eta,\vvv;\eta')\Big)}{\va'(\eta,\vvv;\eta')}\exp(-H_{\eta,\eta'})\ud{\eta'}}\leq\abs{\int_0^{\infty}\ue^{-y}\ud{y}}=1.
\end{eqnarray}
\end{itemize}
For $\va\geq\d_0$, we do not need the mild formulation for $\a$. Instead, we directly estimate
\begin{eqnarray}
\abs{\bvv I_1} &\leq&\abs{\bvv\frac{\p\gg}{\p\eta}}.
\end{eqnarray}
We rewrite the equation (\ref{Milne difference problem}) along the characteristics as
\begin{eqnarray}\label{mt 01}
\gg(\eta,\vvv)&=&\exp\left(-H_{\eta,0}\right)\Bigg(p\Big(\vvv'(\eta,\vvv;0)\Big)
+\int_0^{\eta}\frac{K[\gg]\Big(\eta',\vvv'(\eta,\vvv;\eta')\Big)}{\va'(\eta,\vvv;\eta')}
\exp\left(H_{\eta',0}\right)\ud{\eta'}\Bigg).
\end{eqnarray}
where $(\eta',\vvv')$ and $(\eta,\vvv)$ are on the same characteristic with $\va'\geq0$, and
\begin{eqnarray}
H_{t,s}&=&\int_{s}^{t}\frac{\nu(\vvv'(\eta,\vvv;y))}{\va'(\eta,\vvv;y)}\ud{y}.
\end{eqnarray}
for any $s,t\geq0$.

Taking $\eta$ derivative on both sides of (\ref{mt 01}), we have
\begin{eqnarray}\label{mt 04}
\frac{\p\gg}{\p\eta}&=&X=X_1+X_2+X_3+X_4+X_5+X_6,
\end{eqnarray}
where
\begin{eqnarray}
X_1&=&-\exp\left(-H_{\eta,0}\right)\frac{\p{H_{\eta,0}}}{\p{\eta}}\Bigg(p\Big(\vvv'(\eta,\vvv;0)\Big)
+\int_0^{\eta}\frac{K[\gg]\Big(\eta',\vvv'(\eta,\vvv;\eta')\Big)}{\va'(\eta')}
\exp\left(H_{\eta',0}\right)\ud{\eta'}\Bigg),\\
X_2&=&\exp\left(-H_{\eta,0}\right)\frac{\p p\Big(\vvv'(\eta,\vvv;0)\Big)}{\p\eta},\\
X_3&=&\frac{K[\gg](\eta,\vvv)}{\va},\\
X_4&=&-\exp\left(-H_{\eta,0}\right)\int_0^{\eta}\Bigg(K[\gg]\Big(\eta',\vvv'(\eta,\vvv;\eta')\Big)
\exp\left(H_{\eta',0}\right)\frac{1}{\va'^2(\eta,\vvv;\eta')}\frac{\p\va'(\eta,\vvv;\eta')}{\p\eta}\ud{\eta'}\bigg)\\
X_5&=&\exp\left(-H_{\eta,0}\right)\int_0^{\eta}\frac{K[\gg]\Big(\eta',\vvv'(\eta,\vvv;\eta')\Big)}{\va'(\eta,\vvv;\eta')}
\exp\left(H_{\eta',0}\right)\frac{\p{H_{\eta',0}}}{\p{\eta}}\ud{\eta'},\no\\
X_6&=&\exp\left(-H_{\eta,0}\right)
\int_0^{\eta}\frac{1}{\va'(\eta,\vvv;\eta')}\bigg(\nabla_{\mathfrak{v}'} K[\gg]\Big(\eta',\vvv'(\eta,\vvv;\eta')\Big)\frac{\p\vvv'(\eta,\vvv;\eta')}{\p\eta}\bigg)
\exp\left(H_{\eta',0}\right)\ud{\eta'}.
\end{eqnarray}
We need to estimate each term.
Note that
\begin{eqnarray}\label{rt 03}
\frac{\p{H_{t,s}}}{\p{\eta}}&=&\int_{s}^{t}\frac{\p{}}{\p{\eta}}\bigg(\frac{\nu(\vvv'(\eta,\vvv;y))}{\va'(\eta,\vvv;y)}\bigg)\ud{y}\\
&=&\int_{s}^{t}\frac{1}{\va'(\eta,\vvv;y)}\frac{\p\nu(\abs{\vvv'})}{\p\abs{\vvv'}}\frac{1}{\abs{\vvv'}}
\bigg(\va'(\eta,\vvv;y)\frac{\p\va'(\eta,\vvv;y)}{\p\eta}+\vb'(\eta,\vvv;y)\frac{\p\vb'(\eta,\vvv;y)}{\p\eta}\bigg)\ud{y}\no\\
&&-\int_{s}^{t}\frac{\nu(\vvv'(\eta,\vvv;y))}{\va'^2(\eta,\vvv;y)}\frac{\p\va'(\eta,\vvv;y)}{\p\eta}\ud{y}.\no
\end{eqnarray}
Considering
\begin{eqnarray}
\vb'(\eta,\vvv;y)&=&\vb\ue^{W(\eta')-W(\eta)}=\vb\frac{\rk-\e\eta}{\rk-\e\eta'},\\
\va'(\eta,\vvv;y)&=&\sqrt{\va^2+\vb^2-\vb'^2}=\sqrt{\va^2+\vb^2-\vb^2\bigg(\frac{\rk-\e\eta}{\rk-\e\eta'}\bigg)^2},
\end{eqnarray}
we know
\begin{eqnarray}
\frac{\p\vb'(\eta,\vvv;y)}{\p\eta}=\frac{\e\vb}{\rk-\e\eta'},\quad
\frac{\p\va'(\eta,\vvv;y)}{\p\eta}=\frac{2\e\vb^2}{\va'(\eta,\vvv;y)}\frac{\rk-\e\eta}{\rk-\e\eta'}.
\end{eqnarray}
This implies
\begin{eqnarray}
\abs{\frac{\p\va'(\eta,\vvv;y)}{\p\eta}}\leq \frac{C\e\abs{\vvv}}{\va'(\eta,\vvv;y)}\leq \frac{C\e\abs{\vvv}}{\d_0},\ \ \abs{\frac{\p\vb'(\eta,\vvv;y)}{\p\eta}}\leq C\e\abs{\vvv}.
\end{eqnarray}
The method to estimate $X_i$ is standard and we simply use the facts (SF) and direct computation, so we omit the details and only list the result
\begin{eqnarray}
\abs{\bv X_1}&\leq&\frac{C}{\d_0}\lnnmv{\gg},\\
\abs{\bv X_2}&\leq&\frac{C}{\d_0}\bigg(\lnmv{\frac{\p p}{\p\va}}+\lnmv{\frac{\p p}{\p\vb}}\bigg),\\
\abs{\bv X_3}&\leq&\frac{C}{\d_0}\lnnmv{\gg},\\
\abs{\bv X_4}&\leq&\frac{C}{\d_0}\lnnmv{\gg},\\
\abs{\bv X_5}&\leq&\frac{C}{\d_0}\lnnmv{\gg},\\
\abs{\bv X_6}&\leq&\frac{C}{\d_0}\lnnmv{\gg}.
\end{eqnarray}
In summary, we have
\begin{eqnarray}
\abs{\bvv I_1} &\leq& \frac{C}{\d_0}\bigg(\lnmv{\frac{\p p}{\p\va}}+\lnmv{\frac{\p p}{\p\vb}}+\lnnmv{\gg}\bigg).
\end{eqnarray}
\ \\
Step 2: Estimate of $I_2$ for $0\leq\va\leq\d_0$ and $\chi(\ua)<1$.\\
We have
\begin{eqnarray}
I_2&=&\int_0^{\eta}\bigg(\int_{\r^2}\frac{\zeta(\eta',\vvv')}{\zeta(\eta',\vuu)}\Big(1-\chi(\ua)\Big)
k(\vuu,\vvv')\a(\eta',\vuu)\ud{\vuu}\bigg)
\frac{1}{\va'}\exp(-H_{\eta,\eta'})\ud{\eta'}\\
&=&\int_0^{\eta}\bigg(\int_{\r^2}\Big(1-\chi(\ua)\Big)
k(\vuu,\vvv')\frac{\gg(\eta',\vuu)}{\p\eta'}\ud{\vuu}\bigg)
\frac{\zeta(\eta',\vvv')}{\va'}\exp(-H_{\eta,\eta'})\ud{\eta'}.\no
\end{eqnarray}
Based on the $\e$-Milne problem with geometric correction of $\gg$ as
\begin{eqnarray}
\ua\frac{\p\v(\eta',\vuu)}{\p\eta'}
+G(\eta')\bigg(\ub^2\frac{\p\gg(\eta',\vuu)}{\p\ua}-\ua\ub\frac{\p\gg(\eta',\vuu)}{\p\ub}\bigg)
+\nu\gg(\eta',\vuu)-K[\gg](\eta',\vuu)=0,
\end{eqnarray}
we have
\begin{eqnarray}
\frac{\p\v(\eta',\vuu)}{\p\eta'}=-\frac{1}{\ua}
\bigg(G(\eta')\bigg(\ub^2\frac{\p\gg(\eta',\vuu)}{\p\ua}-\ua\ub\frac{\p\gg(\eta',\vuu)}{\p\ub}\bigg)
+\nu\gg(\eta',\vuu)-K[\gg](\eta',\vuu)\bigg).
\end{eqnarray}
Hence, we have
\begin{eqnarray}
\tilde A&:=&\int_{\r^2}\Big(1-\chi(\ua)\Big)
k(\vuu,\vvv')\frac{\gg(\eta',\vuu)}{\p\eta'}\ud{\vuu}\\
&=&-\int_{\r^2}\Big(1-\chi(\ua)\Big)
k(\vuu,\vvv')\frac{1}{\ua}\bigg(\nu\gg(\eta',\vuu)-K[\gg](\eta',\vuu)\bigg)\ud{\vuu}\no\\
&&-\int_{\r^2}\Big(1-\chi(\ua)\Big)
k(\vuu,\vvv')\frac{1}{\ua}
G(\eta')\bigg(\ub^2\frac{\p\gg(\eta',\vuu)}{\p\ua}-\ua\ub\frac{\p\gg(\eta',\vuu)}{\p\ub}\bigg)\ud{\vuu}\no\\
&=&\tilde A_1+\tilde A_2.\no
\end{eqnarray}
Using $\gg$ estimates and $\abs{\ua}\geq\d$, we may directly obtain
\begin{eqnarray}
&&\abs{\bvvp\tilde A_1}\\
&\leq&\abs{\bvvp\int_{\r^2}\Big(1-\chi(\ua)\Big)
k(\vuu,\vvv')\frac{1}{\ua}\bigg(\nu\gg(\eta',\vuu)-K[\gg](\eta',\vuu)\bigg)\ud{\vuu}}\no\\
&\leq&\frac{C}{\d}\lnnmv{\v}.\no
\end{eqnarray}
On the other hand, an integration by parts yields
\begin{eqnarray}
\tilde A_2&=&\int_{\r^2}\frac{\p}{\p\ua}\bigg(\frac{\ub^2}{\ua}
G(\eta')\Big(1-\chi(\ua)\Big)
k(\vuu,\vvv')\bigg)\gg(\eta',\vuu)\ud{\vuu}\\
&&-\int_{\r^2}\frac{\p}{\p\ub}\bigg(\ub
G(\eta')\Big(1-\chi(\ua)\Big)
k(\vuu,\vvv')\bigg)\gg(\eta',\vuu)\ud{\vuu},\no
\end{eqnarray}
which further implies
\begin{eqnarray}
\abs{\bvvp\tilde A_2}&\leq&\frac{C}{\d^2}\lnnmv{\gg}.
\end{eqnarray}
As before we can use substitution to show that
\begin{eqnarray}\label{inte 2}
\abs{\int_0^{\eta}\frac{\zeta\Big(\eta',\vvv'(\eta,\vvv;\eta')\Big)}{\va'(\eta,\vvv;\eta')}\exp(-H_{\eta,\eta'})\ud{\eta'}}\leq \abs{\int_0^{\eta}\frac{\nu\Big(\vvv'(\eta,\vvv;\eta')\Big)}{\va'(\eta,\vvv;\eta')}\exp(-H_{\eta,\eta'})\ud{\eta'}}\leq 1,
\end{eqnarray}
and $\abs{\vvv}$ is a constant along the characteristics. Then we have
\begin{eqnarray}
\abs{\bvv I_2}&\leq&\frac{C}{\d^2}\lnnmv{\v}.\no
\end{eqnarray}
\ \\
Step 3: Estimate of $I_3$ for $0\leq\va\leq\d_0$, $\chi(\ua)=1$ and $\sqrt{\e\eta'}\vb'\geq\va'$.\\
Based on (\ref{pt 01}), this implies
\begin{eqnarray}
\zeta(\eta',\vvv')\leq C\sqrt{\e\eta'}\vb'.\no
\end{eqnarray}
Also, we know that
\begin{eqnarray}
\zeta(\eta',\vuu)\geq C\sqrt{\e\eta'}\abs{\vuu}.\no
\end{eqnarray}
Then we may decompose
\begin{eqnarray}
\tilde A&:=&\int_{\r^2}\frac{\zeta(\eta',\vvv')}{\zeta(\eta',\vuu)}\chi(\ua)k(\vuu,\vvv')\a(\eta',\vuu)\ud{\vuu}\\
&\leq&\int_{\abs{\vuu}\geq\sqrt{\d}}\frac{\vb'}{\abs{\vuu}}\chi(\ua)k(\vuu,\vvv')\a(\eta',\vuu)\ud{\vuu}
+\int_{\abs{\vuu}\leq\sqrt{\d}}\frac{\vb'}{\abs{\vuu}}\chi(\ua)k(\vuu,\vvv')\a(\eta',\vuu)\ud{\vuu}\no\\
&=&\vb'\Big(\tilde A_1+\tilde A_2\Big).\no
\end{eqnarray}
Using Lemma \ref{wellposedness prelim lemma 8}, we directly estimate
\begin{eqnarray}
\abs{\bvvp\tilde A_1}&\leq&C\abs{\bvvp\int_{\abs{\vuu}\geq\sqrt{\d}}\frac{1}{\abs{\vuu}}\chi(\ua)k(\vuu,\vvv')\a(\eta',\vuu)\ud{\vuu}}\\
&\leq&\frac{C}{\sqrt{\d}}\abs{\bvvp\int_{\abs{\ua}\leq2\d}k(\vuu,\vvv')\a(\eta',\vuu)\ud{\vuu}}\no\\
&\leq&C\sqrt{\d}\lnnmv{\a}.\no
\end{eqnarray}
Also, based on Lemma \ref{wellposedness prelim lemma 8}, we obtain
\begin{eqnarray}
\abs{\bvvp\tilde A_2}&\leq&C\abs{\bvvp\int_{\abs{\vuu}\leq\sqrt{\d}}\frac{1}{\abs{\vuu}}k(\vuu,\vvv')\a(\eta',\vuu)\ud{\vuu}}\\
&\leq&C\d\lnnmv{\a},\no
\end{eqnarray}
Hence, since $\abs{\vvv}$ is a constant along the characteristics, we have
\begin{eqnarray}
\abs{\bvv I_3}&\leq&C\sqrt{\d}\lnnmv{\a}\bigg(\int_0^{\eta}\frac{\vb'}{\va'}\exp(-H_{\eta,\eta'})\ud{\eta'}\bigg)\\
&\leq&C\sqrt{\d}\lnnmv{\a}\bigg(\int_0^{\eta}\frac{\nu(\vvv')}{\va'}\exp(-H_{\eta,\eta'})\ud{\eta'}\bigg)\leq C\sqrt{\d}\lnnmv{\a}.\no
\end{eqnarray}
\ \\
Step 4: Estimate of $I_4$ for $0\leq\va\leq\d_0$, $\chi(\ua)=1$, $\sqrt{\e\eta'}\vb'\leq\va'$ and $\va^2\leq\e(\eta-\eta')\vb^2$.\\
Based on (\ref{pt 01}), this implies
\begin{eqnarray}
\zeta(\eta',\vvv')\leq C\va'.
\end{eqnarray}
Based on (\ref{pt 03}), we have
\begin{eqnarray}
-H_{\eta,\eta'}=-\int_{\eta'}^{\eta}\frac{\nu(\vvv)}{\va'(y)}\ud{y}&\leq&-\frac{\nu(\vvv)(\eta-\eta')}{2\vb\sqrt{\e(\eta-\eta')}}
\leq-\frac{C'\nu(\vvv)}{\vb}\sqrt{\frac{\eta-\eta'}{\e}}.
\end{eqnarray}
Hence, since $\abs{\vvv}$ is a constant along the characteristics, we know
\begin{eqnarray}
&&\abs{\bvv I_4}\\
&\leq&C\abs{\int_0^{\eta}\bigg(\bvvp\int_{\r^2}\frac{\zeta(\eta',\vvv')}{\zeta(\eta',\vuu)}\chi(\ua)k(\vuu,\vvv')\a(\eta',\vuu)\ud{\vuu}\bigg)
\frac{1}{\va'}\exp(-H_{\eta,\eta'})\ud{\eta'}}\no\\
&\leq&C\abs{\int_0^{\eta}\bigg(\bvvp\int_{\r^2}\frac{1}{\sqrt{\e\eta'}\abs{\vuu}}\chi(\ua)k(\vuu,\vvv')\a(\eta',\vuu)\ud{\vuu}\bigg)
\frac{\va'}{\va'}\exp(-H_{\eta,\eta'})\ud{\eta'}}\no\\
&=&C\abs{\int_0^{\eta}\bigg(\bvvp\int_{\r^2}\frac{1}{\abs{\vuu}}\chi(\ua)k(\vuu,\vvv')\a(\eta',\vuu)\ud{\vuu}\bigg)
\frac{1}{\sqrt{\e\eta'}}\exp(-H_{\eta,\eta'})\ud{\eta'}}\no
\end{eqnarray}
Using a similar argument as in Step 3, we obtain
\begin{eqnarray}
\abs{\bvvp\int_{\r^2}\frac{1}{\abs{\vuu}}\chi(\ua)k(\vuu,\vvv')\a(\eta',\vuu)\ud{\vuu}}\leq C\sqrt{\d}\lnnmv{\a}.
\end{eqnarray}
Hence, we have
\begin{eqnarray}
\abs{\bvv I_4}&\leq&C\sqrt{\d}\lnnmv{\a}\bigg(\int_0^{\eta}\frac{1}{\sqrt{\e\eta'}}\exp(-H_{\eta,\eta'})\ud{\eta'}\bigg)\no\\
&\leq&C\sqrt{\d}\lnnmv{\a}\int_0^{\eta}\frac{1}{\sqrt{\e\eta'}}\exp\bigg(-\frac{C'\nu(\vvv)}{\vb}\sqrt{\frac{\eta-\eta'}{\e}}\bigg)\ud{\eta'}\no
\end{eqnarray}
Define $z=\dfrac{\eta'}{\e}$, which implies $\ud{\eta'}=\e\ud{z}$. Substituting this into above integral, we have
\begin{eqnarray}
&&\abs{\bvv I_4}\\
&\leq&C\sqrt{\d}\lnnmv{\a}\int_0^{\frac{\eta}{\e}}\frac{1}{\sqrt{z}}\exp\bigg(-\frac{C'\nu(\vvv)}{\vb}\sqrt{\frac{\eta}{\e}-z}\bigg)\ud{z}\no\\
&=&C\sqrt{\d}\lnnmv{\a}\Bigg(\int_0^{1}\frac{1}{\sqrt{z}}\exp\bigg(-\frac{C'\nu(\vvv)}{\vb}\sqrt{\frac{\eta}{\e}-z}\bigg)\ud{z}
+\int_1^{\frac{\eta}{\e}}\frac{1}{\sqrt{z}}\exp\bigg(-\frac{C'\nu(\vvv)}{\vb}\sqrt{\frac{\eta}{\e}-z}\bigg)\ud{z}\Bigg).\no
\end{eqnarray}
We can estimate these two terms separately.
\begin{eqnarray}
\int_0^{1}\frac{1}{\sqrt{z}}\exp\bigg(-\frac{C'\nu(\vvv)}{\vb}\sqrt{\frac{\eta}{\e}-z}\bigg)\ud{z}&\leq&\int_0^{1}\frac{1}{\sqrt{z}}\ud{z}=2.
\end{eqnarray}
\begin{eqnarray}
\int_1^{\frac{\eta}{\e}}\frac{1}{\sqrt{z}}\exp\bigg(-\frac{C'\nu(\vvv)}{\vb}\sqrt{\frac{\eta}{\e}-z}\bigg)\ud{z}
&\leq&\int_1^{\frac{\eta}{\e}}\exp\bigg(-\frac{C'\nu(\vvv)}{\vb}\sqrt{\frac{\eta}{\e}-z}\bigg)\ud{z}\\
&\overset{t^2=\frac{\eta}{\e}-z}{\leq}&2\int_0^{\infty}t\ue^{-\frac{C'\nu(\vvv)}{\vb}t}\ud{t}<C\left(\frac{\vb}{\nu(\vvv)}\right)^2\leq C.\no
\end{eqnarray}
Therefore, we know
\begin{eqnarray}
\abs{\bvv I_4}&\leq&C\sqrt{\d}\lnnmv{\a}.
\end{eqnarray}
\ \\
Step 5: Estimate of $I_5$ for $0\leq\va\leq\d_0$, $\chi(\ua)=1$, $\sqrt{\e\eta'}\vb'\leq\va'$ and $\va^2\geq\e(\eta-\eta')\vb^2$.\\
Based on (\ref{pt 01}), this implies
\begin{eqnarray}
\zeta(\eta',\vvv')\leq C\va'.\no
\end{eqnarray}
Based on (\ref{pt 03}), we have
\begin{eqnarray}
-H_{\eta,\eta'}=-\int_{\eta'}^{\eta}\frac{\nu(\vvv')}{\va'(y)}\ud{y}&\leq&-\frac{C\nu(\vvv')(\eta-\eta')}{\va}.
\end{eqnarray}
Hence, we know
\begin{eqnarray}
&&\abs{\bvv I_5}\\
&\leq&C\abs{\int_0^{\eta}\bigg(\bvvp\int_{\r^2}\frac{\zeta(\eta',\vvv')}{\zeta(\eta',\vuu)}\chi(\ua)k(\vuu,\vvv')\a(\eta',\vuu)\ud{\vuu}\bigg)
\frac{1}{\va'}\exp(-H_{\eta,\eta'})\ud{\eta'}}\no\\
&\leq&C\abs{\int_0^{\eta}\bigg(\bvvp\int_{\r^2}\frac{1}{\zeta(\eta',\vuu)}\chi(\ua)k(\vuu,\vvv')\a(\eta',\vuu)\ud{\vuu}\bigg)
\frac{\va'}{\va'}\exp(-H_{\eta,\eta'})\ud{\eta'}}\no\\
&\leq&C\abs{\int_0^{\eta}\bigg(\bvvp\int_{\r^2}\frac{1}{\zeta(\eta',\vuu)}\chi(\ua)k(\vuu,\vvv')\a(\eta',\vuu)\ud{\vuu}\bigg)
\exp\bigg(-\frac{C\nu(\vvv')(\eta-\eta')}{\va}\bigg)\ud{\eta'}}\no
\end{eqnarray}
Using H\"{o}lder's inequality, we obtain
\begin{eqnarray}
&&\abs{\bvvp\int_{\r^2}\frac{1}{\zeta(\eta',\vuu)}\chi(\ua)k(\vuu,\vvv')\a(\eta',\vuu)\ud{\vuu}}\\
&=&\abs{\bvvp\int_{\r^2}\frac{1}{\zeta^{\frac{1}{1+s}}(\eta',\vuu)}\frac{1}{\zeta^{\frac{s}{1+s}}(\eta',\vuu)}\chi(\ua)
k(\vuu,\vvv')\a(\eta',\vuu)\ud{\vuu}}\no\\
&\leq&\frac{C}{(\e\eta')^{\frac{s}{1+s}}}\abs{\bvvp\int_{\r^2}\frac{1}{\zeta^{\frac{1}{1+s}}(\eta',\vuu)}\frac{1}{\abs{\ub}^{\frac{s}{1+s}}}\chi(\ua)
k(\vuu,\vvv')\a(\eta',\vuu)\ud{\vuu}}\no\\
&\leq&\frac{C}{(\e\eta')^{\frac{s}{1+s}}}\abs{\bvvp\int_{\r^2}\frac{1}{\zeta^{\frac{1}{1+s}}(\eta',\vuu)}\frac{1}{\abs{\ub}^{\frac{s}{1+s}}}
\frac{1}{\abs{\ub}^{\frac{1}{1+2s}}}\chi(\ua)
\abs{\ub}^{\frac{1}{1+2s}}k(\vuu,\vvv')\a(\eta',\vuu)\ud{\vuu}}\no\\
&\leq&\frac{C\br{\vvv'}^{\frac{1}{1+2s}}}{(\e\eta')^{\frac{s}{1+s}}}
\bigg(\int_{\r^2}\frac{1}{\zeta(\eta',\vuu)}\frac{1}{\abs{\ub}^{s+\frac{1+s}{1+2s}}}\chi^{1+s}(\ua)\ud{\vuu}\bigg)^{\frac{1}{1+s}}\no\\
&&\times\bigg(\left(\bvvp\right)^{\frac{1+s}{s}}
\int_{\r^2}\abs{\frac{\ub^{\frac{1}{1+2s}}}{\br{\vvv'}^{\frac{1}{1+2s}}}k(\vuu,\vvv')\a(\eta',\vuu)}^{\frac{1+s}{s}}\ud{\vuu}\bigg)^{\frac{s}{1+s}},\no
\end{eqnarray}
where $0<s<<1$. Note the fact that in 2D, $k(\vuu,\vvv')$ does not contain the singularity of $\abs{\vuu-\vvv'}^{-1}$. Using a similar argument as in Step 3, we may directly compute
\begin{eqnarray}
&&\bigg(\left(\bvvp\right)^{\frac{1+s}{s}}
\int_{\r^2}\abs{\frac{\ub^{\frac{1}{1+2s}}}{\br{\vvv'}^{\frac{1}{1+2s}}}k(\vuu,\vvv')\a(\eta',\vuu)}^{\frac{1+s}{s}}\ud{\vuu}\bigg)^{\frac{s}{1+s}}\\
&\leq&\bigg(\left(\bvvp\right)^{\frac{1+s}{s}}
\int_{\r^2}\abs{\frac{\br{\vuu}^{\frac{1}{1+2s}}}{\br{\vvv'}^{\frac{1}{1+2s}}}k(\vuu,\vvv')\a(\eta',\vuu)}^{\frac{1+s}{s}}\ud{\vuu}\bigg)^{\frac{s}{1+s}}\no\\
&\leq& C\lnnmv{\a}.\no
\end{eqnarray}
Then we estimate
\begin{eqnarray}
\int_{\r^2}\frac{1}{\zeta(\eta',\vuu)}\frac{1}{\abs{\ub}^{s+\frac{1+s}{1+2s}}}\chi^{1+s}(\ua)\ud{\vuu}
&\leq&
\int_{\r}\bigg(\int_{-\d}^{\d}\frac{1}{\zeta(\eta',\vuu)}\chi(\ua)\ud{\ua}\bigg)\frac{1}{\abs{\ub}^{s+\frac{1+s}{1+2s}}}\ud{\ub}.
\end{eqnarray}
We may directly compute
\begin{eqnarray}
\int_{-\d}^{\d}\frac{1}{\zeta(\eta',\vuu)}\chi(\ua)\ud{\ua}
&\leq&\int_{-\d}^{\d}\frac{1}{\sqrt{\rk^2(\ua)^2+\left(\rk^2-(\rk-\e\eta')^2\right)(\ub)^2}}\chi(\ua)\ud{\ua}\\
&\leq&C\int_{-\d}^{\d}\frac{1}{\sqrt{(\ua)^2+r^2(\ub)^2}}\ud{\ua},\no
\end{eqnarray}
where
\begin{eqnarray}
r=\sqrt{\frac{\rk^2-(\rk-\e\eta')^2}{\rk^2}}\leq C\sqrt{\e\eta'}.
\end{eqnarray}
We may further compute
\begin{eqnarray}
\int_{-\d}^{\d}\frac{1}{\sqrt{(\ua)^2+r^2(\ub)^2}}\ud{\ua}&=&2\int_{0}^{\d}\frac{1}{\sqrt{(\ua)^2+r^2(\ub)^2}}\ud{\ua}\\
&=&2\bigg(\ln\left(\ua+\sqrt{r^2(\ub)^2+(\ua)^2}\right)-\ln(r\ub)\bigg)\bigg|_0^{\d}\no\\
&=&2\bigg(\ln\left(\d+\sqrt{r^2(\ub)^2+\d^2}\right)-\ln(r\ub)\bigg)\no\\
&\leq&C\bigg(1+\ln(r)+\ln(\ub)\bigg).\no
\end{eqnarray}
Then considering $s+\dfrac{1+s}{1+2s}>1$, we know
\begin{eqnarray}
\int_{\r^2}\frac{1}{\zeta(\eta',\vuu)}\frac{1}{\abs{\ub}^{s+\frac{1+s}{1+2s}}}\chi^{1+s}(\ua)\ud{\vuu}
&\leq&
C\int_{\r}\bigg(1+\ln(r)+\ln(\ub)\bigg)\frac{1}{\abs{\ub}^{s+\frac{1+s}{1+2s}}}\ud{\ub}\\
&\leq&C\bigg(1+\abs{\ln(\e)}+\abs{\ln{(\eta')}}\bigg).\no
\end{eqnarray}
Hence, we can obtain
\begin{eqnarray}
\\
\abs{\bvvp\int_{\r^2}\frac{1}{\zeta(\eta',\vuu)}\chi(\ua)k(\vuu,\vvv')\a(\eta',\vuu)\ud{\vuu}}\leq \frac{C\br{\vvv'}^{\frac{1}{1+2s}}}{(\e\eta')^{\frac{s}{1+s}}}\lnnmv{\a}\bigg(1+\abs{\ln(\e)}+\abs{\ln{(\eta')}}\bigg).\no
\end{eqnarray}
Hence, we know
\begin{eqnarray}
\\
\abs{\bvv I_5}&\leq&\frac{C}{\e^{s}}\lnnmv{\a}\int_0^{\eta}\frac{\br{\vvv'}^{\frac{1}{1+2s}}}{\eta'^{\frac{s}{1+s}}}\bigg(1+\abs{\ln(\e)}+\abs{\ln(\eta')}\bigg)
\exp\left(-\frac{C\nu(\vvv')(\eta-\eta')}{\va}\right)\ud{\eta'}.\no
\end{eqnarray}
Hence, we first estimate
\begin{eqnarray}
\abs{\int_0^{\eta}\frac{\br{\vvv'}^{\frac{1}{1+2s}}}{\eta'^{\frac{s}{1+s}}}\abs{\ln(\eta')}
\exp\left(-\frac{C\nu(\vvv')(\eta-\eta')}{\va}\right)\ud{\eta'}}.
\end{eqnarray}
If $0\leq\eta\leq 2$, using H\"{o}lder's inequality, we have
\begin{eqnarray}
&&\abs{\int_0^{\eta}\frac{\br{\vvv'}^{\frac{1}{1+2s}}}{\eta'^{\frac{s}{1+s}}}\abs{\ln(\eta')}
\exp\left(-\frac{C\nu(\vvv')(\eta-\eta')}{\va}\right)\ud{\eta'}}\\
&\leq&\bigg(\int_0^{\eta}\frac{1}{\eta'^{\frac{1}{2}}}\abs{\ln(\eta')}^{\frac{1+2s}{2s}}\ud{\eta'}\bigg)^{\frac{2s}{1+2s}}\bigg(\int_0^{\eta}\br{\vvv'}
\exp\left(-\frac{(1+2s)\nu(\vvv')C(\eta-\eta')}{\va}\right)\ud{\eta'}\bigg)^{\frac{1}{1+2s}}\no\\
&\leq&C\bigg(\frac{\va\br{\vvv'}}{\nu(\vvv')}\bigg)^{\frac{1}{1+2s}}\leq\va^{\frac{1}{1+2s}}\leq C\d_0^{\frac{1}{1+2s}}\leq\sqrt{\d_0}.\no
\end{eqnarray}
If $\eta\geq 2$, it suffices to estimate
\begin{eqnarray}
&&\abs{\int_2^{\eta}\frac{\br{\vvv'}^{\frac{1}{1+2s}}}{\eta'^{\frac{s}{1+s}}}\abs{\ln(\eta')}
\exp\left(-\frac{C\nu(\vvv')(\eta-\eta')}{\va}\right)\ud{\eta'}}\\
&\leq&\ln(L)\abs{\int_2^{\eta}\br{\vvv'}^{\frac{1}{1+2s}}
\exp\left(-\frac{C\nu(\vvv')(\eta-\eta')}{\va}\right)\ud{\eta'}}\leq C\abs{\ln(\e)}\va\leq C\abs{\ln(\e)}\d_0.\no
\end{eqnarray}
With a similar argument, we may justify
\begin{eqnarray}
\abs{\int_0^{\eta}\bigg(1+\abs{\ln(\e)}\bigg)\frac{\br{\vvv'}^{\frac{1}{1+2s}}}{\eta'^{\frac{s}{1+s}}}
\exp\left(-\frac{C\nu(\vvv')(\eta-\eta')}{\va}\right)}\leq C\bigg(\sqrt{\d_0}+\abs{\ln(\e)}\d_0\bigg).
\end{eqnarray}
Hence, we have
\begin{eqnarray}
\abs{\bvv I_5}\leq \frac{C}{\e^{s}}\bigg(\sqrt{\d_0}+\abs{\ln(\e)}\d_0\bigg)\lnnmv{\a}.
\end{eqnarray}
\ \\
Step 6: Synthesis.\\
Collecting all the terms in previous steps, we have proved
\begin{eqnarray}
\abs{\bvv I}
&\leq&\frac{C}{\e^{s}}\Big(1+\abs{\ln(\e)}\Big)\sqrt{\d_0}\lnnmv{\a}+C\sqrt{\d}\lnnmv{\a}\\
&&+\frac{C}{\d^2}\lnnmv{\gg}+\frac{C}{\d_0}\bigg(\lnmv{\frac{\p p}{\p\va}}+\lnmv{\frac{\p p}{\p\vb}}+\lnnmv{\gg}\bigg).\no
\end{eqnarray}

\subsection{Region II: $\va<0$ and $\va^2+\vb^2\geq \vb'^2(\eta,\vvv;L)$}

Based on Lemma \ref{Milne lemma 3} and Lemma \ref{Milne lemma 4},
we can directly obtain
\begin{eqnarray}
\abs{\k[p_{\a}]}&\leq&\lnmv{p_{\a}},\\
\abs{\t[S_{\a}]}&\leq&\lnnmv{\frac{S_{\a}}{\nu}}.
\end{eqnarray}
Hence, we only need to estimate
\begin{eqnarray}
II=\t[\tilde\a]&=&\int_0^{L}\frac{\tilde\a\Big(\eta',\vvv'(\eta,\vvv;\eta')\Big)}{\va'(\eta,\vvv;\eta')}
\exp(-H_{L,\eta'}-\rr[H_{L,\eta}])\ud{\eta'}\\
&&+\int_{\eta}^{L}\frac{\tilde\a\Big(\eta',\rr[\vvv'(\eta,\vvv;\eta')]\Big)}{\va'(\eta,\vvv;\eta')}\exp(\rr[H_{\eta,\eta'}])\ud{\eta'}.\no
\end{eqnarray}
In particular, we can decompose
\begin{eqnarray}
\t[\tilde\a]&=&\int_0^{\eta}\frac{\tilde\a\Big(\eta',\vvv'(\eta,\vvv;\eta')\Big)}{\va'(\eta,\vvv;\eta')}\exp(-H_{L,\eta'}-\rr[H_{L,\eta}])\ud{\eta'}\\
&&+\int_{\eta}^{L}\frac{\tilde\a\Big(\eta',\vvv'(\eta,\vvv;\eta')\Big)}{\va'(\eta,\vvv;\eta')}\exp(-H_{L,\eta'}-\rr[H_{L,\eta}])\ud{\eta'}
+\int_{\eta}^{L}\frac{\tilde\a\Big(\eta',\rr[\vvv'(\eta,\vvv;\eta')]\Big)}{\va'(\eta,\vvv;\eta')}\exp(\rr[H_{\eta,\eta'}])\ud{\eta'}.\no
\end{eqnarray}
The integral $\displaystyle\int_0^{\eta}$ part can be estimated as in Region I, so we only need to estimate the integral $\displaystyle\int_{\eta}^L$ part. Also, noting that fact that
\begin{eqnarray}
\exp(-H_{L,\eta'}-\rr[H_{L,\eta}])\leq \exp(-\rr[H_{\eta',\eta}]),
\end{eqnarray}
we only need to estimate
\begin{eqnarray}
\int_{\eta}^{L}\frac{\tilde\a\Big(\eta',\vvv'(\eta,\vvv;\eta')\Big)}{\va'(\eta,\vvv;\eta')}\exp(-H_{\eta',\eta})\ud{\eta'}.
\end{eqnarray}
Here the proof is almost identical to that in Region I, so we only point out the key differences.\\
\ \\
Step 0: Preliminaries.\\
We need to update one key result. For $0\leq\eta\leq\eta'$,
\begin{eqnarray}
\va'&=&\sqrt{E_1-\vb'^2}=\sqrt{E_1-\bigg(\frac{\rk-\e\eta}{\rk-\e\eta'}\bigg)^2\vb^2}\leq\va.
\end{eqnarray}
Then we have
\begin{eqnarray}\label{pt 04}
-\int_{\eta}^{\eta'}\frac{1}{\va'(y)}\ud{y}&\leq&-\frac{\eta'-\eta}{\va}.
\end{eqnarray}
In the following, we will divide the estimate of $II$ into several cases based on the value of $\va$, $\va'$ and $\e\eta'$. We write
\begin{eqnarray}
II&=&\int_{\eta}^L\id_{\{\va\leq-\d_0\}}+\int_{\eta}^L\id_{\{-\d_0\leq\va\leq0\}}\id_{\{\chi(\ua)<1\}}\\
&&+\int_{\eta}^L\id_{\{-\d_0\leq\va\leq0\}}\id_{\{\chi(\ua)=1\}}\id_{\{\sqrt{\e\eta'}\vb'\geq\va'\}}
+\int_{\eta}^L\id_{\{-\d_0\leq\va\leq0\}}\id_{\{\chi(\ua)=1\}}\id_{\{\sqrt{\e\eta'}\vb'\leq\va'\}}\no\\
&=&II_1+II_2+II_3+II_4.\no
\end{eqnarray}
\ \\
Step 1: Estimate of $II_1$ for $\va\leq-\d_0$.\\
We first estimate $\va'$. Along the characteristics, we know
\begin{eqnarray}
\ue^{-W(\eta')}\vb'=\ue^{-W(\eta)}\vb,
\end{eqnarray}
which implies
\begin{eqnarray}
\abs{\vb'}&=&\ue^{W(\eta')-W(\eta)}\abs{\vb}\leq \ue^{W(L)-W(0)}\abs{\vb}\leq \ue^{W(L)-W(0)}\sqrt{E_1-\d_0^2}.
\end{eqnarray}
Then we can further deduce that
\begin{eqnarray}
\abs{\vb'}\leq \bigg(1-\frac{\e^{\frac{1}{2}}}{\rk}\bigg)^{-1}\sqrt{E_1-\d_0^2}.
\end{eqnarray}
Then we have
\begin{eqnarray}
\va'\geq\sqrt{E_1-\bigg(1-\frac{\e^{\frac{1}{2}}}{\rk}\bigg)^{-2}(E_1-\d_0^2)}\geq \d_0-C\e^{\frac{1}{4}}>\frac{\d_0}{2},
\end{eqnarray}
when $\e$ is sufficiently small. Then this implies that for $\abs{E_1(\eta,\vvv)}\geq \abs{\vb'(\eta,\vvv;L)}$, for $\e$ sufficiently small, we know $\min\va'\geq\d_0$ where $(\eta',\vvv')$ is on the same characteristics as $(\eta,\vvv)$ with $\va'\geq0$.

Similar to the estimate of $I_1$, in this step, we will prove estimates based on the characteristics of $\v$ itself instead of $\a$.
Here, we rewrite the equation (\ref{Milne difference problem}) along the characteristics as
\begin{eqnarray}
\va\frac{\ud{\gg}}{\ud{\eta}}+\nu\gg=K[\v].
\end{eqnarray}
Also, we will still use simple facts (SF):
\begin{itemize}
\item
Based on the well-posedness and decay theorem for $\gg$, we know $\lnmv{\gg}\leq C$.
\item
Based on Lemma \ref{prelim 1}, we get $\lnmv{K[\gg]}\leq\lnmv{\gg}\leq C$ and $\lnmv{\nabla_vK[\gg]}\leq\lnmv{\gg}\leq C$.
\item
Since $E_1$ is conserved along the characteristics, we must have $\abs{\vvv}=\abs{\vvv'}$.
\item
For $\eta'\leq\eta$, we must have $\va'\geq\abs{\va}\geq\d_0$.
\item
Using substitution $y=H_{\eta,\eta'}$, we know
\begin{eqnarray}\label{inte 1}
\abs{\int_{\eta}^L\frac{\nu\Big(\vvv'(\eta,\vvv;\eta')\Big)}{\va'(\eta,\vvv;\eta')}\exp(H_{\eta,\eta'})\ud{\eta'}}\leq\abs{\int_{-\infty}^0\ue^{y}\ud{y}}=1.
\end{eqnarray}
\end{itemize}
For $\va\leq-\d_0$, we do not need the mild formulation for $\a$. Instead, we directly estimate
\begin{eqnarray}
\abs{\bvv II_1} &\leq&\abs{\bvv\frac{\p\gg}{\p\eta}}.
\end{eqnarray}
We rewrite the equation along the characteristics as
\begin{eqnarray}\label{mt 03}
\gg(\eta,\vvv)&=&p\Big(\vvv'(\eta,\vvv;0)\Big)\exp(-H_{L,0}-\rr[H_{L,\eta}])\\
&&+\int_0^{L}\frac{K[\gg]\Big(\eta',\vvv'(\eta,\vvv;\eta')\Big)}{\va'(\eta,\vvv;\eta')}\exp(-H_{L,\eta'}-\rr[H_{L,\eta}])\ud{\eta'}\no\\
&&+\int_{\eta}^{L}\frac{K[\gg]\Big(\eta',\rr[\vvv'(\eta,\vvv;\eta')]\Big)}{\va'(\eta,\vvv;\eta')}\exp(\rr[H_{\eta,\eta'}])\ud{\eta'},
\end{eqnarray}
where $\vvv'(\eta')=\vvv'(\eta,\vvv; \eta')$ satisfying $(\eta',\vvv')$ and $(\eta,\vvv)$ are on the same characteristic with $\va'\geq0$, and
\begin{eqnarray}
H_{t,s}&=&\int_{s}^{t}\frac{\nu(\vvv'(\eta,\vvv;y))}{\va'(\eta,\vvv;y)}\ud{y}.\no
\end{eqnarray}
for any $s,t\geq0$.

Then taking $\eta$ derivative on both sides of (\ref{mt 03}) yields
\begin{eqnarray}
\frac{\p\v}{\p\eta}&=&Y=Y_1+Y_2+Y_3+Y_4+Y_5+Y_6+Y_7+Y_8+Y_9,
\end{eqnarray}
where
\begin{eqnarray}
Y_1&=&\frac{\p p\Big(\vvv'(\eta,\vvv;0)\Big)}{\p\eta}\exp(-H_{L,0}-\rr[H_{L,\eta}]),\\
Y_2&=&-p\Big(\vvv'(\eta,\vvv;0)\Big)\exp(-H_{L,0}-\rr[H_{L,\eta}])\bigg(\frac{\p H_{L,0}}{\p\eta}+\frac{\p \rr[H_{L,\eta}]}{\p\eta}\bigg),\\
Y_3&=&\int_0^{L}\frac{K[\gg]\Big(\eta',\vvv'(\eta,\vvv;\eta')\Big)}{\va'^2(\eta,\vvv;\eta')}\frac{\p\va'(\eta,\vvv;\eta')}{\p\eta}\exp(-H_{L,\eta'}-\rr[H_{L,\eta}])\ud{\eta'},\\
Y_4&=&-\int_0^{L}\frac{K[\gg]\Big(\eta',\vvv'(\eta,\vvv;\eta')\Big)}{\va'(\eta,\vvv;\eta')}\exp(-H_{L,\eta'}-\rr[H_{L,\eta}])\bigg(\frac{\p H_{L,\eta'}}{\p\eta}+\frac{\p \rr[H_{L,\eta}]}{\p\eta}\bigg)\ud{\eta'},\\
Y_5&=&\int_0^{L}\frac{1}{\va'(\eta,\vvv;\eta')}\exp(-H_{L,\eta'}-\rr[H_{L,\eta}])
\bigg(\nabla_{\mathfrak{v}'}K[\gg]\Big(\eta',\vvv'(\eta,\vvv;\eta')\Big)\frac{\p\vvv'(\eta,\vvv;\eta')}{\p\eta}\bigg)\ud{\eta'},\\
Y_6&=&\int_{\eta}^{L}\frac{K[\gg]\Big(\eta',\rr[\vvv'(\eta,\vvv;\eta')]\Big)}{\va'^2(\eta,\vvv;\eta')}\frac{\p\va'(\eta,\vvv;\eta')}{\p\eta}\exp(\rr[H_{\eta,\eta'}])\ud{\eta'},\\
Y_7&=&\int_{\eta}^{L}\frac{K[\gg]\Big(\eta',\rr[\vvv'(\eta,\vvv;\eta')]\Big)}{\va'(\eta,\vvv;\eta')}\exp(\rr[H_{\eta,\eta'}])\frac{\rr[H_{\eta,\eta'}]}{\p\eta}\ud{\eta'},\\
Y_8&=&\int_{\eta}^{L}\frac{1}{\va'(\eta,\vvv;\eta')}\exp(\rr[H_{\eta,\eta'}])
\bigg(\nabla_{\mathfrak{v}'}K[\gg]\Big(\eta',\vvv'(\eta,\vvv;\eta')\Big)\frac{\p\vvv'(\eta,\vvv;\eta')}{\p\eta}\bigg)\ud{\eta'},\\
Y_9&=&-\frac{K[\eta,\rr[\vvv]]}{\va}.
\end{eqnarray}
We need to estimate each term. Since the techniques are very similar to the estimate of $I_1$ without introducing new tricks, we just list the results here:
\begin{eqnarray}
\abs{\bvv Y_1}&\leq&C\left(1+\frac{1}{\d_0}\right)\bigg(\lnmv{\frac{\p p}{\p\va}}+\lnmv{\frac{\p p}{\p\vb}}\bigg),\\
\abs{\bvv Y_2}&\leq&C\left(1+\frac{1}{\d_0}\right)\lnmv{p},\\
\abs{\bvv Y_3}&\leq&C\left(1+\frac{1}{\d_0}\right)\lnnmv{\gg},\\
\abs{\bvv Y_4}&\leq&C\left(1+\frac{1}{\d_0}\right)\lnnmv{\gg},\\
\abs{\bvv Y_5}&\leq&C\left(1+\frac{1}{\d_0}\right)\lnnmv{\gg},\\
\abs{\bvv Y_6}&\leq&C\left(1+\frac{1}{\d_0}\right)\lnnmv{\gg},\\
\abs{\bvv Y_7}&\leq&C\left(1+\frac{1}{\d_0}\right)\lnnmv{\gg},\\
\abs{\bvv Y_8}&\leq&C\left(1+\frac{1}{\d_0}\right)\lnnmv{\gg}.\\
\abs{\bvv Y_9}&\leq&C\left(1+\frac{1}{\d_0}\right)\lnnmv{\gg}.
\end{eqnarray}
In summary, we have
\begin{eqnarray}
\abs{\bvv II_1} &\leq& \frac{C}{\d_0}\bigg(\lnmv{\frac{\p p}{\p\va}}+\lnmv{\frac{\p p}{\p\vb}}+\lnnmv{\gg}\bigg).
\end{eqnarray}
\ \\
Step 2: Estimate of $II_2$ for $-\d_0\leq\va\leq0$ and $\chi(\ua)<1$.\\
This is similar to the estimate of $I_2$ based on the integral
\begin{eqnarray}
\int_{\eta}^{L}\frac{\nu\Big(\vvv'(\eta,\vvv;\eta')\Big)}{\va'(\eta,\vvv;\eta')}\exp(-H_{\eta',\eta})\ud{\eta'}\leq 1.
\end{eqnarray}
Then we have
\begin{eqnarray}
\abs{\bvv II_2}&\leq&\frac{C}{\d^2}\lnnmv{\v}.\no
\end{eqnarray}
\ \\
Step 3: Estimate of $II_3$ for $-\d_0\leq\va\leq0$, $\chi(\ua)=1$ and $\sqrt{\e\eta'}\vb'\geq\va'$.\\
This is identical to the estimate of $I_3$, we have
\begin{eqnarray}
\abs{\bvv II_3}&\leq& C\sqrt{\d}\lnnmv{\a}.
\end{eqnarray}
\ \\
Step 4: Estimate of $II_4$ for $-\d_0\leq\va\leq0$, $\chi(\ua)=1$ and $\sqrt{\e\eta'}\vb'\leq\va'$.\\
This step is different. We do not need to further decompose the cases.
Based on (\ref{pt 04}), we have,
\begin{eqnarray}
-H_{\eta,\eta'}&\leq&-\frac{\nu(\vvv)(\eta'-\eta)}{\va}.
\end{eqnarray}
Then following the same argument in estimating $I_5$, we know
\begin{eqnarray}
\\
\abs{\bvv II_4}&\leq&\frac{C}{\e^{s}}\lnnmv{\a}\int_{\eta}^L\frac{\br{\vvv'}^{\frac{1}{1+2s}}}{\eta'^{\frac{s}{1+s}}}\bigg(1+\abs{\ln(\e)}+\abs{\ln(\eta')}\bigg)
\exp\left(-\frac{C\nu(\vvv')(\eta-\eta')}{\va}\right)\ud{\eta'}.\no
\end{eqnarray}
Hence, we first estimate
\begin{eqnarray}
\abs{\int_{\eta}^L\frac{\br{\vvv'}^{\frac{1}{1+2s}}}{\eta'^{\frac{s}{1+s}}}\abs{\ln(\eta')}
\exp\left(-\frac{C\nu(\vvv')(\eta-\eta')}{\va}\right)\ud{\eta'}}.
\end{eqnarray}
If $\eta\geq 2$, we have
\begin{eqnarray}
&&\abs{\int_{\eta}^L\frac{\br{\vvv'}^{\frac{1}{1+2s}}}{\eta'^{\frac{s}{1+s}}}\abs{\ln(\eta')}
\exp\left(-\frac{C\nu(\vvv')(\eta-\eta')}{\va}\right)\ud{\eta'}}\\
&\leq&\ln(L)\abs{\int_{\eta}^L\br{\vvv'}^{\frac{1}{1+2s}}
\exp\left(-\frac{C\nu(\vvv')(\eta-\eta')}{\va}\right)\ud{\eta'}}\leq C\abs{\ln(\e)}\va\leq C\abs{\ln(\e)}\d_0.\no
\end{eqnarray}
If $0\leq\eta\leq 2$, using H\"{o}lder's inequality, it suffices to estimate
\begin{eqnarray}
&&\abs{\int_0^{2}\frac{\br{\vvv'}^{\frac{1}{1+2s}}}{\eta'^{\frac{s}{1+s}}}\abs{\ln(\eta')}
\exp\left(-\frac{C\nu(\vvv')(\eta-\eta')}{\va}\right)\ud{\eta'}}\\
&\leq&\bigg(\int_0^{2}\frac{1}{\eta'^{\frac{1}{2}}}\abs{\ln(\eta')}^{\frac{1+2s}{2s}}\ud{\eta'}\bigg)^{\frac{2s}{1+2s}}\bigg(\int_0^{2}\br{\vvv'}
\exp\left(-\frac{(1+2s)\nu(\vvv')C(\eta-\eta')}{\va}\right)\ud{\eta'}\bigg)^{\frac{1}{1+2s}}\no\\
&\leq&C\bigg(\frac{\va\br{\vvv'}}{\nu(\vvv')}\bigg)^{\frac{1}{1+2s}}\leq\va^{\frac{1}{1+2s}}\leq C\d_0^{\frac{1}{1+2s}}\leq\sqrt{\d_0}.\no
\end{eqnarray}
With a similar argument, we may justify
\begin{eqnarray}
\abs{\int_{\eta}^L\bigg(1+\abs{\ln(\e)}\bigg)\frac{\br{\vvv'}^{\frac{1}{1+2s}}}{\eta'^{\frac{s}{1+s}}}
\exp\left(-\frac{C\nu(\vvv')(\eta-\eta')}{\va}\right)}\leq C\bigg(\sqrt{\d_0}+\abs{\ln(\e)}\d_0\bigg).
\end{eqnarray}
Hence, we have
\begin{eqnarray}
\abs{\bvv II_4}\leq \frac{C}{\e^{s}}(1+\abs{\ln(\e)})\sqrt{\d_0}\lnnmv{\a}.
\end{eqnarray}
Hence, we have
\begin{eqnarray}
\abs{\bvv II_5}\leq \frac{C}{\e^{s}}\bigg(\sqrt{\d_0}+\abs{\ln(\e)}\d_0\bigg)\lnnmv{\a}.
\end{eqnarray}
\ \\
Step 5: Synthesis.\\
Collecting all the terms in previous steps, we have proved
\begin{eqnarray}
\abs{\bvv II}
&\leq&\frac{C}{\e^{s}}\Big(1+\abs{\ln(\e)}\Big)\sqrt{\d_0}\lnnmv{\a}+C\sqrt{\d}\lnnmv{\a}\\
&&+\frac{C}{\d^2}\lnnmv{\gg}+\frac{C}{\d_0}\bigg(\lnmv{\frac{\p p}{\p\va}}+\lnmv{\frac{\p p}{\p\vb}}+\lnnmv{\gg}\bigg).\no
\end{eqnarray}

\subsection{Region III: $\va<0$ and $\va^2+\vb^2\leq \vb'^2(\eta,\vvv;L)$}

Based on Lemma \ref{Milne lemma 3} and Lemma \ref{Milne lemma 4}, we still have
\begin{eqnarray}
\abs{\k[p_{\a}]}&\leq&\lnmv{p_{\a}},\\
\abs{\t[S_{\a}]}&\leq&\lnnmv{\frac{S_{\a}}{\nu}}.
\end{eqnarray}
Hence, we only need to estimate
\begin{eqnarray}
III=\t[\tilde\a]&=&\int_0^{\eta^+}\frac{\tilde\a\Big(\eta',\vvv'(\eta,\vvv;\eta')\Big)}{\va'(\eta,\vvv;\eta')}
\exp(-H_{\eta^+,\eta'}-\rr[H_{\eta^+,\eta}])\ud{\eta'}\\
&&+\int_{\eta}^{\eta^+}\frac{\tilde\a\Big(\eta',\rr[\vvv'(\eta,\vvv;\eta')]\Big)}{\va'(\eta,\vvv;\eta')}\exp(\rr[H_{\eta,\eta'}])\ud{\eta'}.\no
\end{eqnarray}
In particular, we can decompose
\begin{eqnarray}
\t[\tilde\a]&=&\int_0^{\eta}\frac{\tilde\a\Big(\eta',\vvv'(\eta,\vvv;\eta')\Big)}{\va'(\eta,\vvv;\eta')}\exp(-H_{\eta^+,\eta'}-\rr[H_{\eta^+,\eta}])\ud{\eta'}\\
&&+\int_{\eta}^{\eta^+}\frac{\tilde\a\Big(\eta',\vvv'(\eta,\vvv;\eta')\Big)}{\va'(\eta,\vvv;\eta')}\exp(-H_{\eta^+,\eta'}-\rr[H_{\eta^+,\eta}])\ud{\eta'}\no\\
&&+\int_{\eta}^{\eta^+}\frac{\tilde\a\Big(\eta',\rr[\vvv'(\eta,\vvv;\eta')]\Big)}{\va'(\eta,\vvv;\eta')}\exp(\rr[H_{\eta,\eta'}])\ud{\eta'}.\no
\end{eqnarray}
Then the integral $\displaystyle\int_0^{\eta}(\cdots)$ is similar to the argument in Region I, and the integral $\displaystyle\int_{\eta}^{\eta^+}(\cdots)$ is similar to the argument in Region II. The only difference is in Step 1 when estimating $\displaystyle\int_{\eta}^{\eta^+}(\cdots)$ part for $\eta\leq-\d_0$. Here, we introduce a special trick.

We first estimate $\va$ in term of $\vb$. Along the characteristics, we know
\begin{eqnarray}
\ue^{-W(L)}\vb'(L)=\ue^{-W(\eta)}\vb,
\end{eqnarray}
which implies
\begin{eqnarray}
\abs{\vb'(L)}&=&\ue^{W(L)-W(\eta)}\abs{\vb}\leq \ue^{W(L)-W(0)}\abs{\vb}= \bigg(1-\frac{\e^{\frac{1}{2}}}{\rk}\bigg)^{-1}\abs{\vb}.
\end{eqnarray}
Then we can further deduce that
\begin{eqnarray}
\va^2+\vb^2\leq \bigg(1-\frac{\e^{\frac{1}{2}}}{\rk}\bigg)^{-2}\vb^2\leq \bigg(1-\frac{\e^{\frac{1}{2}}}{\rk}\bigg)^{-2}\vb^2.
\end{eqnarray}
Then we have
\begin{eqnarray}
\abs{\va}\leq\sqrt{\bigg(1-\frac{\e^{\frac{1}{2}}}{\rk}\bigg)^{-2}\vb^2-\vb^2}\leq \e^{\frac{1}{4}}\abs{\vb}\leq\d_0\abs{\vb},
\end{eqnarray}
when $\e$ is sufficiently small.
\begin{itemize}
\item
Therefore, if $\abs{\vb}\leq 1$, then Step 1 is not necessary at all since we already have $\abs{\va}\leq\d_0$. We directly apply the argument in estimating II to obtain
\begin{eqnarray}
\abs{\bvv III}
&\leq&\frac{C}{\e^{s}}\Big(1+\abs{\ln(\e)}\Big)\sqrt{\d_0}\lnnmv{\a}+C\sqrt{\d}\lnnmv{\a}+\frac{C}{\d^2}\lnnmv{\gg}.
\end{eqnarray}
\item
However, if ${\vb}\geq1$, let $(\eta,\va,\vb)$ and $(\tilde\eta,-\d_0,\tilde\vb)$ be on the same characteristics. Then we have the mild formulation
\begin{eqnarray}
\gg(\eta,\vvv)=\gg(\tilde\eta,-\d_0,\tilde\vb)\exp(-H_{\tilde\eta,\eta})+\int_{\eta}^{\tilde\eta}
\frac{K[\gg]\Big(\eta',\vvv'(\eta,\vvv;\eta')\Big)}{\va'(\eta,\vvv;\eta')}\exp(H_{\eta',\eta})\ud{\eta'}.
\end{eqnarray}
In other words, we try to use a mild formulation and avoid go through $\eta^+$ point. Then similar to the estimate of $II_1$, taking $\eta$ derivative in the mild formulation, we obtain
\begin{eqnarray}
\abs{\bv\frac{\p\gg}{\p\eta}}&\leq& C\bigg(1+\frac{1}{\d_0}\bigg)\lnnmv{\gg}+C\abs{\bv\zeta(\tilde\eta,-\d_0,\tilde\vb)\frac{\p\gg(\tilde\eta,-\d_0,\tilde\vb)}{\p\eta}}.\no
\end{eqnarray}
Also, we may directly verify that
\begin{eqnarray}
&&\abs{\bv\zeta(\tilde\eta,-\d_0,\tilde\vb)\frac{\p\gg(\tilde\eta,-\d_0,\tilde\vb)}{\p\eta}}\\
&\leq&\abs{\bv\zeta(\tilde\eta,-\d_0,\tilde\vb)\frac{\p\gg(\tilde\eta,-\d_0,\tilde\vb)}{\p\tilde\eta}\frac{\p\tilde\eta}{\p\eta}}+
\abs{\bv\zeta(\tilde\eta,-\d_0,\tilde\vb)\frac{\p\gg(\tilde\eta,-\d_0,\tilde\vb)}{\p\tilde\vb}\frac{\p\tilde\vb}{\p\eta}}\no\\
&\leq&\abs{\bv\tilde\a(\tilde\eta,-\d_0,\tilde\vb)},\no
\end{eqnarray}
since $\dfrac{\p\tilde\vb}{\p\eta}=0$. The estimate of $\abs{\bv\tilde\a(\tilde\eta,-\d_0,\tilde\vb)}$ is achieved since now $\abs{\tilde\va}\leq\d_0$.
\end{itemize}
Hence, we have
\begin{eqnarray}
\abs{\bvv III}&\leq&\frac{C}{\e^{s}}\Big(1+\abs{\ln(\e)}\Big)\sqrt{\d_0}\lnnmv{\a}+C\sqrt{\d}\lnnmv{\a}\\
&&+\frac{C}{\d^2}\lnnmv{\gg}+\frac{C}{\d_0}\lnnmv{\gg}.\no
\end{eqnarray}

\subsection{Estimates of Normal Derivative}

Combining the analysis in these three regions, we have for $0<s<<1$,
\begin{eqnarray}\label{pt 05}
\lnnmv{\a}
&\leq&\frac{C}{\e^{s}}\Big(1+\abs{\ln(\e)}\Big)\sqrt{\d_0}\lnnmv{\a}+C\sqrt{\d}\lnnmv{\a}\\
&&+\frac{C}{\d^2}\lnnmv{\gg}+\frac{C}{\d_0}\bigg(\lnmv{\frac{\p p}{\p\va}}+\lnmv{\frac{\p p}{\p\vb}}+\lnnmv{\gg}\bigg)\no\\
&&+\lnmv{p_{\a}}+\lnnmv{\frac{S_{\a}}{\nu}}.\no
\end{eqnarray}
Then we choose these constants to perform absorbing argument. First we choose $0<\d<<1$ sufficiently small such that
\begin{eqnarray}
C\sqrt{\d}\leq\frac{1}{4}.
\end{eqnarray}
Then we take $\d_0=\sqrt{\d}\e^s\abs{\ln(\e)}^{-1}$ such that
\begin{eqnarray}
\frac{C}{\e^{s}}(1+\abs{\ln(\e)})\sqrt{\d_0}\leq 2C\d\leq\frac{1}{2}.
\end{eqnarray}
for $\e$ sufficiently small. Note that this mild decay of $\d_0$ with respect to $\e$ also justifies the assumption in Case II and Case III that
\begin{eqnarray}
\e^{\frac{1}{4}}\leq \d_0,
\end{eqnarray}
for $\e$ sufficiently small. Here since $\d$ and $C$ are independent of $\e$, there is no circulant argument. Hence, we can absorb all the term related to $\lnnmv{\a}$ on the right-hand side of (\ref{pt 05}) to the left-hand side to obtain the desired result.
\begin{lemma}\label{pt lemma 3}
We have
\begin{eqnarray}
\lnnmv{\a}
&\leq&C\bigg(\lnmv{p_{\a}}+\lnnmv{\frac{S_{\a}}{\nu}}\bigg)\\
&&+C\abs{\ln(\e)}\e^{-s}\bigg(\lnmv{\frac{\p p}{\p\va}}+\lnmv{\frac{\p p}{\p\vb}}+\lnnmv{\gg}\bigg).\no
\end{eqnarray}
\end{lemma}

\subsection{Estimates Velocity Derivative}

Consider the general $\e$-Milne problem with geometric correction for $\b=\zeta\dfrac{\p\v}{\p\va}$ as
\begin{eqnarray}
\left\{
\begin{array}{l}\displaystyle
\va\frac{\p\b}{\p\eta}+G(\eta)\bigg(\vb^2\dfrac{\p\b }{\p\va}-\va\vb\dfrac{\p\b }{\p\vb}\bigg)+\nu\b=\tilde\b+S_{\b},\\\rule{0ex}{1.5em}
\b(0,\vvv)=p_{\b}(\vvv)\ \ \text{for}\ \ \va>0,\\\rule{0ex}{1.5em}
\b(L,\vvv)=\b(L,\rr[\vvv]),
\end{array}
\right.
\end{eqnarray}
where $p_{\b}$ and $S_{\b}$ will be specified later with
\begin{eqnarray}
\tilde\b(\eta,\vvv)=\int_{\r^2}\zeta(\eta,\vvv)\p_{\va}k(\vuu,\vvv)\gg(\eta,\vuu)\ud{\vuu}.
\end{eqnarray}
This is much simpler than normal derivative, since $\tilde\b$ does not contain $\b$ directly. Then by a similar argument as before, we obtain the desired result.
\begin{lemma}\label{pt lemma 4}
We have
\begin{eqnarray}
\lnnmv{\b}
&\leq&C\bigg(\lnmv{p_{\b}}+\lnnmv{\frac{S_{\b}}{\nu}}\bigg)\\
&&+C\bigg(\lnmv{\frac{\p p}{\p\va}}+\lnmv{\frac{\p p}{\p\vb}}+\lnnmv{\gg}\bigg).\no
\end{eqnarray}
\end{lemma}
In a similar fashion, consider the general $\e$-Milne problem with geometric correction for $\c=\zeta\dfrac{\p\v}{\p\vb}$ as
\begin{eqnarray}
\left\{
\begin{array}{l}\displaystyle
\va\frac{\p\c}{\p\eta}+G(\eta)\bigg(\vb^2\dfrac{\p\c }{\p\va}-\va\vb\dfrac{\p\c }{\p\vb}\bigg)+\nu\c=\tilde\c+S_{\c},\\\rule{0ex}{1.5em}
\c(0,\vvv)=p_{\c}(\vvv)\ \ \text{for}\ \ \va>0,\\\rule{0ex}{1.5em}
\c(L,\vvv)=\c(L,\rr[\vvv]),
\end{array}
\right.
\end{eqnarray}
where $p_{\c}$ and $S_{\c}$ will be specified later with
\begin{eqnarray}
\tilde\c(\eta,\vvv)=\int_{\r^2}\zeta(\eta,\vvv)\p_{\vb}k(\vuu,\vvv)\gg(\eta,\vuu)\ud{\vuu}.
\end{eqnarray}
This is also much simpler than normal derivative, since $\tilde\c$ does not contain $\c$ directly. Then by a similar argument as before, we obtain the desired result.
\begin{lemma}\label{pt lemma 5}
We have
\begin{eqnarray}
\lnnmv{\c}
&\leq&C\bigg(\lnmv{p_{\c}}+\lnnmv{\frac{S_{\c}}{\nu}}\bigg)\\
&&+C\bigg(\lnmv{\frac{\p p}{\p\va}}+\lnmv{\frac{\p p}{\p\vb}}+\lnnmv{\gg}\bigg).\no
\end{eqnarray}
\end{lemma}

\subsection{Estimates of Tangential Derivative}

In this subsection, we combine above a priori estimates of normal and velocity derivatives.
\begin{theorem}\label{pt theorem 1}
We have
\begin{eqnarray}
\lnnmv{\zeta\frac{\p\v}{\p\eta}}+\lnnmv{\zeta\frac{\p\v}{\p\va}}+\lnnmv{\zeta\frac{\p\v}{\p\vb}}\leq C\abs{\ln(\e)}\e^{-s},
\end{eqnarray}
for some $0<s<<1$.
\end{theorem}
\begin{proof}
Collecting the estimates for $\a$, $\b$, and $\c$ in Lemma \ref{pt lemma 3}, Lemma \ref{pt lemma 4}, and Lemma \ref{pt lemma 5}, we have
\begin{eqnarray}
\\
\lnnmv{\a}
&\leq&C\bigg(\lnmv{p_{\a}}+\lnnmv{\frac{S_{\a}}{\nu}}\bigg)+C_0\abs{\ln(\e)}\e^{-s},\no\\
\lnnmv{\b}
&\leq&C\bigg(\lnmv{p_{\b}}+\lnnmv{\frac{S_{\b}}{\nu}}\bigg)+C_0,\\
\lnnmv{\c}
&\leq&C\bigg(\lnmv{p_{\c}}+\lnnmv{\frac{S_{\c}}{\nu}}\bigg)+C_0,
\end{eqnarray}
where
\begin{eqnarray}
C_0&=&\lnmv{p}+\lnmv{\frac{\p p}{\p\va}}+\lnmv{\frac{\p p}{\p\vb}}+\lnnmv{\gg}.
\end{eqnarray}
Taking derivatives on both sides of (\ref{Milne difference problem}) and multiplying $\zeta$, we have
\begin{eqnarray}
p_{\a}&=&-\frac{\e}{\rk}\bigg(\vb^2\frac{\p p}{\p\va}-\va\vb\frac{\p p}{\p\vb}\bigg)+\nu p-K[\gg](0,\vvv),\label{dt 04}\\
p_{\b}&=&\va\frac{\p p}{\p\va},\label{dt 05}\\
p_{\c}&=&\va\frac{\p p}{\p\vb},\label{dt 06}\\
S_{\a}&=&\frac{\p{G}}{\p{\eta}}\bigg(\vb^2\b-\va\vb\c\bigg),\label{dt 01}\\
S_{\b}&=&\a-G\vb\c,\label{dt 02}\\
S_{\c}&=&G\bigg(2\vb\b-\va\c\bigg).\label{dt 03}
\end{eqnarray}
We can directly verify that
\begin{eqnarray}
\lnmv{p_{\a}}+\lnmv{p_{\b}}+\lnmv{p_{\c}}\leq C_0.
\end{eqnarray}
Since $\abs{G(\eta)}+\abs{\dfrac{\p G}{\p\eta}}\leq \e$, from (\ref{dt 03}), we obtain
\begin{eqnarray}
\lnnmv{\c}&\leq&C_0+C_0\e\bigg(\lnnmv{\frac{\vb\b}{\nu}}+\lnnmv{\frac{\va\c}{\nu}}\bigg)\\
&\leq& C_0+C_0\e\bigg(\lnnmv{\b}+\lnnmv{\c}\bigg),\no
\end{eqnarray}
which further implies
\begin{eqnarray}\label{dt 07}
\lnnmv{\c}&\leq& C_0+C_0\e\lnnmv{\b}.
\end{eqnarray}
Plugging (\ref{dt 07}) into (\ref{dt 02}), we obtain
\begin{eqnarray}
\lnnmv{\b}&\leq&C_0+C_0\bigg(\lnnmv{\frac{\a}{\nu}}+\e\lnnmv{\frac{\vb\c}{\nu}}\bigg)\\
&\leq& C_0+C_0\bigg(\lnnmv{\frac{\a}{\nu}}+\e\lnnmv{\c}\bigg)\no\\
&\leq&C_0+C_0\bigg(\lnnmv{\frac{\a}{\nu}}+\e^2\lnnmv{\b}\bigg),\no
\end{eqnarray}
which further implies
\begin{eqnarray}\label{dt 08}
\lnnmv{\b}&\leq&C_0+C_0\lnnmv{\frac{\a}{\nu}},\\
\lnnmv{\c}&\leq&C_0+C_0\e\lnnmv{\frac{\a}{\nu}}.\label{dt 09}
\end{eqnarray}
Plugging (\ref{dt 08}) and (\ref{dt 09}) into (\ref{dt 01}), we get
\begin{eqnarray}
\lnnmv{\a}&\leq&C_0\abs{\ln(\e)}\e^{-s}+C_0\e\bigg(\lnnmv{\frac{\vb^2\b}{\nu}}+\lnnmv{\frac{\va\vb\c}{\nu}}\bigg)\\
&\leq&C_0\abs{\ln(\e)}\e^{-s}+C_0\e\bigg(\lnnmv{\frac{\vb^2\a}{\nu^2}}+\e\lnnmv{\frac{\va\vb\a}{\nu^2}}\bigg)\no\\
&\leq&C_0\abs{\ln(\e)}\e^{-s}+C_0\e\lnnmv{\a},\no
\end{eqnarray}
which implies
\begin{eqnarray}
\lnnmv{\a}&\leq&C_0\abs{\ln(\e)}\e^{-s}.
\end{eqnarray}
Hence, we derive
\begin{eqnarray}
\a&\leq&C\abs{\ln(\e)}\e^{-s},\\
\b&\leq&C\abs{\ln(\e)}\e^{-s},\\
\c&\leq&C\abs{\ln(\e)}\e^{-s},
\end{eqnarray}
\end{proof}
Above theorems only provide a priori estimates. The rigorous proof relies on a penalty method and an iteration argument. This step is standard as in \cite{AA007}, so we omit it here.
\begin{theorem}\label{pt theorem 2}
For $K_0>0$ sufficiently small, we have
\begin{eqnarray}
\lnnmv{\ue^{K_0\eta}\zeta\frac{\p\v}{\p\eta}}+\lnnmv{\ue^{K_0\eta}\zeta\frac{\p\v}{\p\va}}+\lnnmv{\ue^{K_0\eta}\zeta\frac{\p\v}{\p\vb}}\leq C\abs{\ln(\e)}\e^{-s},
\end{eqnarray}
for some $0<s<<1$.
\end{theorem}
\begin{proof}
This proof is almost identical to Theorem \ref{pt theorem 1}. The only difference is that $S_{\a}$ is added by $K_0\va\a$, $S_{\b}$ added by $K_0\va\b$, and $S_{\c}$ added by $K_0\va\c$. When $K_0$ is sufficiently small, we can also absorb them into the left-hand side. Hence, this is obvious.
\end{proof}
Now we pull $\theta$ dependence back and study the tangential derivative.
\begin{theorem}\label{Milne tangential}
We have
\begin{eqnarray}
\lnnmv{\ue^{K_0\eta}\frac{\p\v}{\p\theta}(\eta,\theta,\phi)}\leq C\abs{\ln(\e)}\e^{-s},
\end{eqnarray}
for some $0<s<<1$.
\end{theorem}
\begin{proof}
Let $\w=\dfrac{\p\gg}{\p\theta}$. Taking $\theta$ derivative on both sides of (\ref{Milne difference problem}), we have that $\w$ satisfies the equation
\begin{eqnarray}\label{Milne tangential problem}
\left\{
\begin{array}{l}\displaystyle
\va\frac{\p \w}{\p\eta}+G(\eta)\bigg(\vb^2\dfrac{\p
\w}{\p\va}-\va\vb\dfrac{\p
\w}{\p\vb}\bigg)+\nu\w-K[\w]=\dfrac{\rk'}{\rk-\e\eta}G(\eta)\bigg(\vb^2\dfrac{\p
\gg}{\p\va}-\va\vb\dfrac{\p
\gg}{\p\vb}\bigg),\\\rule{0ex}{2.0em}
\w(0,\theta,\vvv)=\dfrac{\p p}{\p\theta}(\theta,\vvv)\ \ \text{for}\ \ \sin\phi>0,\\\rule{0ex}{2.0em}
\w(L,\theta,\vvv)=\w(L,\theta,\rr[\vvv]),
\end{array}
\right.
\end{eqnarray}
where $\rk'$ is the $\theta$ derivative of $\rk$.
For $\eta\in[0,L]$, we have
\begin{eqnarray}
\dfrac{\rk'}{\rk-\e\eta}\leq C\max_{\theta}\rk'\leq C.
\end{eqnarray}
Since $\zeta(\eta,\vvv)\geq\va$, based on Theorem \ref{pt theorem 2} and the equation (\ref{Milne difference problem}), we know
\begin{eqnarray}
\lnnmv{\ue^{K_0\eta}\va\frac{\p
\w}{\p\eta}}\leq C\abs{\ln(\e)}\e^{-s},
\end{eqnarray}
which further implies
\begin{eqnarray}
\lnnmv{\ue^{K_0\eta}G(\eta)\bigg(\vb^2\frac{\p
\w}{\p\va}-\va\vb\frac{\p
\w}{\p\vb}\bigg)}\leq C\abs{\ln(\e)}\e^{-s},
\end{eqnarray}
for some $0<s<<1$.
Therefore, the source term in the equation (\ref{Milne tangential problem}) is in $L^{\infty}$ and decays exponentially. By Theorem \ref{Milne theorem 3}, we have that
\begin{eqnarray}
\lnnmv{\ue^{K_0\eta}\w(\eta,\theta,\phi)}\leq C\abs{\ln(\e)}\e^{-s},
\end{eqnarray}
for some $0<s<<1$.
\end{proof}

\newpage

\section{Hydrodynamic Limits}

\begin{theorem}
For given $M_0>0$ and $\mb>0$ satisfying (\ref{expansion assumption}) and (\ref{smallness assumption}) with $0<\e<<1$, there exists a unique positive
solution $\fs^{\e}=M_0\m+\m^{\frac{1}{2}}f^{\e}$ to the stationary Boltzmann equation (\ref{large system}), and $f^{\e}$ fulfils that for integer $\vth\geq3$ and $0\leq\varrho<\dfrac{1}{4}$,
\begin{eqnarray}
\im{\bv\Big(f^{\e}-\e\f\Big)}\leq C(\d)\e^{2-\d},
\end{eqnarray}
for any $0<\d<<1$, where
\begin{eqnarray}
\f&=&\m^{\frac{1}{2}}\left(\rh+\vu\cdot\vw+\th\frac{\abs{\vw}^2-2}{2}\right),
\end{eqnarray}
satisfies the steady Navier-Stokes-Fourier system
\begin{eqnarray}\label{interior 1}
\left\{
\begin{array}{rcl}
\nx(\rh +\th )&=&0,\\\rule{0ex}{1.0em}
\uh\cdot\nx\uh -\gamma_1\dx\uh +\nx P_2 &=&0,\\\rule{0ex}{1.0em}
\nx\cdot\uh &=&0,\\\rule{0ex}{1.0em}
\uh \cdot\nx\th -\gamma_2\dx\th &=&0,\\\rule{0ex}{1.0em}
\rh (\vx_0)&=&\rh_{\bb,1}(\vx_0)+M(\vx_0),\\
\uh (\vx_0)&=&\vu_{\bb,1}(\vx_0),\\
\th (\vx_0)&=&\th_{\bb,1}(\vx_0),\\
\end{array}
\right.
\end{eqnarray}
where $\gamma_1>0$ and $\gamma_2>0$ are some constants, $M(\vx_0)$ is a constant such that the Boussinesq relation
\begin{eqnarray}
\rh+\th=\text{constant},
\end{eqnarray}
and the normalization condition
\begin{eqnarray}
\int_{\Omega}\int_{\r^2}\f(\vx,\vw)\m^{\frac{1}{2}}(\vw)\ud{\vw}\ud{\vx}=0,
\end{eqnarray}
hold.
\end{theorem}
\begin{proof}
The asymptotic analysis already reveals that the construction of the interior solution and boundary layer is valid. Here, we focus on the remainder estimates. We divide the proof into several steps:\\
\ \\
Step 1: Remainder definitions.\\
Define the remainder as
\begin{eqnarray}\label{pf 1}
R&=&\frac{1}{\e^3}\bigg(f^{\e}-\Big(\e\f_1+\e^2\f_2+\e^3\f_3\Big)-\Big(\e\fb_1+\e^2\fb_2\Big)\bigg)=\frac{1}{\e^3}\bigg(f^{\e}-\q -\qb \bigg),
\end{eqnarray}
where
\begin{eqnarray}
\q &=&\e\f_1+\e^2\f_2+\e^3\f_3,\\
\qb &=&\e\fb_1+\e^2\fb_2.
\end{eqnarray}
In other words, we have
\begin{eqnarray}
f^{\e}=\q+\qb+\e^3R.
\end{eqnarray}
We write $\lll$ to denote the
linearized Boltzmann operator as follows:
\begin{eqnarray}
\lll[f]&=&\e\vv\cdot\nx u+\ll[f]\\
&=&\va\dfrac{\p f}{\p\eta}-\dfrac{\e}{\rk-\e\eta}\bigg(\vb^2\dfrac{\p f}{\p\va}-\va\vb\dfrac{\p
f}{\p\vb}\bigg)-\frac{\e}{\rk-\e\eta}\dfrac{\rk}{(r^2+r'^2)^{\frac{1}{2}}}\vb\dfrac{\p
f}{\p\theta}+\ll[f].\nonumber
\end{eqnarray}
\ \\
Step 2: Representation of $\lll[R]$.\\
The equation (\ref{small system_}) is actually
\begin{eqnarray}
\lll[f^{\e}]=\Gamma[f^{\e},f^{\e}],
\end{eqnarray}
which means
\begin{eqnarray}
\lll[\q+\qb+\e^3R]=\Gamma[\q+\qb+\e^3R,\q+\qb+\e^3R].
\end{eqnarray}
Note that the nonlinear term can be decomposed as
\begin{eqnarray}
\Gamma[\q+\qb+\e^3R,\q+\qb+\e^3R]&=&\e^6\Gamma[R,R]+2\e^3\Gamma[R,\q+\qb]+\Gamma[\q+\qb,\q+\qb].
\end{eqnarray}
The interior contribution can be represented as
\begin{eqnarray}
\lll[\q ]&=&\e\vv\cdot\nx\Big(\e\f_1+\e^2\f_2+\e^3\f_3\Big) +\ll[\e\f_1+\e^2\f_2+\e^3\f_3]\\
&=&\e^4\vv\cdot\nx\f_3+\e^2\Gamma[\f_1,\f_1]+2\e^3\Gamma[\f_1,\f_1].\no
\end{eqnarray}
The nonlinear term will be handled by $\Gamma[\q+\qb,\q+\qb]$. On the other hand, we consider the boundary layer contribution. Since $\fb_1=0$, we may directly compute
\begin{eqnarray}\label{remainder temp 1}
\lll[\qb ]&=&\e^2\bigg(\va\dfrac{\p\fb_2}{\p\eta}-\dfrac{\e}{\rk-\e\eta}\bigg(\vb^2\dfrac{\p\fb_2}{\p\va}-\va\vb\dfrac{\p
\fb_2}{\p\vb}\bigg)-\frac{\e}{\rk-\e\eta}\dfrac{\rk}{(r^2+r'^2)^{\frac{1}{2}}}\vb\dfrac{\p
\fb_2}{\p\theta}+\ll[\fb_2]\bigg)\\
&=&-\frac{\e^3}{\rk-\e\eta}\dfrac{\rk}{(r^2+r'^2)^{\frac{1}{2}}}\vb\dfrac{\p
\fb_2}{\p\theta}.\no
\end{eqnarray}
Therefore, we have
\begin{eqnarray}
\lll[R]&=&\e^3\Gamma[R,R]+2\Gamma[R,\q+\qb]+S_1+S_2,
\end{eqnarray}
where
\begin{eqnarray}
S_1&=&-\e\vv\cdot\nx\f_3+\frac{1}{\rk-\e\eta}\dfrac{\rk}{(r^2+r'^2)^{\frac{1}{2}}}\vb\dfrac{\p
\fb_2}{\p\theta},\\
S_2&=&2\Gamma[\f_1,\fb_2]+2\e\Gamma[\f_1,\f_3]+\e\Gamma[\fb_2,\fb_2]+2\e\Gamma[\f_2,\fb_1]+2\e^2\Gamma[\f_2,\f_3]+\e^3\Gamma[\f_3,\f_3]
\end{eqnarray}
\ \\
Step 3: Representation of $R-\pp[R]$.\\
Since
\begin{eqnarray}
f^{\e}(\vx_0,\vw)&=&\mb(\vx_0,\vv)\m^{-\frac{1}{2}}(\vv)
\displaystyle\int_{\vuu\cdot\vn(\vx_0)>0}\m^{\frac{1}{2}}(\vuu)
f^{\e}(\vx_0,\vuu)\abs{\vuu\cdot\vn(\vx_0)}\ud{\vuu}+\m^{-\frac{1}{2}}(\vw)\bigg(\mb(\vx_0,\vv)-\m(\vv)\bigg),\no
\end{eqnarray}
where both sides are linear, we may directly write
\begin{eqnarray}
R(\vx_0,\vw)-\pp[R](\vx_0)
&=&H[R](\vx_0,\vv)+h(\vx_0,\vv),
\end{eqnarray}
where
\begin{eqnarray}
H[R](\vx_0,\vv)=\Big(\mb(\vx_0,\vv)-\m(\vv)\Big)\m^{-\frac{1}{2}}(\vv)
\displaystyle\int_{\vuu\cdot\vn(\vx_0)>0}\m^{\frac{1}{2}}(\vuu)
R(\vx_0,\vuu)\abs{\vuu\cdot\vn(\vx_0)}\ud{\vuu},
\end{eqnarray}
and
\begin{eqnarray}
&&h(\vx_0,\vw)\\
&=&\Big(\mb(\vx_0,\vv)-\m(\vv)\Big)\m^{-\frac{1}{2}}(\vv)
\displaystyle\int_{\vuu\cdot\vn(\vx_0)>0}\m^{\frac{1}{2}}(\vuu)
\f_3(\vx_0,\vuu)\abs{\vuu\cdot\vn(\vx_0)}\ud{\vuu}\no\\
&&+\Big(\mb(\vx_0,\vv)-\m(\vv)-\e\m_1(\vx_0,\vv)\Big)\m^{-\frac{1}{2}}(\vv)
\displaystyle\int_{\vuu\cdot\vn(\vx_0)>0}\m^{\frac{1}{2}}(\vuu)
(\e^{-1}\f_2+\e^{-1}\fb_2+\f_3)(\vx_0,\vuu)\abs{\vuu\cdot\vn(\vx_0)}\ud{\vuu}\no\\
&&+\Big(\mb(\vx_0,\vv)-\m(\vv)-\e\m_1(\vx_0,\vv)-\e^2\m_2(\vx_0,\vv)\Big)\m^{-\frac{1}{2}}(\vv)\cdot\no\\
&&\displaystyle\int_{\vuu\cdot\vn(\vx_0)>0}\m^{\frac{1}{2}}(\vuu)
\e^{-3}(\q+\qb)(\vx_0,\vuu)\abs{\vuu\cdot\vn(\vx_0)}\ud{\vuu}\no\\
&&+\e^{-3}\m^{-\frac{1}{2}}(\vw)\bigg(\mb(\vx_0,\vv)-\m(\vv)-\e\m_1(\vx_0,\vv)-\e^2\m_2(\vx_0,\vv)-\e^3\m_3(\vx_0,\vv)\bigg).\no
\end{eqnarray}
\ \\
Step 4: $L^{2m}$ Estimates of $R$.\\
Using Theorem \ref{LN estimate}, we have the $L^{2m}$ estimate of $R$
\begin{eqnarray}\label{ht 06}
&&\frac{1}{\e}\um{(\ik-\pk)[R]}+\frac{1}{\e^{\frac{1}{2}}}\tss{(1-\pp)[R]}{+}+\nm{\pk[R]}_{L^{2m}}\\
&\leq&C\bigg(o(1)\e^{\frac{1}{m}}\nm{R}_{L^{\infty}}+\frac{1}{\e^{2}}\nm{\pk\Big[\lll[R]\Big]}_{L^{\frac{2m}{2m-1}}}+\frac{1}{\e}\tm{\lll[R]}
+\abs{R-\pp[R]}_{L^{m}_-}+\frac{1}{\e}\tss{R-\pp[R]}{-}\bigg)\no\\
&\leq& C\bigg(o(1)\e^{\frac{1}{m}}\nm{R}_{L^{\infty}}+\frac{1}{\e^{2}}\nm{\pk[S_1]}_{L^{\frac{2m}{2m-1}}}\no\\
&&+\frac{1}{\e}\bigg(\tm{\e^3\Gamma[R,R]}+\tm{2\Gamma[R,\q+\qb]}+\tm{S_1}+\tm{S_2}\bigg)\no\\
&&+\abs{H[R]}_{L^{m}_-}+\abs{h}_{L^{m}_-}\no\\
&&+\frac{1}{\e}\tss{H[R]}{-}+\frac{1}{\e}\tss{h}{-}\bigg).\no
\end{eqnarray}
Note that here we do not have other source terms in the $L^{\frac{2m}{2m-1}}$ norm because
for any $f,g\in L^2$,
\begin{eqnarray}\label{ft 01}
\pk[\Gamma(f,g)]=0.
\end{eqnarray}
We need to estimate each term. It is easy to check
\begin{eqnarray}\label{pf 2_}
\tm{\e\vv\cdot\nx\f_3}&\leq& C\e,\\
\nm{\e\vv\cdot\nx\f_3}_{L^{\frac{2m}{2m-1}}}&\leq& C\e,
\end{eqnarray}
and also by Theorem \ref{Milne tangential}, using the rescaling and exponential decay, we have
\begin{eqnarray}
\tm{\frac{1}{\rk-\e\eta}\dfrac{\rk}{(r^2+r'^2)^{\frac{1}{2}}}\vb\dfrac{\p
\fb_2}{\p\theta}}&\leq&
C\Bigg(\int_{-\pi}^{\pi}\int_0^{R_{\min}\e^{\frac{1}{2}}}(\rk-\mathfrak{N})\lnmv{\frac{\p\fb_2}{\p\theta}(\mathfrak{N},\theta)}^2\ud{\mathfrak{N}}\ud{\theta}\Bigg)^{1/2}\\
&\leq&C\e^{\frac{1}{2}}\Bigg(\int_{-\pi}^{\pi}\int_0^{R_{\min}\e^{-\frac{1}{2}}}\lnmv{\frac{\p\fb_2}{\p\theta}(\eta,\theta)}^2\ud{\eta}\ud{\theta}\Bigg)^{1/2}\no\\
&\leq&C\e^{\frac{1}{2}}\Bigg(\int_{-\pi}^{\pi}\int_0^{R_{\min}\e^{-\frac{1}{2}}}\ue^{-2K_0\eta}\abs{\ln(\e)}^{2}\e^{-2s}\ud{\eta}\ud{\theta}\Bigg)^{1/2}\no\\
&\leq& C\e^{\frac{1}{2}-s}\abs{\ln(\e)},\no
\end{eqnarray}
for some $0<s<<1$. Similarly, we can prove that
\begin{eqnarray}
\nm{\frac{1}{\rk-\e\eta}\dfrac{\rk}{(r^2+r'^2)^{\frac{1}{2}}}\vb\dfrac{\p
\fb_2}{\p\theta}}_{L^{\frac{2m}{2m-1}}}&\leq&C\e^{1-\frac{1}{2m}-s}\abs{\ln(\e)}.
\end{eqnarray}
In total, we have
\begin{eqnarray}
\tm{S_1}&\leq&C\e^{\frac{1}{2}-s}\abs{\ln(\e)},\\
\nm{\pk[S_1]}_{L^{\frac{2m}{2m-1}}}&\leq&C\e^{1-\frac{1}{2m}-s}\abs{\ln(\e)}.
\end{eqnarray}
On the other hand, by Lemma \ref{nonlinear}, we know
\begin{eqnarray}
\tm{2\Gamma[R,\q+\qb]}&\leq&C\bigg(\tm{R}\lnmv{\q+\qb}+\um{R}\im{\q+\qb}\bigg).
\end{eqnarray}
Based on the smallness assumption (\ref{smallness assumption}) on the boundary Maxwellian, it is easy to check that
\begin{eqnarray}
\tm{R}\lnmv{\q+\qb}\leq o(1)\e\tm{R}.
\end{eqnarray}
Also, we may decompose
\begin{eqnarray}
\um{R}\im{\q+\qb}&\leq&\um{(\ik-\pk)[R]}\im{\q+\qb}+\um{\pk[R]}\im{\q+\qb}\\
&\leq&o(1)\e\um{(\ik-\pk)[R]}+o(1)\e\nm{\pk[R]}_{L^{2}}.\no
\end{eqnarray}
Then we may derive that
\begin{eqnarray}
\tm{2\Gamma[R,\q+\qb]}&\leq&o(1)\e\bigg(\nm{\pk[R]}_{L^{2}}+\um{(\ik-\pk)[R]}\bigg).
\end{eqnarray}
Also, using the smallness assumption (\ref{smallness assumption}) again, we can directly estimate
\begin{eqnarray}
\abs{H[R]}_{L^{m}_-}&\leq&o(1)\e\abs{R}_{L^{m}_-}\leq o(1)\e\im{R},\\
\tss{H[R]}{-}&\leq& o(1)\e\tss{\pp[R]}{+}.
\end{eqnarray}
Using Lemma \ref{nonlinear}, it is easy to check
\begin{eqnarray}
\tm{S_2}&\leq&C\e^{\frac{1}{2}-s}\abs{\ln(\e)},\\
\abs{h}_{L^{m}_-}&\leq&C\e,\\
\tss{h}{-}&\leq&C\e.
\end{eqnarray}
Summarizing all of them, we have proved that
\begin{eqnarray}
&&\frac{1}{\e}\um{(\ik-\pk)[R]}+\frac{1}{\e^{\frac{1}{2}}}\tss{(1-\pp)[R]}{+}+\nm{\pk[R]}_{L^{2m}}\\
&\leq& C\bigg(o(1)\e^{\frac{1}{m}}\nm{R}_{L^{\infty}}+\e^{-1-\frac{1}{2m}-s}\abs{\ln(\e)}\no\\
&&+\e^2\tm{\Gamma[R,R]}+o(1)\nm{\pk[R]}_{L^{2}}+o(1)\um{(\ik-\pk)[R]}+\e^{-\frac{1}{2}-s}\abs{\ln(\e)}+\e^{-\frac{1}{2}-s}\abs{\ln(\e)}\no\\
&&+o(1)\e\im{R}+\e\no\\
&&+o(1)\tss{\pp[R]}{+}+1\bigg)\no
\end{eqnarray}
Absorbing $\um{(\ik-\pk)[R]}$ into the left-hand side, we obtain
\begin{eqnarray}\label{ht 02}
&&\frac{1}{\e}\um{(\ik-\pk)[R]}+\frac{1}{\e^{\frac{1}{2}}}\tss{(1-\pp)[R]}{+}+\nm{\pk[R]}_{L^{2m}}\\
&\leq&C\bigg(o(1)\e^{\frac{1}{m}}\nm{R}_{L^{\infty}}+\e^{-1-\frac{1}{2m}-s}\abs{\ln(\e)}+\e^2\tm{\Gamma[R,R]}+o(1)\nm{\pk[R]}_{L^{2}}+o(1)\tss{\pp[R]}{+}\bigg).\no
\end{eqnarray}
Here, we apply the estimate (\ref{wt 32}) to bound
\begin{eqnarray}\label{ht 03}
\tss{\pp[R]}{+}&\leq&C\left(\tm{\pk[R]}+\frac{1}{\e^{\frac{1}{2}}}\tm{(\ik-\pk)[R]}+\frac{1}{\e^{\frac{1}{2}}}\left(\int_{\Omega\times\r^2}R\lll[R]\right)^{\frac{1}{2}}\right)\\
&\leq&C\left(\tm{\pk[R]}+\frac{1}{\e^{\frac{1}{2}}}\tm{(\ik-\pk)[R]}+\frac{1}{\e^{\frac{1}{2}}}\tm{(\ik-\pk)\Big[\lll[R]\Big]}+\frac{1}{\e}\tm{\pk\Big[\lll[R]\Big]}\right).\no
\end{eqnarray}
Then using Lemma \ref{wellposedness prelim lemma 3}, we obtain
\begin{eqnarray}
\tm{\pk[R]}\leq C\left(\tss{(1-\pp)[R]}{+}+\frac{1}{\e}\tm{(\ik-\pk)[R]}+\frac{1}{\e}\tm{\lll[R]}+\tss{R-\pp[R]}{-}\right).
\end{eqnarray}
Therefore, we have
\begin{eqnarray}\label{ht 05}
\tss{\pp[R]}{+}&\leq&C\left(\tm{\pk[R]}+\frac{1}{\e^{\frac{1}{2}}}\tm{(\ik-\pk)[R]}+\frac{1}{\e^{\frac{1}{2}}}\left(\int_{\Omega\times\r^2}R\lll[R]\right)^{\frac{1}{2}}\right)\\
&\leq&C\bigg(\tss{(1-\pp)[R]}{+}+\frac{1}{\e}\tm{(\ik-\pk)[R]}\no\\
&&+\frac{1}{\e}\tm{(\ik-\pk)\Big[\lll[R]\Big]}+\frac{1}{\e}\tm{\pk\Big[\lll[R]\Big]}+\tss{R-\pp[R]}{-}\bigg).\no
\end{eqnarray}
Note that the right-hand side of (\ref{ht 05}) has been estimated in (\ref{ht 06}) and (\ref{ht 02}), so we obtain
\begin{eqnarray}
\\
\tss{\pp[R]}{+}&\leq&C\bigg(o(1)\e^{\frac{1}{m}}\nm{R}_{L^{\infty}}+\e^{-1-\frac{1}{2m}-s}\abs{\ln(\e)}+\e^2\tm{\Gamma[R,R]}+o(1)\nm{\pk[R]}_{L^{2}}+o(1)\tss{\pp[R]}{+}\bigg).\no
\end{eqnarray}
Absorbing $\tss{\pp[R]}{+}$ into the left-hand side, we obtain
\begin{eqnarray}\label{ht 07}
\tss{\pp[R]}{+}&\leq&C\bigg(o(1)\e^{\frac{1}{m}}\nm{R}_{L^{\infty}}+\e^{-1-\frac{1}{2m}-s}\abs{\ln(\e)}+\e^2\tm{\Gamma[R,R]}+o(1)\nm{\pk[R]}_{L^{2}}\bigg).
\end{eqnarray}
Plugging (\ref{ht 07}) into (\ref{ht 02}), we have
\begin{eqnarray}
&&\frac{1}{\e}\um{(\ik-\pk)[R]}+\frac{1}{\e^{\frac{1}{2}}}\tss{(1-\pp)[R]}{+}+\nm{\pk[R]}_{L^{2m}}\\
&\leq&C\bigg(o(1)\e^{\frac{1}{m}}\nm{R}_{L^{\infty}}+\e^{-1-\frac{1}{2m}-s}\abs{\ln(\e)}+\e^2\tm{\Gamma[R,R]}+o(1)\nm{\pk[R]}_{L^{2}}\bigg).\no
\end{eqnarray}
Note that the estimate of $\nm{\pk[R]}_{L^{2}}$ has been incorporated in above analysis for $\tss{\pp[R]}{+}$, so we may further simplify
\begin{eqnarray}\label{ht 04}
&&\frac{1}{\e}\um{(\ik-\pk)[R]}+\frac{1}{\e^{\frac{1}{2}}}\tss{(1-\pp)[R]}{+}+\nm{\pk[R]}_{L^{2m}}\\
&\leq&C\bigg(o(1)\e^{\frac{1}{m}}\nm{R}_{L^{\infty}}+\e^{-1-\frac{1}{2m}-s}\abs{\ln(\e)}+\e^2\tm{\Gamma[R,R]}\bigg).\no
\end{eqnarray}
\ \\
Step 5: $L^{\infty}$ Estimates of $R$.\\
Based on Theorem \ref{LI estimate}, we have
\begin{eqnarray}
&&\im{\bv R}+\iss{\bv R}{+}\\
&\leq&C\bigg(\frac{1}{\e^{\frac{1}{m}}}\nm{\pk[R]}_{L^{2m}}+\frac{1}{\e}\nm{(\ik-\pk)[R]}_{L^{2}}\no\\
&&+\im{\frac{\bv \e^3\Gamma[R,R]}{\nu}}+\im{\frac{\bv \Gamma[R,\q+\qb]}{\nu}}\no\\
&&+\im{\frac{\bv S_1}{\nu}}+\im{\frac{\bv S_2}{\nu}}+\iss{\bv H[R]}{-}+\iss{\bv h}{-}\bigg).\no
\end{eqnarray}
Hence, using (\ref{ht 04}), we know
\begin{eqnarray}
&&\im{\bv R}+\iss{\bv R}{+}\\
&\leq&C\bigg(o(1)\nm{R}_{L^{\infty}}+\e^{-1-\frac{1}{2m}-4s}\abs{\ln(\e)}^8+\e^2\tm{\Gamma[R,R]}\no\\
&&+\im{\frac{\bv \e^3\Gamma[R,R]}{\nu}}+\im{\frac{\bv \Gamma[R,\q+\qb]}{\nu}}\no\\
&&+\im{\frac{\bv S_1}{\nu}}+\im{\frac{\bv S_2}{\nu}}+\iss{\bv H[R]}{-}+\iss{\bv h}{-}\bigg).\no
\end{eqnarray}
We can directly estimate
\begin{eqnarray}
\tm{\e^2\Gamma[R,R]}&\leq& C\e^2\im{\bv R}^2,\\
\e^3\im{\frac{\bv \Gamma[R,R]}{\nu}}&\leq&C\e^3\im{\bv R}^2,\\
\im{\frac{\bv \Gamma[R,\q+\qb]}{\nu}}&\leq&C\e\im{\bv R}.
\end{eqnarray}
Also, we know
\begin{eqnarray}
\im{\frac{\bv S_1}{\nu}}&\leq&C,\\
\im{\frac{\bv S_2}{\nu}}&\leq&C,\\
\iss{\bv H[R]}{-}&\leq&\e\iss{\bv R}{+},\\
\iss{\bv h}{-}&\leq&\e.
\end{eqnarray}
Hence, in total, we have
\begin{eqnarray}
&&\im{\bv R}+\iss{\bv R}{+}\\
&\leq&C\bigg(o(1)\im{\bv R}+o(1)\iss{\bv R}{+}+\e^{-1-\frac{1}{2m}-s}\abs{\ln(\e)}+\e^2\im{\bv R}^2\bigg).\no
\end{eqnarray}
Absorbing $o(1)\im{\bv R}$ and $o(1)\iss{\bv R}{+}$ into the left-hand side, we obtain
\begin{eqnarray}
&&\im{\bv R}+\iss{\bv R}{+}\leq C\bigg(\e^{-1-\frac{1}{2m}-s}\abs{\ln(\e)}+\e^2\im{\bv R}^2\bigg),
\end{eqnarray}
which further implies
\begin{eqnarray}
\im{\bv R}+\iss{\bv R}{+}&\leq&C\e^{-1-\frac{1}{2m}-s}\abs{\ln(\e)},
\end{eqnarray}
for $\e$ sufficiently small. This means we have shown
\begin{eqnarray}
\frac{1}{\e^3}\im{\bv\bigg(f^{\e}-\Big(\e\f_1+\e^2\f_2+\e^3\f_3\Big)-\Big(\e\fb_1+\e^2\fb_2\Big)\bigg)}&\leq&C\e^{-1-\frac{1}{2m}-s}\abs{\ln(\e)}.
\end{eqnarray}
Therefore, we know
\begin{eqnarray}
\im{\bv\bigg(f^{\e}-\e\f_1-\e\fb_1\bigg)}\leq C\e^{2-\frac{1}{2m}-s}\abs{\ln(\e)}.
\end{eqnarray}
Since $\fb_1=0$, then we naturally have for $\f=\f_1$.
\begin{eqnarray}
\im{\bv\bigg(f^{\e}-\e\f\bigg)}\leq C\e^{2-\frac{1}{2m}-s}\abs{\ln(\e)}.
\end{eqnarray}
Here $0<s<<1$, so we may further bound
\begin{eqnarray}
\im{\bv\bigg(f^{\e}-\e\f\bigg)}\leq C(\d)\e^{2-\d},
\end{eqnarray}
for any $0<\d<<1$. Also, this justifies that the solution $f^{\e}$ to the equation (\ref{small system_})
exists and is well-posed. The uniqueness and positivity follow from a standard argument as in \cite{Esposito.Guo.Kim.Marra2013}.
\end{proof}

\newpage


\bibliographystyle{siam}
\bibliography{Reference}

\end{document}